\documentclass[12pt,a4paper,twoside,reqno]{amsart}
\title[Coupled equations for K\"ahler metrics and connections]{Coupled
equations for K\"ahler metrics\\ and Yang--Mills connections}

\author[L. \'Alvarez-C\'onsul]{Luis \'Alvarez-C\'onsul}
  \address{Instituto de Ciencias Matem\'aticas (CSIC-UAM-UC3M-UCM)\\ Nicol\'as
    Cabrera 13--15, Cantoblanco\\ 28049 Madrid, Spain}
  \email{l.alvarez-consul@icmat.es}
\author[M. Garcia-Fernandez]{Mario Garcia-Fernandez}
  \address{\'Ecole Polytechnique F\'ed\'eral de Lausanne\\ EPFL SB MATHGEOM GEOM\\ MA B1 437, Station 8\\ CH-1015 Lausanne, Switzerland}
  \email{mario.garcia@epfl.ch}
\author[O. Garc\'ia-Prada]{Oscar Garc\'ia-Prada}
  \address{Instituto de Ciencias Matem\'aticas (CSIC-UAM-UC3M-UCM)\\ Nicol\'as
    Cabrera 13--15, Cantoblanco\\ 28049 Madrid, Spain}
  \email{oscar.garcia-prada@icmat.es}

\thanks{The first and the third authors are partially supported by the
  Spanish Ministerio de Econom\'ia y Competitividad (MINECO)
  under grant~MTM2010-17717. The initial work of LAC was supported by
  the Spanish ``Programa Ram\'on y Cajal''. Partial support of LAC was
  also provided by CSIC research grant~200950I027. The initial work of
  MGF was supported by an I3P grant of the Consejo Superior de
  Investigaciones Cient\'ificas. Subsequent support of MGF was
  provided by QGM (Centre for Quantum Geometry of Moduli Spaces),
  funded by the Danish National Research Foundation and by the EPFL
  (\'Ecole Polytechnique F\'ed\'eral de Lausanne). MGF wishes also to
  thank the Max Planck Institute for Mathematics in Bonn ---that he
  was visiting while part of this research was carried out--- and the
  SFB 647 project (Humboldt University, Berlin) for financial support.
}

\usepackage{hyperref,url}
\usepackage{amssymb,latexsym}
\usepackage{amsmath,amsthm}
\usepackage[latin1]{inputenc}
\usepackage{enumerate}
\usepackage[all]{xy} \CompileMatrices
\SelectTips{cm}{12} 

\usepackage{verbatim} 

\theoremstyle{plain}
\newtheorem{theorem}{Theorem}[section]
\newtheorem{lemma}[theorem]{Lemma}
\newtheorem{corollary}[theorem]{Corollary}
\newtheorem{proposition}[theorem]{Proposition}

\theoremstyle{definition}
\newtheorem{definition}[theorem]{Definition}
\newtheorem{definition-theorem}[theorem]{Definition-Theorem}
\newtheorem{example}[theorem]{Example}
\newtheorem*{acknowledgements}{Acknowledgements}
\theoremstyle{remark}
\newtheorem{remark}[theorem]{Remark}

\newcommand{\secref}[1]{\S\ref{#1}}
\newcommand{\secrefs}[1]{\S\S\ref{#1}}

\numberwithin{equation}{section} 

\newcommand{\tr}{\operatorname{tr}}
\newcommand{\pr}{p}
\newcommand{\Id}{\operatorname{Id}}

\newcommand{\End}{\operatorname{End}}
\newcommand{\Hom}{\operatorname{Hom}}
\newcommand{\ad}{\operatorname{ad}}
\newcommand{\Ad}{\operatorname{Ad}}
\newcommand{\Aut}{\operatorname{Aut}}
\newcommand{\dbar}{\bar{\partial}}
\newcommand{\imag}{\mathop{{\fam0 {\textbf{i}}}}\nolimits}
\renewcommand{\AA}{{\mathbb A}}
\newcommand{\CC}{{\mathbb C}}
\newcommand{\PP}{{\mathbb P}}

\newcommand{\RR}{{\mathbb R}}
\newcommand{\ZZ}{{\mathbb Z}}

\renewcommand{\(}{\left(}
\renewcommand{\)}{\right)}
\newcommand{\vol}{\operatorname{vol}}
\newcommand{\Vol}{\operatorname{Vol}}
\newcommand{\defeq}{\mathrel{\mathop:}=} 

\newcommand{\surj}{\to\kern-1.8ex\to}

\newcommand{\lto}{\longrightarrow}
\newcommand{\lra}[1]{\stackrel{#1}{\longrightarrow}}

\renewcommand{\implies}{\Rightarrow}

\newcommand{\cA}{\mathcal{A}}
\newcommand{\cC}{\mathcal{C}}

\newcommand{\cJ}{\mathcal{J}}
\newcommand{\cJi}{\mathcal{J}^{i}}
\newcommand{\cK}{\mathcal{K}}
\newcommand{\cM}{\mathcal{M}}
\newcommand{\cP}{\mathcal{P}}
\newcommand{\cF}{\mathcal{F}}
\newcommand{\cV}{\mathcal{V}}

\newcommand{\XX}{\mathcal{X}}
\newcommand{\cY}{\mathcal{Y}}

\newcommand{\cZ}{\mathcal{Z}}
\newcommand{\cG}{\mathcal{G}}

\newcommand{\cO}{\mathcal{O}}

\newcommand{\cR}{\mathcal{R}}

\newcommand{\Lie}{\operatorname{Lie}}

\newcommand{\LieG}{\operatorname{Lie} \cG}
\newcommand{\cX}{{\widetilde{\mathcal{G}}}}
\newcommand{\LieX}{\operatorname{Lie} \cX}

\newcommand{\cH}{\mathcal{H}} 
\newcommand{\LieH}{\Lie\cH}

\newcommand{\GL}{\operatorname{GL}}

\newcommand{\U}{\operatorname{U}}
\newcommand{\SU}{\operatorname{SU}}
\newcommand{\PU}{\operatorname{PU}}
\newcommand{\Diff}{\operatorname{Diff}}
\newcommand{\cB}{\mathcal{B}}

\renewcommand{\Re}{\operatorname{Re}}
\renewcommand{\Im}{\operatorname{Im}}
\newcommand{\LieGamma}{\operatorname{Lie} \Gamma}

\newcommand{\cbI}{{\boldsymbol{I}}}

\begin{document}

\begin{abstract}
  We study equations on a principal bundle over a compact complex
  manifold coupling a connection on the bundle with a K\"ahler
  structure on the base. These equations generalize the conditions of
  constant scalar curvature for a K\"ahler metric and
  Hermite--Yang--Mills for a connection. We provide a moment map
  interpretation of the equations and study obstructions for the
  existence of solutions, generalizing the Futaki invariant, the
  Mabuchi K-energy and geodesic stability. We finish by giving some
  examples of solutions.
\end{abstract}

\maketitle

\tableofcontents

\section*{Introduction}

In this paper we consider a system of partial differential equations
coupling a K\"ahler metric on a compact complex manifold and a
connection on a principal bundle over it. These equations, inspired by
the Hitchin--Kobayashi correspondence for bundles and the
Yau--Tian--Donaldson conjecture for constant scalar curvature
K\"{a}hler (cscK) metrics, intertwine the curvature of a
Hermitian--Yang--Mills (HYM) connection on the bundle and the scalar
curvature of a K\"ahler metric on the manifold.

To write our equations explicitly, let $X$ be a smooth compact
manifold and let $G$ be a compact real Lie group with Lie algebra
$\mathfrak{g}$. Let $E$ be a principal $G$-bundle over $X$. We fix a
positive definite inner product $(\cdot,\cdot)$ on $\mathfrak{g}$
invariant under the adjoint representation. Let $\Omega^k$ be the
space of smooth $k$-forms on $X$.  Considering the space $\Omega^k(\ad
E)$ of smooth $k$-forms on $X$ with values in the adjoint bundle $\ad
E$, the inner product of $\mathfrak{g}$ induces a pairing
\begin{equation}
\label{eq:Pairing-Intro}
\Omega^p(\ad E) \times \Omega^q(\ad E) \lto \Omega^{p+q},
\end{equation}
that we write simply as $a_p \wedge a_q$ for any $a_j \in \Omega^j(\ad
E)$, $j = p,q$. The unknown variables of the equations are a K\"ahler
structure $(g,\omega,J)$ on the base $X$ and a connection $A$ on $E$,
where $g$, $\omega$ and $J$ are respectively the metric, the
symplectic form and the complex structure. We will say that a K\"ahler
structure $(g,\omega,J)$ on $X$ and a connection $A$ on $E$ satisfy
the \emph{coupled K\"ahler--Yang--Mills equations} with coupling constants $\alpha_0,
\alpha_1 \in \RR$ if
\begin{equation}
\label{eq:CYMeq00}
\left. \begin{array}{l}
    \Lambda F_A = z\\
    \alpha_0 S_g \; + \; \alpha_1 \Lambda^2 (F_A \wedge F_A) = c
\end{array}\right \}.
\end{equation}
Here $S_g$ is the scalar curvature of $g$, $F_A$ is the curvature of
$A$, $z$ is an element of $\mathfrak{g}$ which is invariant under the
adjoint $G$-action and $c$ is a real number. The precise values of $z$
and $c$ are determined by the topology of $E$, the cohomology class of
$\omega$ and the coupling constants $\alpha_0, \alpha_1$ (see
Remark~\ref{rem:z(Omega,E)} and~\eqref{eq:constant-c}). The map
$\Lambda\colon \Omega^{p,q}(\ad E) \to \Omega^{p-1,q-1}(\ad E)$ is the
contraction operator acting on $(p,q)$-type valued forms determined by
the K\"ahler structure. In the sequel, we will refer to \eqref{eq:CYMeq00} simply as \emph{the coupled equations}.

A link with holomorphic geometry is provided by the additional
integrability condition
\begin{equation}
\label{eq:integrabilityconnection}
F_A^{0,2} = 0,
\end{equation}
between the complex structure $J$ on the base and the connection
$A$. Here $F_A^{0,2}$ denotes the $(0,2)$ part of the curvature,
regarded as an $(\ad E)$-valued smooth form on $X$. Let $G^c$ be the
complexification of the group $G$. When
\eqref{eq:integrabilityconnection} holds, the pair $(J,A)$ endows the
associated principal $G^c$-bundle $E^c = E \times_G G^c$ with a
structure of holomorphic principal bundle over the complex manifold
$(X,J)$.

The moment map interpretations of the constant scalar curvature
equation for a K\"ahler metric (cscK) and the HYM equation provide a
guiding principle, leading to \eqref{eq:CYMeq00}.  Indeed, equations
\eqref{eq:CYMeq00} have an interpretation in terms of a moment map.
This is the subject of \secref{sec:MM} and \secref{chap:Ceq}. As
observed by Fujiki~\cite{Fj} and Donaldson~\cite{D1}, the cscK
equation has a moment map interpretation in terms of a symplectic form
$\omega$ on the smooth compact manifold $X$. The group of symmetries
of the theory for cscK metrics is the group $\cH$ of Hamiltonian
symplectomorphisms. This group acts on the space $\cJ^i$ of integrable
almost complex structures on $X$ which are compatible with $\omega$,
and this action is Hamiltonian for a natural symplectic form
$\omega_\cJ$ on $\cJ^i$. The moment map interpretation of the HYM
equation was pointed out first by Atiyah and Bott~\cite{AB} for the
case of Riemann surfaces and generalized by Donaldson~\cite{D3} to
higher dimensions.  Here one considers the symplectic action of the
gauge group $\cG$ of the bundle $E$ on the space of connections $\cA$
endowed with a natural symplectic form $\omega_\cA$. Relying on these
two previous cases, the phase space for our theory is provided by the
subspace of the product
\begin{equation}
\label{eq:pairsinvacs0}
\cP \subset \cJ^i \times \cA
\end{equation}
defined by the condition \eqref{eq:integrabilityconnection}. Our
choice of symplectic structure is the restriction to $\cP$ of the
symplectic form
\begin{equation}
\label{eq:Sympfamily0}
\omega_\alpha = \alpha_0 \omega_\cJ + 4\alpha_1 \omega_\cA,
\end{equation}
for a pair of non-zero coupling constants $\alpha =
(\alpha_0,\alpha_1) \in \RR^2$.

Consider now the \emph{extended gauge group} $\widetilde{\mathcal{G}}$
defined as the group of automorphisms of the bundle $E$ covering
Hamiltonian symplectomorphisms of $X$.  This is a non trivial
extension
\begin{equation}\label{eq:Ext-Lie-groupsintro}
1 \to \cG \lto \cX \lto \cH \to 1,
\end{equation}
where $\cG$ is the group of automorphisms of $E$ covering the identity
on $X$, and $\cH$, as above, is the group of Hamiltonian
symplectomorphisms of $X$. The group $\cX$ acts on $\cP$ and
in~Proposition~\ref{prop:momentmap-pairs} we show that this action is
Hamiltonian for any value of the coupling constants, we compute a
moment map $\mu_\alpha$, and show that its zero locus corresponds to
solutions of~\eqref{eq:CYMeq00}. The coupling between the metric and
the connection occurs as a direct consequence of the structure of
$\cX$.  So, away from its singularities, the moduli space of solutions
is given by the \emph{symplectic quotient}
\begin{equation}\label{eq:moduli}
\mathcal{M}_\alpha = \mu_\alpha^{-1}(0)/\widetilde{\mathcal{G}}.
\end{equation}
Furthermore, $\omega_\alpha$ is a K\"ahler form on $\cP$ when
$\alpha_1/\alpha_0 > 0$, for a natural $\cX$-invariant complex
structure on $\cP$. Hence under this condition on the coupling
constants, the smooth locus of the moduli space of
solutions~\eqref{eq:moduli} inherits a K\"ahler structure.

We see that our problem merges the well-studied theories of
Hermitian--Yang--Mills connections (obtained for $\alpha_1/\alpha_0 >
0$) and constant scalar curvature K\"ahler metrics (which correspond
to $\alpha_1/\alpha_0 = 0$) into a unique theory.  We thus expect the
K\"ahler moduli spaces obtained in our symplectic reduction process to
have a rich geometry and topology.  In~\secref{sec:CeqscalarKal-K} we
prove that~\eqref{eq:CYMeq00} arise also as absolute minima of a
purely Riemannian functional for $G$-invariant Riemannian metrics on
the total space of $E$, providing a link to the classical
Kaluza--Klein theory.

In~\secref{chap:analytic}, which is in some sense the heart of the
paper, we undertake the study of obstructions for the existence of
solutions to~\eqref{eq:CYMeq00}, generalizing the Futaki invariant,
the Mabuchi K-energy and geodesic stability that appear in the cscK
theory~\cite{Ft0,Mab2,Ch2}. We do this geometrically, by considering
the following framework. We first fix a cohomology class $\Omega\in
H^2(X,\RR)$ and a smooth principal $G^c$-bundle $E^c\to X$. Let
$\Gamma$ be the Lie group given by those $G^c$-equivariant
diffeomorphisms of $E^c$ which cover an element in the connected
component of the identity of the diffeomorphism group of
$X$. In~\secref{sub:concrete-setup}, we associate an infinite
dimensional canonical $\Gamma$-equivariant double fibration
\begin{equation}
\label{eq:fibration-compatible-pairs}
\begin{gathered}
  \xymatrix{ & \cC \ar[dl]_-{\pi_\cB} \ar[dr]^-{\pi_\cZ} & \\  \cB & & \cZ }
\end{gathered}
\end{equation}
to the data $(X, \Omega, E^c)$. Here $\cB$ is the space of pairs
$(\omega,H)$, where $\omega \in \Omega$ is a symplectic form on $X$
and $H$ is a reduction of $E^c$ to the maximal compact subgroup
$G\subset G^c$, and $\cZ$ is a space parameterizing holomorphic
structures $I$ on $E^c$ inducing a complex structure $J$ on $X$. The
space of \emph{compatible pairs} $\cC \subset \cB \times \cZ$ is
defined as those elements of the product which induce a K\"ahler
structure on $X$. Using the results of~\secref{chap:Ceq},
in~\secref{sub:concrete-setup} we prove that the fibres of $\pi_\cB$
are (formally) K\"ahler manifolds endowed with Hamiltonian group
actions.

As a preliminary step for the study of obstructions
in~\secref{sub:ANStability}, we prove in~\secref{sub:ANsymm-spc} that
the fibres of $\pi_\cZ$ are infinite dimensional symmetric spaces
(that is, each fibre has a canonical torsion-free affine connection
$\nabla$ with covariantly constant curvature), with holonomy group
contained in the extended gauge group. Note that the fibre $B_I$ of
$\pi_\cZ$ over $I\in \cZ$ is
\begin{equation}
\label{eq:BIintro}
B_I = \cK_J \times \cR,
\end{equation}
where $\cK_J$ is the space of K\"ahler forms on $(X,J)$ in the class
$\Omega$ and $\cR$ is the space of $G$-reductions of $E^c$. When
specialized to the case of trivial $G^c$, we recover the symmetric
space structure constructed by Mabuchi~\cite{Mab1} and rediscovered by
Semmes~\cite{Se} and Donaldson~\cite{D6}. Our construction follows
closely Donaldson's in~\cite[\S 2]{D6}. A special feature of the
symmetric space structure on $B_I$ is that in general it does not
carry any canonical compatible Riemannian structure (see
Remark~\ref{rem:riemannian-symm-spc-v1}). A technical assumption in
our construction is that the $G$-invariant metric in $\mathfrak{g}$
used to define~\eqref{eq:Pairing-Intro} extends to a $G^c$-invariant
symmetric bilinear pairing
\[
\mathfrak{g}^c \otimes \mathfrak{g}^c \lto \CC,
\]
where $\mathfrak{g}^c$ is the complexification of $\mathfrak{g}$.

In~\secrefs{sub:ANStability} and \ref{sub:ANCeq} we construct an
\emph{$\alpha$-Futaki character}
\[
\cF_I \colon \Lie\Aut (E^c,I) \lto \CC,
\]
which is a complex character of the Lie algebra of the automorphism
group of the holomorphic principal bundle $(E^c,I)$ and which vanishes
when~\eqref{eq:CYMeq00} is satisfied, and an \emph{$\alpha$-K-energy}
\[
\cM_I\colon B_I\lto \RR,
\]
which is convex along geodesics on $B_I$ and bounded from below
when~\eqref{eq:CYMeq00} is satisfied, provided that the symmetric
space $B_I$ is geodesically convex. Furthermore, we motivate a
definition of \emph{geodesic stability} of the orbit $\Gamma\cdot I$
and conjecture a link with~\eqref{eq:CYMeq00} when $\Gamma_I$ is
finite. We give explicit formulae for the character $\cF_I$, the
functional $\cM_I$ and the geodesic equation on $B_I$. When
specialized to the case in which $G^c$ is trivial, we recover the
Futaki character~\cite{Ft0}, the Mabuchi K-energy~\cite{Mab2} and the
notion of geodesic stability~\cite{Ch2,D6} used in the study of the
cscK equation for K\"ahler metrics. The contents of~\secref{sub:ANCeq}
will be used in Example~\ref{ex:obstruction} to provide an explicit
situation in which there cannot exist solutions to the coupled
equations.

We would like to point out that the framework developed in
\secref{chap:analytic} is rather general and may be applied to other
situations, in particular, to equations with a further coupling with
Higgs fields.

In~\secref{chap:Deformation}, we establish sufficient conditions for
the existence of solutions to the coupled equations near a given
solution, when the coupling constants and the K\"{a}hler cohomology
class are deformed while the complex structure of the base manifold
remains fixed. Our approach is based on a generalization of techniques
developed by LeBrun \& Simanca~\cite{LS2, LS1} for the corresponding
problem in the cscK theory. We fix a complex structure on $X$ and a
structure $I$ of a holomorphic principal $G^c$-bundle on $E^c$, and
consider the space of solutions $(\omega,H)$ of~\eqref{eq:CYMeq00}
with $\omega$ in a fixed cohomology class $\Omega$ and fixed
$\alpha_0, \alpha_1$.
Then we study the behaviour of this space with respect to deformations
of the coupling constants and the K\"ahler class in a parameter space:
\begin{equation}
\label{eq:defo-par-space}
  (\alpha_0,\alpha_1,\Omega) \in \RR^2 \times H^{1,1}(X,\RR).
\end{equation}
Before doing this, in~\secref{sec:Defextremalholomorphic} we introduce
the notion of \emph{extremal pairs} $(\omega,H)$. They are analogues
in our theory of Calabi's extremal metrics in K\"ahler geometry. In
particular, an extremal pair $(\omega,H)$, with $\omega\in\Omega$, is
a solution of \eqref{eq:CYMeq00} if and only if the $\alpha$-Futaki
character associated to $I$ and $\Omega$ vanishes
(Proposition~\ref{prop:extremal-solution}). In
\secref{sec:Defmmapoperator} we study the linearization of
\eqref{eq:CYMeq00} and in \secref{sec:ANdeformingsolutions} we prove
that when $\alpha_1/\alpha_0 > 0$, any solution of the coupled
equations~\eqref{eq:CYMeq00} can be deformed into an extremal pair,
for small deformations in the parameter
space~\eqref{eq:defo-par-space}
(Theorem~\ref{thm:DeformationCYMeq1}). In \secref{sec:Defweakcoupling}
we obtain a criterion for the existence of solutions
of~\eqref{eq:CYMeq00} in the weak coupling limit $\alpha_1 \to 0$,
i.e.  for $0 < \lvert\alpha_1/\alpha_0\rvert\ll 1$
(Theorem~\ref{thm:DeformationCYMeq4}).

In \secref{chap:examples} we discuss some examples of solutions of
\eqref{eq:CYMeq00} and explain how the existence of solutions to the
limit case $\alpha_0 = 0$ can be applied, using results of Y. J. Hong
in \cite{Ho2}, to obtain cscK metrics on ruled manifolds. As for the
examples, in~\secref{sec:example1} we deal with the case of vector
bundles over Riemann surfaces and projectively flat bundles over
K\"ahler manifolds satisfying a topological constraint. In both
situations, the coupled system \eqref{eq:CYMeq00} reduces to the limit
case $\alpha_1 = 0$ (cscK equation and HYM equation). When $\dim_\CC X
> 1$, we use Theorem~\ref{thm:DeformationCYMeq1} to deform the
K\"ahler class and provide non-trivial examples of
solutions. In~\secref{sec:example2} we consider homogeneous Hermitian
holomorphic vector bundles over homogeneous K\"ahler manifolds. In
\secref{sec:example3} we discuss some (well known) examples of stable
bundles over K\"ahler--Einstein manifolds where
Theorem~\ref{thm:DeformationCYMeq4}
applies. Section~\ref{sec:example3} provides examples of solutions in
which the K\"ahler metric is not cscK and also examples where the
invariant $\cF_I$ obstructs the existence of solutions for small ratio
of the coupling constants.

Coupled equations for metrics and connections have of course been
studied for a long time in the context of unified field theories in
physics and more recently in string theory (see
e.g. \cite{LiYau}).  They have also been considered in the
context of Riemannian geometry, like the Eintein--Maxwell equations on
4-manifolds studied in~\cite{Lb}.  Our motivation, however, for this
work has been to find a K\"ahler analogue of these situations.
Another important motivation for us comes from the relation with
algebraic geometry, in particular with the moduli problem for pairs
consisting of a polarised manifold and a holomorphic bundle over
it. Despite its intrinsic mathematical interest and its relevance in
theoretical physics, the latter problem has been little explored,
probably due to the hard technical difficulties which arise in the
algebro-geometric approach as soon as the complex dimension of the
base is greater than $1$ (see~\cite{GiesMorr,Capo,Pd} for the case of
curves, and \cite{Hu,ST} for higher dimensions).  Throughout
this paper we hope to show that the study of our coupled equations
provides a reasonable differential-geometric approach to the moduli
problem for bundles and varieties, giving compelling evidence of the
existence of a Hitchin--Kobayashi correspondence for the coupled
equations as has been conjectured in~\cite{GF1}.

Since this paper was finished, there have been several developments in
the theory of the coupled K\"ahler-Yang-Mills equations. Keller and
T{\o}nnesen-Friedman \cite{KellerTonnesen} have found solutions on
line bundles over complex threefolds (with positive ratio $\alpha >
0$) that do not admit any cscK metric in the class of the
polarization. The second author jointly with C. Tipler have recently
found new examples of solutions \cite{GFT}, by deformation of the
holomorphic structure on a homogeneous bundle over $\PP^1 \times
\PP^1$. More remarkably, an interesting relation between the coupled
equations and physical equations, describing gravitating vortices over
a Riemann surface, has been recently found by the authors
\cite{AGGvortices}. These vortices represent the coupling of gravity
and a condensed matter system and are known in the physics literature
as cosmic strings (or topological defects) in the Abelian Higgs
model. Based on classical results by Y. Yang \cite{Yang}, this
relation provides a plethora of solutions of the coupled equations in
$\PP^1 \times \PP^1$ and a explicit (conjectural) description of the
moduli space \eqref{eq:moduli} and the stability condition for this
particular case.

\begin{acknowledgements}
We want to thank Olivier Biquard, David Calderbank, Simon Donaldson,
Nigel Hitchin, Julien Keller, Alastair King, Ignasi Mundet i Riera, Vicente
Mu\~noz, Julius Ross, Ignacio Sols, Jacopo Stoppa and Richard Thomas
for helpful discussions and suggestions.  We also wish to thank the
Max Planck Institute for Mathematics (Bonn), and the Isaac Newton
Institute for Mathematical Sciences for their hospitality and
support. MGF thanks Instituto de Ciencias Matem\'aticas (Madrid),
Imperial College (London), University of Paris 6 and Humboldt
University (Berlin) for their hospitality.
\end{acknowledgements}

\section{Hamiltonian action of the extended gauge group}
\label{sec:MM}

In this section we define the \emph{extended gauge group} $\cX$ of a
bundle over a compact symplectic manifold, an extension of the
infinite dimensional Lie groups involved in the moment map problems
for the HYM and the cscK equation. We show that the action of $\cX$ on
the space of connections of the bundle is Hamiltonian and compute an
equivariant moment map. Symplectic reductions by Lie group extensions
have been studied in the literature in various degrees of generality
(see~\cite{MMOPR} and references therein). Previous work includes
split group extensions and more general ones, although it seems that
the moment map calculations of~\secref{sec:Ceqcoupling-term}, based on
Proposition~\ref{prop:ham-act-ext-grp}, have not been previously made
(cf.~\cite[\S 3.2]{MMOPR}).

\subsection{The Hermitian--Yang--Mills equation}
\label{section:YMmmap}

First we set out some notation in order to review the moment map
interpretations of the HYM equation. Let $X$ be a compact symplectic
manifold of dimension $2n$, with symplectic form $\omega$, $G$ a real
compact Lie group with Lie algebra $\mathfrak{g}$, and $E$ a smooth
principal $G$-bundle over $X$, with the $G$-action on the right. In
the sequel $\omega^{[k]}$ will denote $\frac{\omega^k}{k!}$. The
spaces of smooth $k$-forms on $X$ and smooth $k$-forms with values in
any given vector bundle $F$ on $X$ are denoted by $\Omega^{k}$ and
$\Omega^{k}(F)$, respectively. Fix a positive definite inner product
on $\mathfrak{g}$, invariant under the adjoint action, denoted
\[
  (\cdot,\cdot): \mathfrak{g} \otimes \mathfrak{g} \lto \RR.
\]
This product induces a metric on the adjoint bundle $\ad E=E\times_G
\mathfrak{g}$, which extends to a bilinear map on $(\ad E)$-valued
differential forms (we use the same notation as in~\cite[\S 3]{AB})
\begin{equation}
\label{eq:Pairing}
\begin{gathered}
\xymatrix @R=0ex @C=-7ex{ **[l]
\Omega^p(\ad E) \times \Omega^q(\ad E)\ar[r] 
& **[r]
\Omega^{p+q}
\\ **[l]
(a_p,a_q)\ar@{|->}[r] & **[r] a_p\wedge a_q.
}\end{gathered}
\end{equation}
We consider the operator
\begin{equation}
\label{eq:Lambda}
\begin{gathered}
\xymatrix @R=0ex @C=-3ex{
**[l]\Lambda=\Lambda_\omega\colon\;\Omega^k\ar[r] & **[r]\Omega^{k-2}\\
**[l]\psi\ar@{|->}[r] & **[r] \omega^{\sharp}\lrcorner\psi,
}\end{gathered}
\end{equation}
where $\sharp$ is the operator acting on $k$-forms induced by the
symplectic duality $\sharp\colon T^*X \to TX$ and $\lrcorner$ denotes
the contraction operator. Its linear extension to $\Omega^k(\ad E)$ is
also denoted $\Lambda:\Omega^k(\ad E)\to \Omega^{k-2}(\ad E)$ (we use
the same notation as, e.g., in~\cite{D3}).

Let $\cA$ be the set of connections on $E$. This is an affine space
modelled on $\Omega^1(\ad E)$, with a left action of the gauge group
$\cG$ of $E$, i.e. the group of $G$-equivariant diffeomorphisms of $E$
covering the identity map on $X$. The 2-form on $\cA$ defined by
\begin{equation}
\label{eq:SymfC}
\omega_{\cA}(a,b) = \int_X a \wedge b \wedge \omega^{[n-1]}
\end{equation}
for $a,b \in T_A \cA = \Omega^1(\ad E)$, $A\in\cA$, is a
$\cG$-invariant symplectic form. As observed by Atiyah and
Bott~\cite{AB} when $X$ is a Riemann surface and by
Donaldson~\cite{D3,D5} in higher dimensions, the $\cG$-action on $\cA$
is Hamiltonian, with equivariant moment map $\mu_\cG\colon \cA\to
(\LieG)^*$ given by
\begin{equation}
\label{eq:momentmap-cG}
  \langle\mu_\cG(A),\zeta\rangle = \int_X \zeta \wedge (\Lambda F_A - z) \omega^{[n]},
\end{equation}
for $A\in\cA$, $\zeta \in \LieG=\Omega^0(\ad E)$, where
$F_A\in\Omega^2(\ad E)$ is the curvature of $A\in\cA$ and $z$ is an
element of the space
\begin{equation}
\label{eq:centre-z}
  \mathfrak{z}=\mathfrak{g}^G
\end{equation}
of elements of $\mathfrak{g}$ which are invariant under the adjoint
$G$-action, that we identify with sections of $\ad E$. Recall that the
moment map satisfies
\[
  d\langle \mu_\cG,\zeta\rangle=Y_\zeta\lrcorner\omega_\cA
\]
for all $\zeta\in\LieG$, where $Y_\zeta$ is the vector field on $\cA$
generated by the infinitesimal action of $\zeta$, and equivariance
means that, for all $g\in\cG$ and $A\in \cA$,
\[
\mu_\cG(g\cdot A) = \Ad(g^{-1})^* \mu_\cG(A).
\]
Suppose now that $X$ is a K\"ahler manifold, with K\"ahler form
$\omega$ and complex structure $J$. Consider the complexification
$G^c$ of $G$ and the associated principal $G^c$-bundle $E^c = E
\times_G G^c$, where $G$ acts on $G^c$ by left multiplication. There
is a distinguished $\cG$-invariant subspace
\begin{equation}
\label{eq:cA^{1,1}_J}
\cA_J^{1,1} \subset \cA
\end{equation}
consisting of connections $A$ with $F_A\in\Omega^{1,1}_J(\ad E)$, or
equivalently satisfying $F_A^{0,2}=0$, where $\Omega_J^{p,q}(\ad E)$
denotes the space of $(\ad E)$-valued smooth $(p,q)$-forms with
respect to $J$ and $F_A^{0,2}$ is the projection of $F_A$ into
$\Omega_J^{0,2}(\ad E)$. This space is in bijection with the space of
holomorphic structures on the principal $G^c$-bundle $E^c$ over the
complex manifold $(X,J)$ (see~\cite{Si}).

\begin{definition}
 A connection $A \in \cA^{1,1}_J$ is called
 \emph{Hermitian--Yang--Mills} if it satisfies the
 \emph{Hermitian--Yang--Mills equation}
\begin{equation}
\label{eq:HYM}
\Lambda F_A=z.
\end{equation}
\end{definition}

\begin{remark}
\label{rem:z(Omega,E)}
The element $z \in \mathfrak{z}$ in the right-hand side
of~\eqref{eq:HYM} is determined by the cohomology class $\Omega\defeq
[\omega]\in H^2(X)$ and the topology of the principal bundle $E$. This
follows after applying $(z_j,\cdot)$ to~\eqref{eq:HYM}, for an
orthonormal basis $\{z_j\}$ of $\mathfrak{z}\subset\mathfrak{g}$, and
then integrating over $X$, we obtain
\begin{equation}
\label{eq:z(Omega,E)}
z=\sum_j \frac{\langle z_j(E)\cup \Omega^{[n-1]},[X]\rangle}{\Vol_\Omega} z_j.
\end{equation}
Here, $\Omega^{[k]}\defeq \Omega^k/k!$, $\Vol_\Omega\defeq
\int_X\omega^{[n]} = \langle \Omega^{[n]},[X]\rangle$ and
$z_j(E)\defeq [z_j \wedge F_A]\in H^2(X)$ is the Chern--Weil class
associated to the $G$-invariant linear form $(z_j, \cdot)$ on
$\mathfrak{g}$, which only depends on the topology of the bundle $E$
(see~\cite[Ch XII, \S 1]{KNII}).
\end{remark}

The moduli space of Hermitian--Yang--Mills connections is defined as the set
of classes of gauge equivalent solutions to~\eqref{eq:HYM}. This coincides
with the quotient
\begin{equation}
\label{eq:SympredHYM}
  \mu_\cG^{-1}(0)/\cG,
\end{equation}
where $\mu_\cG$ is now the restriction of the moment map to $\cA^{1,1}_J$.
Away from its singularities, $\cA^{1,1}_J$ inherits a complex structure
compatible with $\omega_\cA$ and hence a K\"ahler structure. Thus the smooth
locus of $\cA^{1,1}_J$ is a K\"ahler manifold endowed with a Hamiltonian
$\cG$-action and hence, away from singularities, the moduli space of
Hermitian--Yang--Mills connections can be constructed as a K\"ahler reduction,
which, if non-empty, is a finite-dimensional K\"ahler manifold.

\subsection{Hamiltonian actions of extended Lie groups}
\label{sec:Ceqham-act-ext-grp}

Consider a general extension of Lie groups
\begin{equation}
\label{eq:Ext-Lie-groups}
 1 \to \cG \lra{\iota} \cX \lra{\pr} \cH \to 1.
\end{equation}
We will describe now, under certain assumptions, the Hamiltonian action of
$\cX$ on a symplectic manifold, in terms of $\cG$ and $\cH$. In the next
section we will apply this general set up to the case in which the symplectic
manifold is the space of connections of a bundle and $\cX$ is
the extended gauge group mentioned in the introduction --- this may explain the notation.

The extension~\eqref{eq:Ext-Lie-groups} determines an extension of Lie
algebras
\begin{equation}
\label{eq:Ext-Lie-alg}
 0 \to \LieG \lra{\iota} \LieX \lra{\pr} \LieH \to 0,
\end{equation}
where the use of the same symbols $\iota$ and $\pr$ should lead to no
confusion. Note that the short exact sequence~\eqref{eq:Ext-Lie-alg} does not
generally split as a sequence of Lie algebras, but it always does as a short
exact sequence of vector spaces. Let $W\subset \Hom(\LieX,\LieG)$ be the
affine space of vector space splittings.  Since $\cG\subset\cX$ is a normal
subgroup, there is a well-defined $\cX$-action on $W$, given by
\[
g\cdot \theta \defeq \Ad(g)\circ \theta \circ \Ad(g^{-1}), \text{ for
  $g\in\cX$, $\theta\in W$.}
\]
Let $\cA$ be a manifold with an action of the `extended' Lie group
$\cX$. Suppose that there exists a $\cX$-equivariant smooth map $\theta\colon \cA\to
W$. Let $\omega_\cA$ be a symplectic form on $\cA$ preserved by the
$\cX$-action. Using $\theta$, we will characterise the existence of a
$\cX$-equivariant moment map for this action in terms of $\cG$ and $\cH$. The
case considered in this paper (see~\secref{sec:Ceqcoupling-term}) is an
example where such a $\theta$ exists. Observe that if $\cA$ is a point, then
$\theta$ determines an isomorphism $\LieX \cong \LieG\rtimes\LieH$, which
shows that in this case the existence of $\theta$ is a very strong condition.

Suppose that the $\cX$-action is Hamiltonian, with $\cX$-equivariant
moment map $\mu_{\cX}\colon \cA\to (\LieX)^*$. We can use $\theta$ to
decompose this map into two pieces corresponding to $\LieG$ and
$\LieH$.  Consider $\theta^\perp$ uniquely defined by
$\label{eq:theta-perp} \Id - \iota\circ \theta = \theta^\perp\circ p$,
where $\iota$ and $p$ given in~\eqref{eq:Ext-Lie-groups}. Then the map
\[
\xymatrix @R=0ex @C=-4ex{
**[l]W\ar[r] & **[r]\Hom(\LieH,\LieX)\\
**[l]\theta\ar@{|->}[r] & **[r] \theta^\perp
}
\]
is $\cX$-equivariant, where the $\cX$-action on $\Hom(\LieH,\LieX)$
given by
\[
g\cdot\theta^\perp= \Ad(g)\circ \theta^\perp \circ \Ad(\pr (g^{-1}))
\]
for $g\in\cX$. Moreover, the map
\[
\theta^\perp \colon \cA \lto \Hom(\LieH,\LieX)
\]
is $\cX$-equivariant and we can decompose  the moment map as
\begin{equation}
\label{eq:piecescXmmap}
\langle \mu_{\cX},\zeta\rangle = \langle \mu_{\cX}, \iota \theta\zeta\rangle + \langle \mu_{\cX},
\theta^{\perp}\pr(\zeta) \rangle,
\end{equation}
for all $\zeta\in\LieX$, where the summands in the right hand side define a
pair of $\cX$-equivariant maps $\mu_{\cG}\colon \cA\to (\LieG)^*$,
$\sigma_\theta\colon \cA\to (\LieH)^*$, given by
\begin{align*}
  \langle \mu_{\cG}, \zeta \rangle & \defeq \langle \mu_{\cX},
  \iota\zeta\rangle, \text{ for all $\zeta\in\LieG$,} \notag
  \\
\langle \sigma_\theta, \eta \rangle & \defeq \langle\mu_{\cX},
\theta^\perp \eta \rangle, \text{ for all $\eta\in\LieH$.}
\end{align*}
Note that since $\cG$ is a normal subgroup of $\cX$, we can require the
map $\mu_{\cG}$ to be $\cX$-equivariant. It is now straightforward
from the moment map condition for $\mu_{\cX}$ to check that
$\mu_{\cG}$ is a 
moment map for the $\cG$-action on $\cA$, i.e. $d\langle
\mu_{\cG},\zeta\rangle = Y_{\zeta} \lrcorner \omega_{\cA}$ for all
$\zeta \in \LieG$. In order to see that $\sigma_\theta$ satisfies a
similar infinitesimal condition, giving our characterization of
Hamiltonian $\cX$-action, we first introduce some notation. Given a
smooth map $\zeta\colon \cA\to\LieX$, $Y_\zeta$ denotes the vector
field on $\cA$ given by
\begin{equation}
\label{eq:infact-A}
  Y_{\zeta|A} \defeq \frac{d}{dt}_{|t = 0}\exp(t\zeta_A)\cdot A,
\end{equation}
for all $A\in\cA$.  In particular, $\theta\colon \cA
\to W$ induces a map
\[
\xymatrix @R=0ex @C=-3ex{
**[l]Y_{\theta^\perp}\colon \LieH \ar[r] & **[r]\Omega^0(T\cA)\\
**[l]\eta\ar@{|->}[r] & **[r] Y_{\theta^\perp\eta}.
}
\]
Note also that, by definition, $d\theta$ is a $\cX$-invariant
$\Hom(\LieH,\LieG)$-valued 1-form on $\cA$.

\begin{proposition}
\label{prop:ham-act-ext-grp}
The $\cX$-action on $\cA$ is Hamiltonian if and only if the action of
$\cG\subset\cX$ on $\cA$ is Hamiltonian, with a $\cX$-equivariant
moment map $\mu_{\cG}\colon \cA\to (\LieG)^*$, and there exists a
smooth $\cX$-equivariant map $\sigma_\theta\colon \cA\to (\LieH)^*$
satisfying
\begin{equation}
\label{eq:muX}
Y_{\theta^\perp\eta} \lrcorner \omega_\cA = \langle \mu_\cG,\langle d\theta,\eta\rangle\rangle + d\langle \sigma_\theta, \eta\rangle,
\end{equation}
for all $\eta\in\LieH$.
In this case, a $\cX$-equivariant moment map $\mu_{\cX}\colon \cA\to
(\LieX)^*$ is given by
\begin{equation}
\label{eq:def-muX}
\langle \mu_{\cX},\zeta\rangle = \langle \mu_\cG, \theta\zeta\rangle +
\langle \sigma_\theta, \pr (\zeta) \rangle, \text{ for all $\zeta\in\LieX$.}
\end{equation}
\end{proposition}

\begin{proof}
  To prove the ``only if'' part it remains to
  check~\eqref{eq:muX}. This follows by definition, differentiating in
  \eqref{eq:piecescXmmap} and using that
  \begin{gather*}
    d \langle \mu_\cG,\theta \zeta\rangle = \langle d\mu_\cG,\theta
    \zeta\rangle + \langle \mu_\cG,\langle d\theta,\eta\rangle\rangle
    \, \text{ and }\\
    Y_\zeta \lrcorner \omega = Y_{\theta\zeta} \lrcorner \omega +
    Y_{\theta^\perp\eta} \lrcorner \omega, \text{ with $\eta \defeq
      \pr(\zeta)$,}
  \end{gather*}
  where the first equation is obtained applying the chain rule, and
  the second one holds because $\zeta = \theta\zeta +
  \theta^\perp\eta$ and $Y_\zeta$ is linear in $\zeta$. The ``if''
  part is straightforward from the statement and is left to the
  reader.
\end{proof}

Note that condition~\eqref{eq:muX} for $\sigma_\theta$ generalizes the usual
infinitesimal condition $Y_\eta \lrcorner \omega_\cA =
d\langle\mu_\cH,\eta\rangle$ ($\eta\in\cH$) for moment maps $\mu_\cH$ for the
induced $\cH$-action on $\cA$ when the Lie group
extension~\eqref{eq:Ext-Lie-groups} splits.

\subsection{The extended gauge group action on the space of
  connections}
\label{sec:Ceqcoupling-term}

We apply now the general theory developed in \secref{sec:Ceqham-act-ext-grp}
to compute the moment map for the action of the \emph{extended gauge group} of
a bundle over a compact symplectic manifold, on the space of connections.

Let $X$ be a compact symplectic manifold of dimension $2n$, with
symplectic form $\omega$. Let $G$ be a Lie group and $E$ be a smooth
principal $G$-bundle on $X$, with projection map $\pi\colon E\to
X$. Let $\cH$ be the group of Hamiltonian symplectomorphisms of
$(X,\omega)$ and $\Aut E$ be the group of automorphisms of the bundle
$E$. Recall that an \emph{automorphism} of $E$ is a $G$-equivariant
diffeomorphism $g\colon E\to E$. Any such automorphism covers a unique
diffeomorphism $\check{g}\colon X\to X$, i.e. a unique $\check{g}$
such that $\pi\circ g=\check{g}\circ \pi$. We define the
\emph{Hamiltonian extended gauge group} (to which we will simply refer
as extended gauge group) of $E$,
\[
\cX \subset \Aut E,
\]
as the group of automorphisms which cover elements of $\cH$. Then the
gauge group of $E$, already defined in~\secref{section:YMmmap}, is the
normal subgroup $\cG\subset\cX$ of automorphisms covering the
identity.

The map $\cX \lra{\pr} \cH$ assigning to each automorphism $g$ the
Hamiltonian symplectomorphism $\check{g}$ that it covers is
surjective. To show this, let $h \in \cH$. By definition there exists
a Hamiltonian isotopy $[0,1] \times X \to X \colon (t,x) \mapsto
h_t(x)$ from $h_0=\Id$ to $h_1=h$, which is the flow of a smooth
family of vector fields $\eta_t \in \LieH$, i.e. with $dh_t/dt =
\eta_t \circ h_t$ (see e.g.~\cite[\S 3.2]{McS}).  Choose a connection
$A$ on $E$.  Let $\zeta_t\in\LieX$ be the horizontal lift to $E$ of
$\eta_t$ given by $A$. The vector fields $\zeta_t$ are $G$-invariant
so their time-dependent flow $g_t$ exists for all $t\in [0,1]$ and the
$g_t\colon E\to E$ are $G$-equivariant. Since $\zeta_t$ is a lift of
$\eta_t$ to $E$, its flow $g_t$ covers $h_t$ (i.e. $h_t=\check{g}_t$),
so in particular $g_t\in\cX$ for all $t$ and $g_1\in\cX$ covers
$h=h_1$. Thus $\pr$ is surjective.  We thus have an exact sequence of
Lie groups
\begin{equation}
\label{eq:coupling-term-moment-map-1}
  1\to \cG \lra{\iota} \cX \lra{\pr} \cH \to 1,
\end{equation}
where $\iota$ is the inclusion map.

\begin{remark}
\label{rem:exact-ses}
Note that the sequence~\eqref{eq:coupling-term-moment-map-1} is exact even
when the structure group $G$ and the base manifold $X$ are non-compact. The
crucial fact is that $\cH$ lies in the identity component of the
diffeomorphism group $\Diff X$ of $X$ (see~\cite{ACMM} for further details).
\end{remark}

There is an action of $\Aut E$, and hence of the extended gauge group, on the
space $\cA$ of connections on $E$. To define this action, we view the elements
of $\cA$ as $G$-equivariant splittings $A\colon TE\to VE$ of the short exact
sequence
\begin{equation}
\label{eq:principal-bundle-ses}
  0 \to VE\lto TE\lto \pi^*TX \to 0,
\end{equation}
where $VE=\ker d\pi$ is the vertical bundle. Using the action of $g\in \Aut E$
on $TE$, its action on $\cA$ is given by $g \cdot A \defeq g\circ A \circ
g^{-1}$. Any such splitting $A$ induces a vector space splitting of the Atiyah
short exact sequence
\begin{equation}
\label{eq:Ext-Lie-alg-3}
0\to \LieG \lra{\iota} \Lie(\Aut E) \lra{\pr} \Lie(\Diff X) \to
0
\end{equation}
(cf.~\cite[equation~(3.4)]{AB}), where $\Lie(\Diff X)$ is the Lie algebra of
vector fields on $X$ and $\Lie(\Aut E)$ is the Lie algebra of $G$-invariant
vector fields on $E$. This splitting is given by maps
\begin{equation}
\label{eq:theta-thetaperp}
\theta_A\colon \Lie(\Aut E)\lto \LieG, \quad \theta_A^\perp\colon
\Lie(\Diff X) \lto \Lie(\Aut E)
\end{equation}
such that $\iota\circ \theta_A + \theta_A^\perp\circ \pr =\Id$, where $\theta_A$
is the vertical projection given by $A$ and $\theta_A^\perp$ the horizontal
lift of vector fields on $X$ to vector fields on $E$ given by $A$.

\begin{lemma}
\label{lem:infinit-action-connections}
Let $A\in\cA$, $\zeta\in\Lie(\Aut E)$ and $\check{\zeta}\defeq
\pr(\zeta)\in\Lie(\Diff X)$. Then the infinitesimal action $Y_{\zeta|A}\in
T_A\cA=\Omega^1(\ad E)$ of $\zeta$ on $A$ is given by
\begin{equation}
\label{eq:infinit-action-connections}
  Y_{\zeta|A}=-d_A(\theta_A\zeta)-\check{\zeta}\lrcorner F_A,
\end{equation}
where $d_A\colon \Omega^k(\ad E)\to \Omega^{k+1}(\ad E)$ is the
covariant derivative associated to $A$.
\end{lemma}

\begin{proof}
By the Leibninz rule, for all
$v\in\Omega^0(TE)$,
\[
  \frac{d}{dt}_{|t = 0}\(e^{t\zeta}\circ A\circ e^{-t\zeta}(v)\)
  = \theta_A[\zeta,v] - [\zeta,\theta_A v]
  = \theta_A[\zeta,v - \theta_A v],
\]
where in the second equality we have used the fact that $\zeta$ covers a
vector field $\check{\zeta}$ on $X$, so that the vector field
$[\zeta,\theta_A v]$ is vertical. It is easy to see that this expression is
tensorial in $v$, so at each point of $E$ it only depends on its projection
$\pi_*v$. Hence the vector $Y_{\zeta|A}\in T_A\cA$, regarded as an element of
$\Omega^1(\ad E)$, is
given by
\begin{equation*}
\begin{split}
  Y_{\zeta|A}(y) & = \theta_A[\zeta,\theta_A^\perp y]
  = [\theta_A\zeta,\theta_A^\perp y] + \theta_A[\theta_A^\perp\check{\zeta},\theta_A^\perp y]
  \\ &
  = (-d_A(\theta_A\zeta) - \check{\zeta} \lrcorner F_A)(y),
\end{split}
\end{equation*}
for any $y \in \Omega^0(TX)$, where we have used the formulae
\begin{equation}
\label{eq:cov-derivative-commutators}
y \lrcorner d_A\zeta =
[\theta_A^\perp y,\zeta], \quad F_A(y,y')= -\theta_A[\theta_A^\perp
y,\theta_A^\perp y']
\end{equation}
(see the equation before (4.2) and the equation after (3.4) in \cite{AB} and
note that we are using a different sign convention for the curvature).
\end{proof}

The splitting~\eqref{eq:theta-thetaperp} restricts to a splitting of the exact
sequence
\begin{equation}
\label{eq:Ext-Lie-alg-2}
 0 \to \LieG \lra{\iota} \LieX \lra{\pr} \LieH \to 0
\end{equation}
induced by~\eqref{eq:coupling-term-moment-map-1}.  Following the
notation of~\secref{sec:Ceqham-act-ext-grp}, it is easy to see that
the map
\begin{equation}
\label{eq:theta}
\begin{gathered}
\xymatrix @R=0ex @C=2.5ex{
**[l]\theta\colon \cA \ar[r] & **[r]W\\
**[l]A\ar@{|->}[r] & **[r] \theta_A
}
\end{gathered}
\end{equation}
is a $\cX$-equivariant smooth map. It is also clear that the
$\cX$-action on $\cA$ is symplectic, for the symplectic
form~\eqref{eq:SymfC}. The methods of~\secref{sec:Ceqham-act-ext-grp}
apply here to provide a moment map. To see this, we use the
isomorphism of Lie algebras
\begin{equation}
\label{eq:LieH}
\LieH\cong C^{\infty}_0(X),
\end{equation}
where $\LieH$ is the Lie algebra of Hamiltonian vector fields on $X$
and $C^{\infty}_0(X)$ is the Lie algebra of smooth real functions on
$X$ with zero integral over $X$ with respect to $\omega^{[n]}$, with
the Poisson bracket. This isomorphism is induced by the map
$C^{\infty}(X)\to \LieH \colon \phi\mapsto \eta_\phi$, which to each
function $\phi$ assigns its Hamiltonian vector field $\eta_\phi$,
defined by
\begin{equation}
\label{eq:eta_phi}
  d\phi= \eta_\phi \lrcorner \omega.
\end{equation}

\begin{proposition}
\label{prop:momentmap-X}
The $\cX$-action on $\cA$ is Hamiltonian, with equivariant moment map
$\mu_\cX\colon \cA\to (\LieX)^*$ given by
\begin{equation}
\label{eq:thm-muX}
  \langle \mu_\cX,\zeta\rangle = \langle \mu_\cG, \theta\zeta\rangle +
  \langle \sigma, \pr (\zeta) \rangle, \text{ for all $\zeta\in\LieX$,}
\end{equation}
where $\mu_\cG: \cA \to (\LieG)^*$ and $\sigma\colon \cA\to (\LieH)^*$
are given by
\begin{equation}
\label{eq:sigma}
\begin{split} &
\langle \mu_\cG,\theta\zeta\rangle(A)
= \int_X \theta_A\zeta \wedge (\Lambda F_A - z) \omega^{[n]},
\\ &
\langle\sigma,\eta_\phi\rangle(A)
= -\frac{1}{4}\int_X \phi\(\Lambda^2 (F_A\wedge F_A) - 4 \Lambda F_A \wedge z\)\omega^{[n]},
\end{split}
\end{equation}
for all $A\in\cA$, $\phi\in C_0^\infty(X)$.
\end{proposition}

\begin{proof}
The result follows, by Proposition~\ref{prop:ham-act-ext-grp}, from the facts
that $\mu_\cG$ and $\sigma$ are $\cX$-equivariant, which is immediate
from~\eqref{eq:sigma} by the change of variable theorem, and the map $\sigma$
defined by~\eqref{eq:sigma} satisfies~\eqref{eq:muX}.  To show this, let
$\zeta\in \Lie(\Aut E)$, $A\in\cA$ and note
that~\eqref{eq:infinit-action-connections} also applies to maps $\zeta\colon
\cA\to \Lie(\Aut E)$ (with $Y_{\zeta|A}$ defined by~\eqref{eq:infact-A}). In
particular,
\begin{equation*}
   Y_{\theta_A^\perp\eta}(A)   = - \eta \lrcorner F_A, \text{ for $\eta\in\LieH$.}
\end{equation*}
The $\Hom(\LieH,\LieG)$-valued 1-form $d\theta$ on $\cA$ is given by
\[
\begin{gathered}
\xymatrix @R=0ex @C=-8.5ex{
**[l]d\theta(a)\colon\LieH\ar[r]&**[r]\LieG\\
**[l]\eta\ar@{|->}[r]&**[r]\langle d\theta(a),\eta\rangle=a(\eta),
}
\end{gathered}
\]
for $A\in\cA$ and $a\in T_A\cA=\Omega^1(\ad E)$. Observe that the quantity
$$
\mu'(A) = \int_X (\theta_A \zeta \wedge z - \phi \Lambda F_A \wedge z) \omega^{[n]}
$$
is locally constant on $\cA$, so it is enough to assume $z = 0$. To see this, we use the path $A_t=A+ta$ to calculate
\begin{align*}
\frac{d}{dt}_{|t=0} \mu'(A + ta) & = \int_X a(\eta_\phi) \wedge z \omega^{[n]} - \phi d_A a \wedge z \omega^{[n-1]}\\
& = \int_X a \wedge z d\phi \wedge \omega^{[n-1]} - \phi d_A a \wedge z \omega^{[n-1]} = 0.
\end{align*}
Here we have used the identity $dF_{A_t}/dt = d_A a$ for $t=0$ and integration by parts, combined with the equality
\[
a\wedge z d\phi \wedge \omega^{[n-1]} = a(\eta_\phi) \wedge z \omega^{[n]}.
\]
Assuming $z = 0$, for the last term of the right hand side of~\eqref{eq:muX}, we have
\[
\Lambda^2(F_{A}\wedge F_{A}) \omega^{[n]} = 2F_A\wedge F_A\wedge\omega^{[n-2]}.
\]
Using now the Bianchi identity $d_A F_A=0$, a similar calculation as before shows that
\begin{equation}
\label{eq:momentmap-X-3}
\begin{split}
  d\langle \sigma,\eta\rangle (a) & = -\frac{1}{2} \frac{d}{dt}_{|t = 0}
  \int_X \phi\, (F_{A_t}\wedge F_{A_t})\wedge\omega^{[n-2]}\\
  & = - \int_X \phi \, d_Aa \wedge F_A \wedge \omega^{[n-2]}\\
  & = \int_X (\eta \lrcorner \omega) \wedge a \wedge F_A \wedge \omega^{[n-2]}.\\
\end{split}
\end{equation}
To compute the integral in the last equality, note
that 
$(a \wedge F_A) \wedge \omega^{n-1}=0$, so
contracting with $\eta$ we obtain
\[
a\wedge F_A \wedge(\eta \lrcorner \omega)\wedge  \omega^{[n-2]} = a(\eta)\wedge\Lambda F_A \omega^{[n]} - a\wedge (\eta \lrcorner F_A) \wedge \omega^{[n-1]},
\]
using the identity $F_A \wedge \omega^{[n-1]} = \Lambda F_A \omega^{[n]}$. Combined with~\eqref{eq:momentmap-X-3}, we thus obtain~\eqref{eq:muX}:
\begin{align*}
  d\langle \sigma, \eta \rangle (a) & =
\int_X a\wedge (\eta \lrcorner F_A)\wedge \omega^{n-1} - \int_X a(\eta) \wedge (\Lambda F_A - z) \omega^{[n]}\\
& = (Y_{\theta_A^\perp\eta} \lrcorner \omega_\cA)(a) -
\langle\mu_\cG,\langle d\theta(a),\eta\rangle\rangle.
\qedhere
\end{align*}
\end{proof}

\section{The coupled equations}
\label{chap:Ceq}

In this section we give a moment map interpretation of the coupled
equations~\eqref{eq:CYMeq00} for the action of the extended gauge
group, introduced in \secref{sec:MM}. We also define a purely
Riemannian functional, the Calabi--Yang--Mills functional, whose
absolute minimum over the phase space are precisely the solutions of
the coupled equations, that we interpret in terms of the Kaluza--Klein
theory for $G$-invariant metrics on the total space of the
bundle. With this purpose we first recall the moment map
interpretation of the cscK equation given by Fujiki and Donaldson.

\subsection{The Hermitian scalar curvature}
\label{section:cscKmmap}

The moment map interpretation of the scalar curvature was first given
by Fujiki~\cite{Fj} for the Riemannian scalar curvature of K\"ahler
manifolds and generalized independently by Donaldson~\cite{D1} for the
Hermitian scalar curvature of almost K\"ahler manifolds. Here we
follow closely Donaldson's approach.

First we recall the notion of Hermitian scalar curvature of an almost
K\"ahler manifold. Fix a compact symplectic manifold $X$ of dimension
$2n$, with symplectic form $\omega$. An almost complex structure $J$
on $X$ is called compatible with $\omega$ if the bilinear form
$g_J(\cdot,\cdot) \defeq \omega(\cdot,J\cdot)$ is a Riemannian metric
on $X$. Any almost complex structure $J$ on $X$ which is compatible
with $\omega$ defines a Hermitian metric on $T^*X$ and there is a
unique unitary connection on $T^*X$ whose (0,1) component is the
operator $\dbar_J\colon \Omega^{1,0}_J\to \Omega^{1,1}_J$ induced by
$J$. The real $2$-form $\rho_J$ is defined as $-\imag$ times the
curvature of the induced connection on the canonical line bundle $K_X
= \Lambda^n_{\CC}T^{\ast}X$, where $\imag$ is the imaginary unit
$\sqrt{-1}$. The Hermitian scalar curvature $S_J$ is the real function
on $X$ defined by
\begin{equation}
\label{eq:def-S}
  S_J \omega^{[n]} = 2\rho_J \wedge \omega^{[n-1]}.
\end{equation}
The normalization is chosen so that $S_J$ coincides with the
Riemannian scalar curvature when $J$ is integrable. The space $\cJ$ of
almost complex structures $J$ on $X$ which are compatible with
$\omega$ is an infinite dimensional K\"ahler manifold, with complex
structure $\mathbf{J} \colon T_J\cJ \to T_J\cJ$ and K\"ahler form
$\omega_{\cJ}$ given by
\begin{equation}
\label{eq:SympJ}
\mathbf{J}\Phi \defeq J\Phi \text{ and }
\omega_{\cJ} (\Psi,\Phi) \defeq \frac{1}{2}\int_{X}\tr(J\Psi \Phi) \omega^{[n]},
\end{equation}
for $\Phi$, $\Psi \in T_J\cJ$, respectively. Here we identify $T_J\cJ$
with the space of endomorphisms $\Phi\colon TX \to TX$ such that
$\Phi$ is symmetric with respect to the induced metric
$\omega(\cdot,J\cdot)$ and satisfies $\Phi J = - J \Phi$.

The group $\cH$ of Hamiltonian symplectomorphisms $h\colon X\to X$
acts on $\cJ$ by push-forward, i.e. $h \cdot J \defeq h_{\ast} \circ J
\circ h_{\ast}^{-1}$, preserving the K\"ahler form. As proved by
Donaldson~\cite[Proposition~9]{D1}, the $\cH$-action on $\cJ$ is
Hamiltonian with equivariant moment map $\mu_\cH\colon \cJ\to
(\LieH)^*$ given by
\begin{equation}
\label{eq:scmom}
\langle \mu_\cH(J), \eta_\phi\rangle = -\int_X \phi S_J \omega^{[n]},
\end{equation}
for $\phi \in C^{\infty}_0(X)$, identified with an element $\eta_\phi$
in $\LieH$ by~\eqref{eq:LieH} and~\eqref{eq:eta_phi}. The
$\cH$-invariant subspace $\cJi \subset \cJ$ of integrable almost
complex structures is a complex submanifold (away from its
singularities), and therefore inherits a K\"ahler structure.  Over
$\cJi$, the Hermitian scalar curvature $S_J$ is the Riemannian scalar
curvature of the K\"ahler metric determined by $J$ and $\omega$.
Hence the quotient
\begin{equation}
\label{eq:modulicscK}
\mu_{\cH}^{-1} (0)/\cH,
\end{equation}
where $\mu_\cH$ is now the restriction of the moment map to $\cJi$, is
the moduli space of K\"ahler metrics with fixed K\"ahler form
$\omega$ and constant scalar curvature.  Away from singularities, this
moduli space can thus be constructed as a K\"ahler reduction
(see~\cite{Fj} and references therein for details).

\subsection{The coupled equations as a moment map condition}
\label{sec:Ceqcoupled-equations}

Fix a compact symplectic manifold $X$ of dimension $2n$ with symplectic form
$\omega$, a compact Lie group $G$ and a smooth principal $G$-bundle $E$ on
$X$. Let $\cJ$ be the space of almost complex structures compatible with
$\omega$ and $\cA$ the space of connections on $E$. Using the symplectic forms
on $\cA$ and $\cJ$ induced by $\omega$ (see~\eqref{eq:SymfC}
and~\eqref{eq:SympJ}), we define a symplectic form on the product $\cJ \times
\cA$, for each pair of non-zero real constants $\alpha = (\alpha_0,
\alpha_1)$, as the weighted sum
\begin{equation}
\label{eq:Sympfamily}
\omega_\alpha = \alpha_0  \omega_\cJ + 4\alpha_1  \omega_\cA
\end{equation}
(we omit pullbacks to $\cJ\times\cA$). The extended gauge group $\cX$
has a canonical action on $\cJ \times \cA$ and this action is
symplectic for any $\omega_\alpha$. Following the notation
of~\secref{sec:Ceqcoupling-term}, this action is given by
\[
  g\cdot (J,A)=(\pr(g)\cdot J, g\cdot A),
\]
for $g\in\cX$ and $(J,A)\in \cJ \times \cA$, with $\pr$ as
in~\eqref{eq:coupling-term-moment-map-1}. Using the moment maps $\mu_\cH$ and
$\mu_\cX$ given by~\eqref{eq:scmom} and Proposition~\ref{prop:momentmap-X}, we
obtain the following.

\begin{proposition}
\label{prop:momentmap-pairs}
The $\cX$-action on $\cJ\times \cA$ is Hamiltonian with respect to
$\omega_\alpha$, with equivariant moment map $\mu_\alpha\colon \cJ \times
\cA\to (\LieX)^*$ given by
\begin{equation}
\label{eq:thm-muX}
\begin{split}
\langle \mu_\alpha(J,A),\zeta\rangle
& = 4 \alpha_1 \int_X \theta_A\zeta \wedge (\Lambda F_A - z) \omega^{[n]}\\
& - \int_X\phi\(\alpha_0 S_J + \alpha_1 \Lambda^2 (F_A\wedge F_A) - 4 \alpha_1\Lambda F_A \wedge z\)\omega^{[n]},
\end{split}
\end{equation}
for all $(J,A)\in\cJ\times\cA$, $\zeta\in\LieX$, and
$\pr(\zeta) = \eta_\phi$ with $\phi \in C_0^\infty(X)$.
\end{proposition}

The $\cX$-action also preserves the almost complex structure
$\mathbf{I}$ on $\cJ \times \cA$ given by
\begin{equation}
\label{eq:complexstructureI}
\mathbf{I}(\dot{J},a) = (J\dot{J},-a(J \cdot)),
\end{equation}
for all $(\dot{J},a)\in T_J\cJ\times T_A\cA$.
Using the complex structure $\mathbf{J}$ on $\cJ$ given by~\eqref{eq:SympJ},
the canonical projection $\cJ \times \cA \to \cJ$ becomes now a holomorphic
submersion. It is easy to see that, for $\alpha_0, \alpha_1$ positive, the
complex structure $\mathbf{I}$ is compatible with the family of symplectic
structures \eqref{eq:Sympfamily}. The formal integrability of the almost complex structure
$\mathbf{I}$ is not obvious \emph{a priori}, so we now provide a proof of this
fact. By ``formal integrability'' here, we mean, as in \cite{D6}, that the
associated Nijenhuis tensor vanishes.

\begin{proposition}
\label{prop:CeqintegrableI}
The almost complex structure $\mathbf{I}$ is formally integrable.
\end{proposition}

\begin{proof}
Since the complex structure $\mathbf{J}$ on the base $\cJ$ and the one on each
fibre are integrable, the integrability condition for $\mathbf{I}$ reduces to
the vanishing condition for the value of the Nijenhuis tensor $N_{\mathbf{I}}$
on each pair of vectors $\dot{J} \in T_J\cJ$, $a \in T_A\cA$, for $(J,A)
\in \cJ\times\cA$. Now, $a$ and $\dot{J}$ extend to vector fields on $\cA$ and
$\cJ$, respectively, and hence to $\cJ\times\cA$ ($a$ extends to a constant
vector field on the affine space $\cA$ and $\dot{J}$ extends to a vector field
on $\cJ$ given by $\dot{J}_{|J'} = (1/2)(J\dot{J}J' -
J'J\dot{J})$). Furthermore,
\[
N_{\mathbf{I}}(\dot{J},a)
 = [\mathbf{I}\dot{J},\mathbf{I}a] - \mathbf{I}[\mathbf{I}\dot{J},a] - \mathbf{I}[\dot{J},\mathbf{I}a] - [\dot{J},a]
= [\mathbf{I}\dot{J},\mathbf{I}a] - \mathbf{I}[\dot{J},\mathbf{I}a],
\]
where the brackets denote the Lie brackets between vector fields on
$\cJ\times\cA$ and we have used the fact that $[\mathbf{I}\dot{J},a] =
[\dot{J},a] = 0$ because the flow of $a$ covers the identity on $\cJ$.
To compute the remaining terms, we denote by $J_t(\dot{J})$ the flow of
any vector field $\dot{J}$ on $\cJ$, viewed as a vector field on $\cJ \times
\cA$. Then $J_t(\dot{J})$ induces the identity on $\cA$, and hence
\begin{align*}
N_{\mathbf{I}}(\dot{J},a) & = \frac{d}{dt}_{|t = 0} \mathbf{I}a_{|J_t(\mathbf{I}\dot{J})} - \mathbf{I}_{|J}
\frac{d}{dt}_{|t = 0} \mathbf{I} a_{|J_t(\dot{J})}
\\ &
= - \frac{d}{dt}_{|t = 0} a(J_t(\mathbf{I}\dot{J}) \cdot)
+ \mathbf{I}_{|J} \frac{d}{dt}_{|t = 0}a(J_t(\dot{J})
\cdot)
\\ &
= - a(J\dot{J} \cdot) - a(\dot{J}J \cdot)
  = - a(J\dot{J} + \dot{J}J \cdot) = 0,
\end{align*}
where $a$ is now viewed as an element of $\Omega^1(\ad E)$.

Note that the vanishing of $N_{\mathbf{I}}(\dot{J},a)$ does not require any
compatibility condition between $J$ and $\omega$.
\end{proof}

\begin{remark}
  There is another $\cX$-invariant almost complex structure on $\cJ
  \times \cA$ which is given by $\mathbf{I}'(\dot{J},a) = (J\dot{J},
  a(J \cdot))$. This is compatible with $\omega_\alpha$ for $\alpha_0
  > 0 > \alpha_1$, and the projection $\cJ \times \cA \to \cJ$ is
  pseudoholomorphic for this $\mathbf{I}'$, but one can modify the
  proof of Proposition~\ref{prop:CeqintegrableI} to show that
  $\mathbf{I}'$ is not formally integrable.
\end{remark}

Suppose now that $X$ has K\"ahler structures with K\"ahler form
$\omega$. In the notation of~\secref{section:cscKmmap}, this means
that the subspace $\cJi \subset \cJ$ of integrable almost complex
structures compatible with $\omega$ is not empty. Define
\begin{equation}
\label{eq:cP}
  \cP\subset \cJ\times \cA
\end{equation}
as the space of pairs $(J,A)$ with $J\in \cJi$ and $A\in \cA^{1,1}_J$,
where $\cA^{1,1}_J\subset\cA$ is the space of connections defined
in~\eqref{eq:cA^{1,1}_J}. Then $\cP\subset \cJ\times \cA$ is a
$\cX$-invariant complex and hence K\"ahler subspace by construction
(see also Lemma~\ref{lem:pairsinvacs2}).

We say that a pair  $(J,A)\in\cP$ satisfies  the \emph{coupled K\"ahler--Yang--Mills equations} if
\begin{equation}
\label{eq:CYMeq}
\left. \begin{array}{l}
\Lambda F_A = z\\
\alpha_0 S_J \; + \; \alpha_1 \Lambda^2 (F_A \wedge F_A) = c
\end{array}\right \},
\end{equation}
where $S_J$ is the scalar curvature of the metric $g_J =
\omega(\cdot,J\cdot)$ on $X$ and $c \in \RR$. These equations are the
central subject of this paper.  The set of solutions to the coupled
equations is invariant under the action of $\cX$ and we define the
moduli space of solutions as the set of all solutions modulo the
action of $\cX$. We have the following.

\begin{proposition}
\label{mm-equations}
The subset $\mu_{\alpha}^{-1}(0)\subset \cP$ coincides with the set of
pairs $(J,A)\in\cP$ satisfying equations~\textup{\eqref{eq:CYMeq}}.
\end{proposition}

\begin{proof}
Suppose that $(J,A) \in \mu_{\alpha}^{-1}(0)$. First, evaluating $\mu_\alpha(J,A)$ on
elements of the form $\theta_A^\perp\eta$ with $\eta \in \LieH$, we
see that there exists a $c' \in \RR$ such that
\begin{equation}
\begin{split}
\label{eq:identitysquaredFA}
\frac{c' - \alpha_0 S_J}{\alpha_1} & =  \Lambda^2(F_A \wedge F_A )- 4\Lambda F_A \wedge z \\
& = 2|\Lambda F_A|^2
- 2\lvert F_A\rvert^2 +
8\lvert F_A^{0,2}\rvert^2 - 4\Lambda F_A \wedge z,
\end{split}
\end{equation}
where the last equality follows from a pointwise computation
(cf. \cite[proof of Lemma~7.9]{MR}). Here, the pointwise norms are
defined using the metric $g_J=\omega(\cdot,J\cdot)$ and the inner
product $(\cdot,\cdot)$ on $\mathfrak{g}$ and $F_A^{0,2}$ denotes the
$(0,2)$ part of $F_A$ with respect to $J$. Second, as
$\langle\mu_\alpha(J,A), \zeta\rangle=0$ for all $\zeta \in \LieG$, we
have $\Lambda F_A = z$ and hence it is straightforward to see from
\eqref{eq:identitysquaredFA} that
\[
\alpha_0 S_J + \alpha_1 \Lambda^2 (F_A \wedge F_A)
= c' + 4\alpha_1\lvert z\rvert^2 \in \RR.
\]
The converse follows also from \eqref{eq:identitysquaredFA}.
\end{proof}

Note that we have not used the fact that $(J,A) \in \cP$. Observe also
that $c$ is a `topological constant', i.e. it only depends on the
cohomology class $\Omega\defeq [\omega]\in H^2(X)$, the topology of
the bundle $E$ and the coupling constants $\alpha_0, \alpha_1$
(cf. Remark~\ref{rem:z(Omega,E)}). This follows by integrating the
second equation in~\eqref{eq:CYMeq} over $X$, obtaining
\begin{equation}
\label{eq:constant-c}
c = \alpha_0 \hat{S} + 2\alpha_1 \hat{c},
\end{equation}
where $\hat{S}$ is the average of the Hermitian scalar curvature,
\begin{equation}
\label{eq:hat-S}
  \hat{S}\defeq \frac{\int_X S_J \omega^{[n]}}{\int_X\omega^{[n]}}  =
  2\pi \frac{\left\langle c_1(X)\cup \Omega^{[n-1]},[X]\right\rangle}{\Vol_\Omega},
\end{equation}
which only depends on the cohomology class $\Omega\in H^2(X)$, and
\begin{equation}
\label{eq:constant-c-hat}
\hat{c} \defeq \frac{\int_X F_A \wedge F_A \wedge \omega^{[n-2]}}{\int_X\omega^{[n]}} = \frac{\left\langle c(E) \cup \Omega^{[n-2]},[X]\right\rangle}{\Vol_\Omega},
\end{equation}
where $c(E) \defeq [F_A \wedge F_A] \in H^4(X)$ is the Chern--Weil
class associated to the $G$-invariant symmetric bilinear form $(\cdot,
\cdot)$ on $\mathfrak{g}$, and so $\hat c$ only depends on $\Omega$
and the topology of $E$ (see~\cite[Ch XII, \S 1]{KNII}).

From Proposition \ref{mm-equations}, we can identify the moduli space
of solutions to the coupled equations with the quotient
\begin{equation}
\label{eq:symplecticreduccX}
\mu_{\alpha}^{-1}(0)/\cX,
\end{equation}
where $\mu_{\alpha}$ denotes now the restriction of the moment map to
$\cP$.  Away from singularities, this is a K\"ahler quotient for the
action of $\cX$ on the smooth part of $\cP\subset \cJ\times \cA$
equiped with the K\"ahler form obtained by the restriction of
$\omega_\alpha$.

\begin{remark}
\label{rem:coupled-eq}
The coupled equations~\eqref{eq:CYMeq} can also be written as
\begin{equation}
\label{eq:CYMeq-v2}
\left. \begin{array}{l}
\Lambda F_A = z\\
\alpha_0 S_g - 2\alpha_1 \lvert F_A\rvert_g^2 = c - 2 \alpha_1\lvert z\rvert^2
\end{array}\right \}.
\end{equation}
Here $S_g$ is the scalar curvature of the metric $g=
\omega(\cdot,J\cdot)$, $\lvert F_A\rvert^2_g$ is the pointwise norm of
$F_A$ defined using $g$ and the inner product $(\cdot,\cdot)$ on
$\mathfrak{g}$, and $z\in\mathfrak{z}$, $c \in \RR$ are as
in~\eqref{eq:CYMeq}. The purely Riemannian nature of the second
(scalar) equation in~\eqref{eq:CYMeq-v2} will be used
in~\secref{sec:CeqscalarKal-K}. The equivalence of~\eqref{eq:CYMeq}
and~\eqref{eq:CYMeq-v2} follows from~\eqref{eq:identitysquaredFA}
using that $A\in\cA_J^{1,1}$ (i.e. $F_A^{0,2}=0$).
\end{remark}

\subsection{The Calabi--Yang--Mills functional}
\label{sec:CeqscalarKal-K}
\newcommand{\CYM}{\operatorname{CYM}_\alpha}

K\"ahler metrics of constant scalar curvature arise as the absolute
minima of the Calabi functional~\cite{Ca}, which is defined as the
$L^2$-norm of the scalar curvature for K\"ahler metrics running over a
fixed K\"ahler class on a compact complex manifold. Alternatively, we
can see the cscK metrics as the absolute minima of the $L^2$-norm of
the scalar curvature defined over the space $\cJ^i$ of complex
structures compatible with a fixed symplectic form $\omega$ (see
e.g.~\cite{FO}). As a further step in Calabi's programme, in this
section we define the Calabi--Yang--Mills (CYM) functional
$\CYM$. This is a purely Riemannian functional that intertwines the
Yang--Mills functional for connections with the $L^2$-norm of the
scalar curvature of invariant metrics in the total space of the
principal bundle $E$. Interpreting the elements of $\cJ \times \cA$ as
invariant Riemannian metrics $g_\alpha$ on $E$, we prove that the
absolute minima of $\CYM$ over $\cJ^i \times \cA$ are precisely the
solutions $(J,A)\in\cP$ of~\eqref{eq:CYMeq}. We will also see that the
coupled equations~\eqref{eq:CYMeq} can be formulated in terms of the
Ricci tensor and the scalar curvature of $g_\alpha$, when it is
defined by an element of a suitable subspace $\cP^* \subset \cP$.

We start with a principal $G$-bundle $E$ over a compact manifold $X$
and a fixed $G$-invariant inner product $(\cdot,\cdot)$ on
$\mathfrak{g}$. Consider the $G$-invariant metric $g_V$ on the
vertical bundle $VE\subset TE$ induced by $(\cdot,\cdot)$ via the
identification of $VE$ with the trivial bundle
$E\times\mathfrak{g}$. Using a connection $A$ on $E$ and a scaling
constant $\alpha > 0$, each Riemannian metric $g$ on $X$ lifts to a
$G$-invariant Riemannian metric $g_\alpha$ on $E$, given by
\begin{equation}
\label{eq:invariant-metric}
  g_\alpha = \pi^*g + \alpha g_V(\theta_A \cdot , \theta_A \cdot),
\end{equation}
where $\pi\colon E\to X$ is the canonical projection and $\theta_A
\colon TE \to VE$ is the vertical projection determined by $A$. Given
positive constants $\alpha_0, \alpha_1\in\RR$, we denote respectively
by $S_{g_\alpha}$, $\vol_\alpha$ and $\Vol_\alpha(E)$ the scalar
curvature and the volume form of the metric $g_\alpha$ and the volume
of $E$ with respect to $g_\alpha$, where $\alpha =
2\alpha_1/\alpha_0$. We also denote by $\vol_g$ and $\Vol_g(X)$ the
volume form of the metric $g$ and the corresponding volume of $X$,
respectively. We define the Calabi--Yang--Mills functional by the
formula
\begin{equation}
\label{eq:CYMfunctional}
\CYM(g,A) \defeq \frac{1}{\Vol_\alpha(E)} \int_E S_{g_\alpha}^2 \vol_\alpha
+ \frac{\alpha_1}{\Vol_g(X)}\int_X |F_A|^2_g\vol_g,
\end{equation}
for pairs $(g,A)$, where $g$ is a Riemannian metric on $X$, $A$ is a
connection on $E$ and $|F_A|^2_g$ is as in~\eqref{eq:CYMeq-v2}. Note
that~\eqref{eq:CYMfunctional} is a weighted sum of the Calabi
functional~\cite{Ca} for metrics on $E$ and the Yang--Mills functional
for $E$ (see e.g.~\cite[\S 2.1.4]{DK}).

Fix now a symplectic form $\omega$ on $X$ so that
$\vol_{g_J}=\omega^{[n]}$ for all $J\in\cJ$, where $g_J =
\omega(\cdot,J\cdot)$ and $\dim X=2n$. Although the
functional~\eqref{eq:CYMfunctional} is well defined for arbitrary
Riemannian metrics on $X$ and connections on $E$, the solutions of the
coupled equations~\eqref{eq:CYMeq} are the absolute minima of $\CYM$
only when this functional is restricted to metrics of the form
$g=g_J$, where $J$ is in the space $\cJ^i$ of integrable almost
complex structures on $X$ which are compatible with $\omega$. In other
words, we consider the functional
\begin{equation}
\label{eq:CYMfunctional-v2}
\begin{gathered}
\xymatrix @R=0ex @C=-6.5ex{
**[l]\cJ^i\times\cA\colon\;\Omega^k\ar[r]&**[r]\RR\\
**[l](J,A)\ar@{|->}[r]&**[r]\CYM(g_J,A).
}\end{gathered}
\end{equation}

\begin{proposition}
\label{propo:Variational}
If $(J,A)\in\cP$ satisfies the coupled equations~\eqref{eq:CYMeq},
then the pair $(J,A)$ is an absolute minimum of the
functional~\eqref{eq:CYMfunctional-v2}, provided that $\alpha_0$ and
$\alpha_1$ are positive and
\begin{equation}
\label{eq:ineqalpha}
\alpha_1 > 2 \alpha \hat S + \alpha^2(\hat c - |z|^2) + 2s,
\end{equation}
where $\alpha = 2\alpha_1/\alpha_0$, $s$ is the (constant) scalar
curvature of the biinvariant metric induced by $(\cdot,\cdot)$ on $G$,
$z$ is given by~\eqref{eq:z(Omega,E)} and $\hat S$, $\hat c$ are as in
\eqref{eq:constant-c}, with $\Omega=[\omega]$.
\end{proposition}

\begin{proof}
Note first that for any metric as in \eqref{eq:invariant-metric},
$\pi\colon (E,g_\alpha) \to (X,g)$ is a Riemannian submersion with
totally geodesic fibres (see~\cite[Theorem~9.59]{Be}, where the
$G$-Riemannian manifold playing the role of the typical fibre is $G$
itself with its biinvariant metric). Then $g_\alpha$ has scalar
curvature $S_{g_\alpha} = S_\alpha\circ\pi$, where
\begin{equation}
\label{eq:scalarKal-K}
S_\alpha = S_g - \alpha |F_A|^2_g + s/\alpha \in C^\infty(X),
\end{equation}
$S_g$ being the scalar curvature of $g$
(see~\cite[Proposition~9.70]{Be}). Here, the group is identified with
the fibre $E_x$ over $x \in X$. Since the volume of $E_x$ is
independent of $x$, we have
\[
\frac{1}{\Vol_\alpha(E)}\int_E S_{g_\alpha}^2 \vol_\alpha  = \frac{1}{\Vol_g(X)} \int_X S_\alpha^2 \vol_g.
\]
In particular, for $g = g_J$, with $J \in \cJ^i$, and $c'' = \alpha_0 \hat S +
2\alpha_1(\hat c - |z|^2)$, we obtain
\begin{align*}
\CYM (g,A) =& \frac{\alpha_0^{-2}}{\Vol_\Omega}\left\|\alpha_0 S_g - 2\alpha_1 |F_A|_g^2 - c''\right\|^2_{L^2} +   \frac{\alpha_1}{\Vol_\Omega}  \|F_A\|^2_{L^2}\\
&  + \frac{2 (c''/\alpha_0 + s/\alpha)}{\Vol_\Omega}\int_X \(S_g - \alpha |F_A|_g^2 - c''/\alpha_0\)  \vol_g\\
&  +  (c''/\alpha_0 + s/\alpha)^2\\
=& \frac{\alpha_0^{-2}}{\Vol_\Omega}\left\|\alpha_0 S_g - 2\alpha_1 |F_A|_g^2 - c''\right\|^2_{L^2}\\
&  +   \frac{\alpha_1 - 2 \alpha \hat S - \alpha^2(\hat c - |z|^2) - 2s}{\Vol_\Omega}  \|F_A\|^2_{L^2}\\
&   +  \(\hat S + \alpha(\hat c - |z|^2) + s/\alpha\) \(\hat S - \alpha(\hat c - |z|^2 + s/\alpha\),
\end{align*}
where the $L^2$-norms are defined using $g$, $\omega^{[n]}$ and the
inner product on $\mathfrak{g}$. Note that the last summand in the
right-hand side of the last equation only depends on $\alpha$, $s$,
the cohomology class $\Omega\defeq [\omega]$ and the topology of the
bundle $E$. The inequality~\eqref{eq:ineqalpha} implies that the
factor multiplying the Yang--Mills functional is positive, and the
result follows from the alternative formulation~\eqref{eq:CYMeq-v2} of
the coupled equations combined with~\eqref{eq:identitysquaredFA},
which gives
\begin{align*}
\|F_A\|^2_{L^2} & = \|\Lambda F_A\|^2_{L^2} + 4 \|F_A^{0,2}\|^2_{L^2} - \hat c \Vol_\Omega\\
& = \|\Lambda F_A - z\|^2_{L^2} + 4 \|F_A^{0,2}\|^2_{L^2}\\
& + 2 \langle z(E)\cup \Omega^{[n-1]},[X]\rangle - (|z|^2 + \hat c) \Vol_\Omega.
\end{align*}
Here, $z(E) \defeq [z \wedge F_A] \in H^2(X)$ is the Chern--Weil class
associated to the $G$-invariant linear form $(z,\cdot)$ on
$\mathfrak{g}$, with $z$ given by~\eqref{eq:z(Omega,E)}, so the last
line in the previous equation only depends on $\Omega\in H^2(X)$ and
the topology of the bundle $E$ (see~\cite[Ch XII, \S 1]{KNII}).
\end{proof}

\begin{remark}
\label{remark:Variational}
The inequality~\eqref{eq:ineqalpha} imposes no restrictions on the
solutions $(J,A)$ of~\eqref{eq:CYMeq}, because any solution $(J,A)$
of~\eqref{eq:CYMeq} for some $(\alpha_0,\alpha_1)$ is also a solution
for the constants $(t\alpha_0,t\alpha_1)$, for all $t \in \RR$. The
claim follows from the fact the RHS on~\eqref{eq:ineqalpha} is
invariant by this scaling procedure.
\end{remark}

\begin{remark}
\label{remark:Variational2}
Fixing a complex structure on $X$, we can view $\CYM$ as a functional
on the pairs $(\omega,A)$, with $\omega$ as in the second part of
Remark~\ref{rem:coupled-eq}. Exactly as in
Proposition~\ref{propo:Variational}, in this case a solution of the
coupled equations is always an absolute minima of this functional.
\end{remark}

The coupled equations~\eqref{eq:CYMeq} can also be interpreted in
purely Riemannian terms, considering the $G$-invariant metrics
$g_\alpha$ on $E$ defined by \eqref{eq:invariant-metric}. To explain
this, note that given such a metric its Ricci tensor $R_{g_\alpha}$
decomposes as $R_{g_\alpha} = (R_{g_\alpha})_{hh} +
(R_{g_\alpha})_{vv} + (R_{g_\alpha})_{hv}$, where the indices ``$h$''
and ``$v$'' denote the horizontal and vertical directions in $TE$
defined by the connection $A$, respectively. Let $\cP^*\subset\cP$ be
the open subset of pairs $(J,A)$ with $A\in\cA^*$--- the open subset
of $\cA$ consisting of irreducible connections. By irreducible
connection $A \in \cA$ we mean, as in~\cite[\S 4.2.2]{DK}, that its
isotropy group $\cG_A$ inside the gauge group of $E$ is minimal--- the
centre of $G$. Then a pair $(J,A)\in\cP^*$ satisfies~\eqref{eq:CYMeq}
if and only if the associated metric $g_\alpha$ satisfies the
following equations.
\begin{equation}
\label{eq:CYMeq-v3}
\left. \begin{array}{l}
(R_{g_\alpha})_{hv}=0\\
S_{g_\alpha}= \text{const.}
\end{array}\right \}
\end{equation}
We thus have an interpretation of the K\"ahler
quotient~\eqref{eq:symplecticreduccX} (with $\mu_\alpha$ restricted to
the open subset $\cP^*\subset\cP$) as a moduli space of $G$-invariant
metrics on the total space of $E$ satisfying~\eqref{eq:CYMeq-v3}.
An interesting fact here is that the condition $\alpha_1/\alpha_0 > 0$
is needed both to have a K\"ahler form $\omega_\alpha$ on $\cP$ given
by~\eqref{eq:Sympfamily} (see the explanation
before~\eqref{eq:symplecticreduccX}) and $G$-invariant Riemannian
metrics $g_\alpha$ on $E$, as given in~\eqref{eq:invariant-metric}.

To prove the equivalence of~\eqref{eq:CYMeq} and~\eqref{eq:CYMeq-v3}
for a pair $(J,A)\in\cP^*$, note that $J$ defines a structure of
K\"ahler manifold on $(X,\omega)$. The Hermitian--Yang--Mills equation
$\Lambda F_A=0$ for an \emph{irreducible} $A\in\cA^{1,1}_J$ is
equivalent to the \emph{a priori} weaker Yang--Mills equation $d^*_A
F_A=0$ (see~\cite[Proposition~3]{D3}). This follows because if $A \in
\cA^{1,1}_J$ is an irreducible Yang--Mills connection then, by the
K\"ahler identities,
\[
d_A \Lambda F_A = 0 \,\Longrightarrow\, \Lambda F_A \in \LieG_A = \mathfrak{z}.
\]
Therefore the first equations in~\eqref{eq:CYMeq-v2}
and~\eqref{eq:CYMeq-v3} are equivalent because the Yang--Mills
equation is equivalent to the equation $(R_{g_t})_{hv}=0$
(see~\cite[Proposition~9.61]{Be}). Finally, the second equations
in~\eqref{eq:CYMeq-v2} and~\eqref{eq:CYMeq-v3} are equivalent
by~\eqref{eq:scalarKal-K}.

Note that the system~\eqref{eq:CYMeq-v3} is half way between the
Einstein equation and the constant scalar curvature equation, in the
sense that
\begin{equation}
\label{eq:KK-eq}
\text{$g_\alpha$ is an Einstein metric $\implies$  $g_\alpha$
  satisfies~\eqref{eq:CYMeq-v3} $\implies$ $S_{g_\alpha}=\text{const.}$,}
\end{equation}
for all $(J,A)\in\cP^*$, as any metric $g_\alpha$ satisfying the
Einstein equation $R_{g_\alpha}=\lambda g_\alpha$ (with
$\lambda\in\RR$) has constant scalar curvature.

\section{The $\alpha$-Futaki character and the $\alpha$-K-energy}
\label{chap:analytic}

In~\secref{chap:analytic} we construct obstructions to the existence
of solution of the coupled equations, generalizing the Futaki
character~\cite{Ft0}, the Mabuchi K-energy~\cite{Mab2,Mab1} and the
notion of geodesic stability~\cite{Ch2,D6} used in the cscK
Theory. For this,
in~\secrefs{sub:concrete-setup},~\ref{sub:ANsymm-spc},~\ref{sub:ANStability},
we develop an abstract framework that we apply in~\secref{sub:ANCeq}
to the study of the coupled equations.

Throughout~\secref{chap:analytic}, we fix a compact real manifold $X$,
a cohomology class $\Omega\in H^2(X,\RR)$, a reductive complex Lie
group $G^c$ with Lie algebra $\mathfrak{g}^c$, a maximal compact Lie
subgroup $G\subset G^c$ with Lie algebra $\mathfrak{g}$ and a smooth
principal $G^c$-bundle $\pi\colon E^c\to X$. We also fix $z \in
\mathfrak{z}$ as in \eqref{eq:z(Omega,E)}. We assume that the space of
K\"ahler forms in $\Omega$ is non-empty.

\subsection{Invariant Hamiltonian K\"ahler fibrations}
\label{sub:concrete-setup}

In~\secref{sub:concrete-setup}, we will associate to the data $(X,
\Omega, E^c)$ a canonical infinite-dimensional double fibration
\[
\cB\stackrel{\pi_\cB}{\longleftarrow}\cC\stackrel{\pi_\cZ}{\longrightarrow}\cZ,
\]
equivariant for the action of an infinite-dimensional Lie group
$\Gamma$, and show that the fibres of $\pi_\cB$ are (formally)
K\"ahler manifolds with Hamiltonian group actions. The fibres of
$\pi_\cZ$ will be studied in~\secref{sub:ANsymm-spc}.

Let $\Diff_0 X$ be the identity component of the diffeomorphism group
of $X$ and $\Aut E^c$ the group of automorphisms of $E^c$, that is,
the $G^c$-equivariant diffeomorphisms $g\colon E^c\to E^c$. Any such
$g$ determines a unique diffeomorphism $\check{g}\colon X\to X$ such
that $\pi\circ g=\check{g} \circ \pi$. Define the real Lie group
\[
  \Gamma \defeq\{ g\in\Aut E^c \,|\, \check{g}\in\Diff_0 X  \}.
\]
Note that the Lie bracket in the Lie algebra $\LieGamma$ of $\Gamma$ is
\begin{equation}
\label{eq:Lie-bra-Gamma}
  [y,y']_{\Gamma}=-[y,y']
\end{equation}
for $y,y'\in \LieGamma\subset \Omega^0(TE^c)$, where $[\cdot,\cdot]$ is the
Lie bracket of vector fields on $E^c$ (cf.~\cite[Remark~3.3]{McS}).

Let $\cZ$ be the space of holomorphic structures on the principal
$G^c$-bundle $E^c$, i.e. the integrable $G^c$-equivariant almost
complex structures $I$ on the total space of $E^c$ which preserve the
vertical bundle $VE^c$ and whose restriction to $VE^c$ equals
multiplication by $\sqrt{-1}$, via its identification with
$E^c\times\mathfrak{g}^c$. By
$G^c$-equivariance, any such $I$ determines a unique integrable almost
complex structure $\check{I}$ on $X$ such that $\check{I} \circ
d\pi=d\pi\circ I$. The group $\Gamma$ has a left action on $\cZ$ by
push-forward,
preserving the canonical almost complex structure $\mathbf{I}$ on $\cZ$ given
by
\begin{equation}
\label{eq:complexstructureIZ}
\mathbf{I}\dot{I} = I\dot{I},
\text{ for all $I \in \cZ$, $\dot{I} \in T_I\cZ$}
\end{equation}
(cf.~\eqref{eq:SympJ}), where $\dot{I}$ is viewed as a $G^c$-equivariant
endomorphism of $TE^c$.

Recall that the space $\cR=\Omega^0(E^c/G)$ of
smooth sections $H$ of the bundle $E^c/G \to X$ is in bijection with the set
of reductions of $E^c$ to principal $G$-bundles $E_H\subset E^c$, via the map
$H \mapsto E_H\defeq p_G^{-1}(H(X))$, where $p_G$ is the projection
$E^c\to E^c/G$.
Let $\cB$ be the space of pairs $(\omega,H)$, where $\omega\in\Omega$ is a
symplectic form and $H\in\cR$.
The group $\Gamma$ has a left action on $\cB$ given by
\[
  g\cdot(\omega,H) = (\check{g}_*\omega, g\cdot H),
\]
where $(g\cdot H)(x)\defeq g(x)\cdot H(\check{g}^{-1}(x))$ for $x\in X$ and
$\check{g}_*\omega\in\Omega$ by the homotopy invariance of the de Rham
cohomology, as $\check{g}\in\Diff_0 X$.

We define the space of \emph{compatible pairs} as
\[
  \cC \defeq\{ ((\omega,H),I)  \,|\,
  (X,\check{I},\omega) \text{ is a K\"ahler manifold}\}
  \subset \cB\times\cZ.
\]
Note that this space is invariant under the diagonal $\Gamma$-action on
$\cB\times\cZ$. The canonical maps
\begin{equation}
\label{eq:fibration-compatible-pairs}
\begin{gathered}
  \xymatrix{ & \cC \ar[dl]_-{\pi_\cB} \ar[dr]^-{\pi_\cZ} & \\  \cB & & \cZ }
\end{gathered}
\end{equation}
will be viewed as two fibrations with total space $\cC$, whose fibres
are
\[
  Z_b\defeq \pi_\cB^{-1}(b)
  \text{ and }
  B_I\defeq \pi_\cZ^{-1}(I)
  \text{ for all $b\in\cB$, $I\in\cZ$}.
\]
Since $\cC\subset \cB\times\cZ$ is $\Gamma$-invariant, the fibres $B_I$ and
$Z_b$ are invariant under the actions of the isotropy groups $\Gamma_I \subset
\Gamma$ and $\Gamma_b \subset \Gamma$, respectively.

In more concrete terms, for any $I \in \cZ$, the isotropy group $\Gamma_I$ is
the group of automorphisms $g$ of the holomorphic principal $G^c$-bundle
$(E^c,I)$ such that $\check{g}\in\Diff_0X$ is an automorphism of the complex
manifold $(X,\check{I})$.
Similarly, for any $b=(\omega,H)\in\cB$, the isotropy group $\Gamma_b$ is the
group of automorphisms $g$ of the principal $G$-bundle $E_H$ such that
$\check{g}\in\Diff_0X$ is a symplectomorphism of $(X,\omega)$.
Hence the extended gauge group $\cX_b$ of $E_H$ on $(X,\omega)$ (defined
in~\secref{sec:Ceqcoupling-term}) is a subgroup of $\Gamma_b$, which is normal
because the group of Hamiltonian symplectomorphisms is a normal subgroup of
the symplectomorphism group (see e.g.~\cite[Proposition~10.2]{McS}).
Note also that the fibre $B_I$ is a contractible space, as it is
\begin{equation}
\label{eq:B-product-K-R}
B_I=\cK_{\check{I}}\times\cR,
\end{equation}
where $\cK_{\check{I}}$ is the space of K\"ahler forms in $\Omega$ on
the complex manifold $(X,\check{I})$. The fibre $Z_b$ has a gauge-theoretic
description. Let $\cJ_\omega$ be the space of almost complex
structures on $X$ compatible with $\omega$ and $\cA_H$ the space of
connections on $E_H$. Given $b=(\omega,H) \in \cB$, define
\begin{equation}
\label{eq:cP_b}
\cP_b \subset \cJ_\omega\times \cA_H
\end{equation}
as in~\eqref{eq:cP}, i.e. as the space of pairs $(J,A)$ such that $J$
is integrable and $F_A\in\Omega^{1,1}_J(\ad E_H)$. This subspace is
clearly $\Gamma_b$-invariant and has an almost complex structure
$\mathbf{I}$ given by~\eqref{eq:complexstructureI}, which is formally
integrable by Proposition~\ref{prop:CeqintegrableI}.
Note also that for all $H \in \cR$, each connection $A\in\cA_H$
induces canonically a connection on $E^c$, given by $G^c$-equivariant
maps
\begin{equation}
\label{eq:induced-theta}
\theta_A\colon TE^c \to VE^c, \quad \theta_A^{\perp} \colon\pi^*TX \to TE^c,
\end{equation}
where $\pi\colon E^c \to X$ is the canonical projection
(cf.~\eqref{eq:principal-bundle-ses}), via the canonical isomorphism
\begin{equation}
\label{eq:E^c-vs-E_H}
  E^c\cong E_H\times_G G^c
\end{equation}
of principal $G^c$-bundles (with $G$ acting on $G^c$ by left
multiplication).

\begin{lemma}
\label{lem:pairsinvacs2}
The map $\pi_\cB\colon \cC \to \cB$ is a `$\Gamma$-invariant
almost-complex fibration', that is, its fibres $Z_b\subset\cZ$ are
preserved by $\mathbf{I}$ and their induced almost complex structures
are exchanged by the $\Gamma$-action. 
Furthermore, the map
\begin{equation}
\label{eq:pairsinvacs}
\begin{gathered}
\xymatrix @R=0ex @C=-6.5ex{**[l]
\mathbb{I}\colon\;\cP_b\ar[r] & **[r]\cZ\\
**[l](J,A)\ar@{|->}[r] & **[r] \operatorname{\mathbf{i}} \theta_A
+ \theta_A^{\perp} \circ \pi^*J \circ d\pi,
}\end{gathered}
\end{equation}
is a well-defined $\Gamma_b$-equivariant holomorphic embedding whose
image is $Z_b$, for all $b=(\omega,H)\in\cB$.
\end{lemma}

\begin{proof}
The first assertion follows immediately
from~\eqref{eq:complexstructureIZ}. For the second, note
that~\eqref{eq:pairsinvacs} is well defined by direct computation of
the Nijenhuis tensor of $\mathbb{I}(J,A)$. Using the classical
construction~\cite{Si} of the Chern connection $\theta_{H,I}$ of $I\in
Z_b$ on $E_H$, we see that the map~\eqref{eq:pairsinvacs} is injective
with image $Z_b$, as
\begin{equation}
\label{eq:Chrern-inverse}
  I = \mathbb{I} (\check{I}, \theta_{H,I})
\end{equation}
for all $I\in Z_b$. Furthermore,~\eqref{eq:pairsinvacs} is clearly $\Gamma_b$-equivariant.
Another direct computation shows now that~\eqref{eq:pairsinvacs} is a
holomorphic embedding, i.e. its differential is also injective and
exchanges the almost complex structures on $\cP_b$ and $\cZ$.
\end{proof}

As an immediate consequence, $Z_b\cong \cP_b$ equipped with the
restriction of $\mathbf{I}$ is a formally integrable complex manifold,
by Proposition~\ref{prop:CeqintegrableI}. Using
Lemma~\ref{lem:pairsinvacs2}, we can now transfer the constructions
of~\secref{sec:Ceqcoupled-equations} to the fibres
\begin{equation}
\label{eq:pairsinvacs2}
Z_b = \mathbb{I}(\cP_b),
\end{equation}
obtaining the following theorem, where the Lie groups
$\Gamma_b\subset\Gamma$ and their normal subgroups
$\cG_b\subset\Gamma_b$, parameterized by $b\in\cB$, are viewed as the
fibres of two Lie group subbundles
\begin{equation}
\label{eq:Lie-group-bundles}
  \cX_\cB \subset \Gamma_\cB \subset\cB\times\Gamma
\end{equation}
over $\cB$. Their associated Lie algebra bundles are denoted
$\LieX_\cB \subset \LieGamma_\cB\subset \cB\times\LieGamma$.

\begin{theorem}
\label{prop:Kahlerfibration}
Each pair of positive real numbers $\alpha_0,\alpha_1$ determines a
structure of `$\Gamma$-invariant Hamiltonian K\"ahler fibration' on
$\pi_\cB\colon \cC\to\cB$, that is, a smooth family $\omega_\cC$ of
K\"ahler forms $\omega_b$ on the fibres $Z_b$, parameterized by
$b\in\cB$, which are exchanged by the $\Gamma$-action, and a morphism
\begin{equation}
\label{eq:mmapB}
\mu_\cC \colon \cC \lto (\LieX_\cB)^*
\end{equation}
of fibrations over $\cB$, whose fibre $\mu_b\colon Z_b\to(\LieX_b)^*$ is a moment map for the $\cX_b$-action on $Z_b$, and such that
\begin{equation}
\label{eq:mmapBequiv}
\langle \mu_{g\cdot b}(g\cdot I),\zeta \rangle = \langle \mu_b(I),\Ad(g^{-1})\zeta \rangle
\end{equation}
for all $(b,I) \in \cC$, $g \in \Gamma$, $\zeta \in \LieX_{g\cdot b}$.
\end{theorem}

\begin{proof}
As in~\secref{chap:Ceq}, we fix a $G$-invariant positive definite
inner product on $\mathfrak{g}$. Suppose that it extends to a
$G^c$-invariant symmetric bilinear form $(\cdot,\cdot) \colon
\mathfrak{g}^c \otimes \mathfrak{g}^c \to \CC$ (e.g., we can use
$(\cdot,\cdot)\defeq -\tr (\rho(\cdot)\circ\rho(\cdot))$ for a
faithful representation $\rho\colon G^c\to \GL(r,\CC)$ such that
$\rho(G)\subset\U(r)$). This form induces another one on the adjoint
bundle $\ad E^c=E^c\times_{G^c} \mathfrak{g}^c$, which extends to a
$\CC$-bilinear map
\begin{equation}
\label{eq:pairingcomplex}
\begin{gathered}
\xymatrix @R=0ex @C=-8.5ex{**[l]
\Omega^p(\ad E^c) \times \Omega^q(\ad E^c)\ar[r]&**[r]\Omega^{p+q} \otimes \CC\\
**[l](a_p,a_q)\ar@{|->}[r]&**[r] a_p\wedge a_q
}\end{gathered}
\end{equation}
(cf.~\eqref{eq:Pairing}), which clearly is equivariant under the
action of $\Aut E^c$ given by pull-back.

Fix $\alpha_0,\alpha_1>0$. By the results of~\secref{chap:Ceq}, for
each $b = (\omega,H)\in \cB$ we have a K\"ahler manifold
\begin{equation}
\label{eq:Kahlerfibration-Kah-mnfd}
  (Z_b,\mathbf{I},\omega_b),
\end{equation}
where $\mathbf{I}$ is the restriction
of~\eqref{eq:complexstructureIZ} and $\omega_b$ corresponds
to~\eqref{eq:Sympfamily} via the isomorphism $Z_b\cong\cP_b$ of
Lemma~\ref{lem:pairsinvacs2}. Furthermore, the $\cX_b$-action on
$(Z_b,\omega_b)$ is Hamiltonian, with moment map
\begin{equation}
\label{eq:Kahlerfibration-moment-map}
  \mu_b \colon Z_b \lto (\LieX_b)^*
\end{equation}
which corresponds to the moment map in
Proposition~\ref{prop:momentmap-pairs} via the isomorphism
$Z_b\cong\cP_b$ of Lemma~\ref{lem:pairsinvacs2}. Using now the $(\Aut
E^c)$-equivariance of~\eqref{eq:pairingcomplex}, it is easy to see
that $\omega_b$ and $\mu_b$ are the fibres of a family $\omega_\cC$
defining a $\Gamma$-invariant K\"ahler fibration and a morphism of
bundles as in~\eqref{eq:mmapB}, respectively.

To prove~\eqref{eq:mmapBequiv}, note that the actions of $\Aut E^c$ on
the Chern connection $\theta_{H,I}$ of $H\in\cR$ and $I\in Z_b$,
regarded as a connection on $E^c$, and on its curvature
$F_{H,I}\in\Omega^2(\ad E^c)$, satisfy
\begin{equation}
\label{eq:changecurvature}
g \cdot \theta_{H,I} = \theta_{g\cdot H,g \cdot I}, \qquad g \cdot F_{H,I} = F_{g\cdot H,g\cdot I},
\end{equation}
for all $g \in \Aut E^c, H\in\cR, I\in\cZ$ (cf.~\cite[\S
1.1]{D3}). Given $(b,I)\in\cC$, we define
\begin{equation}
\label{eq:Salpha-bI}
S_\alpha(b,I) \defeq - \alpha_0 S_{\omega,\check{I}}
- \alpha_1 \Lambda^2_\omega \(F_{H,I}\wedge F_{H,I}\) + 4\alpha_1 \Lambda_\omega F_{H,I} \wedge z
\in C^\infty(X),
\end{equation}
where $b=(\omega,H)$ and $S_{\omega,\check{I}}$ is the scalar curvature
of $(X,\check{I},\omega)$. By the equivariance
of~\eqref{eq:pairingcomplex} and the second identity
in~\eqref{eq:changecurvature},
\begin{equation}
\label{eq:Salpha-bI-g}
S_\alpha(gb,gI) = S_\alpha(b,I) \circ \check{g}^{-1},
\end{equation}
for all $g\in\Aut E^c$. Combining now~\eqref{eq:Chrern-inverse},
\eqref{eq:changecurvature} and \eqref{eq:Salpha-bI-g}, and making a
change of variable in~\eqref{eq:thm-muX}, we
obtain~\eqref{eq:mmapBequiv}, as required.
\end{proof}

\begin{remark}
\label{rem:Gromov}
  The two fibrations~\eqref{eq:fibration-compatible-pairs} can be compared
  with those in~\cite[\S 2.C]{Gr}, used to see that the spaces of tamed and
  compatible complex structures on a symplectic vector space are contractible
  (cf. \cite[Proposition 2.51]{McS}).
\end{remark}

\subsection{Invariant fibration by symmetric spaces}
\label{sub:ANsymm-spc}

Throughout~\secref{sub:ANsymm-spc}, we will use the framework
introduced in~\secref{sub:concrete-setup} and in particular the first
part of Lemma~\ref{lem:pairsinvacs2} (however, the isomorphism
$\cP_b\cong Z_b$ of Lemma~\ref{lem:pairsinvacs2} and the families
$\omega_\cC$ and $\mu_\cC$ of Theorem~\ref{prop:Kahlerfibration} will
not be used until~\secref{sub:ANStability}). Our task now is to
construct a canonical structure of `$\Gamma$-invariant symmetric space
fibration' on $\pi_\cZ\colon \cC \to \cZ$, that is, symmetric space
structures on the fibres $B_I$ that are exchanged by the
$\Gamma$-action.
As in~\secref{sec:Ceqham-act-ext-grp}, the Lie groups and manifolds
considered here are infinite dimensional, so one has to be careful
with many standard results in finite dimensions. In particular, the
Newlander--Nirenberg theorem fails in general, so we use the notion of
formally integrable complex structure, as in
Proposition~\ref{prop:CeqintegrableI}.

Let $W$ be the space of complex structures on the real vector space
underlying the Lie algebra $\LieGamma$ (i.e. linear maps whose square
is $-\Id$). Consider the tautological $\Gamma$-equivariant map
\begin{equation}
\label{eq:imag_ICeq}
\cZ \lto W
\end{equation}
which assigns to each $I$ the endomorphism $\LieGamma \to \LieGamma
\colon y \mapsto Iy$. Then, since any $I\in\cZ$ is
integrable,~\eqref{eq:imag_ICeq} satisfies the conditions
\begin{equation}
\label{eq:compatible-Y-imag_I}
Y_{Iy|I} = \mathbf{I} Y_{y|I}, \qquad [y,y']_\Gamma + I [y,I y']_\Gamma + I[I y,y']_\Gamma - [I y,I y']_\Gamma = 0,
\end{equation}
for all $y,y'\in \LieGamma$ (with $[\cdot,\cdot]_\Gamma$ as
in~\eqref{eq:Lie-bra-Gamma}), where
\begin{equation}
\label{eq:inf-action-Gamma-Z}
Y_{y|I}\in T_I\cZ
\end{equation}
is the infinitesimal action of $y\in\LieGamma$ on $I\in\cZ$, given by
the Lie derivative $-L_yI$.

To construct the symmetric space fibration, we first prove that $\cZ$
parameterizes right-invariant formally integrable complex structures
on the group $\Gamma$. Given $g \in \Gamma$, define
\begin{equation}
\label{eq:left-right-action}
\begin{gathered}
\xymatrix @R=0ex @C=1ex{
**[l]L_g\colon\;\Gamma\ar[r] & **[r]\Gamma\\
**[l]h\ar@{|->}[r] & **[r] gh,
}\quad
\xymatrix @R=0ex @C=1ex{
**[l]R_g\colon\;\Gamma\ar[r] & **[r]\Gamma\\
**[l]h\ar@{|->}[r] & **[r] hg,
}
\end{gathered}
\end{equation}
as the left and right multiplication by $g$, respectively. To each $I\in\cZ$, we associate a right-invariant almost complex
structure $\cbI$ on $\Gamma$, defined for $v \in T_g \Gamma$, $g \in \Gamma$ by
\begin{equation}
\label{eq:complexstructureGamma}
\cbI v= (R_g)_*I(R_g)^{-1}_*v.
\end{equation}

\begin{proposition}
\label{prop:Gammacomplex}
The almost complex structure $\cbI$ is formally integrable, for all
$I\in\cZ$.
\end{proposition}

\begin{proof}
The statement follows from the second equation in~\eqref{eq:compatible-Y-imag_I}, evaluating the Nijenhuis tensor
$N_\cbI$ of $\cbI$ on right invariant vector fields.
\end{proof}

The next step in the construction of our symmetric space fibration relies on
the following condition for all $I \in \cZ$ such that $B_I$ is non-empty (this
property will be proved in Proposition~\ref{lem:isom-LieX=TB_I}):
\begin{enumerate}
\label{lem:isom-LieX=TB_I}
\item[($\star$)] 
There exists a well-defined isomorphism of vector bundles
\begin{equation}\label{eq:infactcB-2}
\begin{gathered}
\xymatrix @R=0ex @C=0ex{
**[l]\LieX_{\cB|B_I}\ar[r]^{\cong} & **[r]T B_I\\
**[l](b,\zeta)\ar@{|->}[r] & **[r] Y_{I \zeta|b}
}\end{gathered}
\end{equation}
provided by the infinitesimal action of $I\LieX_b\subset\LieGamma$ on $B_I$.
\end{enumerate}
In the sequel, the inverse of~\eqref{eq:infactcB-2} is denoted
\begin{equation}
\label{eq:vectorfield-zeta}
\zeta_I \colon TB_I \lto \LieX_{\cB|B_I}.
\end{equation}

Given a compatible pair $(b,I)\in\cC$, we define a space
\begin{equation}
\label{eq:principal-bundle-Y}
\cY=\cY_{b,I} \defeq \{ g \in \Gamma \; | \; g\cdot b \in B_I \},
\end{equation}
a map $\pi=\pi_{b,I} \colon \cY\to B_I$ given by $\pi(g) = g\cdot b$ and a
right $\Gamma_b$-action on $\cY$ given by right multiplication in $\Gamma$.

\begin{proposition}
\label{prop:4-tuples-C1-4}
For any $(b,I)\in\cC$, the following properties hold:
\begin{enumerate}
\item[\textup{(1)}]
$\cY$ is principal $\Gamma_b$-bundle over $B_I$.
\item[\textup{(2)}]
There exists a canonical connection $\AA$ on $\cY$, with horizontal lift
\begin{equation}\label{eq:connectioncY}
\begin{gathered}
\xymatrix @R=0ex @C=-4ex{
**[l]\theta_\AA^\perp\colon\;\pi^*TB_I\ar[r] & **[r]T\cY\\
**[l](g,v)\ar@{|->}[r] & **[r] (R_g)_*I\zeta_I(v)
}\end{gathered}
\end{equation}
and curvature given by
\begin{equation}\label{eq:curvaturecY}
F_\AA(v_0,v_1) = (R_g)_*[\zeta_I(v_0),\zeta_I(v_1)]_\Gamma,
\end{equation}
for all $g\in\cY$ and $v_0, v_1 \in T_{g\cdot b}B_I$.
\end{enumerate}
\end{proposition}

\begin{proof}
The $\Gamma_b$-action on $\cY$ is clearly free, so leaving aside global
topological questions, to prove part (1), it suffices to show that
$\pi$ is surjective and induces $\cY/\Gamma_b\cong B_I$, that is, for
all $b'\in B_I$, there exists $g\in\Gamma$ such that $b'=g \cdot b$.
Since $B_I$ is contractible (see~\eqref{eq:B-product-K-R}), there
exists a smooth curve $b_t$ on $B_I$ with $b_0=b$, $b_1=b'$. Let
\begin{equation}
\label{eq:vectorfieldyt}
  y_t= I\zeta_I(\dot b_t) \in \LieGamma,
\end{equation}
with $\zeta_I$ given by~\eqref{eq:vectorfield-zeta}. Let $g_t\in\Gamma$ be the
flow of $y_t$, defined by
\begin{equation}
\label{eq:4-tuples-C1-4-2}
    \dot g_t\cdot g_t^{-1} =y_t,
\end{equation}
with initial condition $g_0=1$. Note that the flow $g_t$ exists for
all $t$ because $y_t$ is $G^c$-invariant, so it covers a vector field
$\check{y}_t$ on $X$, whose flow $\check{g}_t\in \Diff_0 X$ exists for
all $t$ because $X$ is compact
(cf.~\eqref{eq:coupling-term-moment-map-1} and
Remark~\ref{rem:exact-ses}). Now, by the Leibniz rule,
\[
  \frac{d}{dt}\(g_t^{-1}\cdot b_t\)
  = g^{-1}_t \cdot\( -Y_{y_t|b_t} + \dot{b}_t\) = 0,
\]
because $\zeta_I$ inverts the infinitesimal action of
$I\LieX_b\subset\LieGamma$ on $B_I$ (cf.~\cite[p. 17]{D6}). Thus
$g_t^{-1}\cdot b_t$ is independent of $t$, so $b'=g_1\cdot b$, as
required.

For (2), note that the horizontal lift of curves on $B_I$ to $\cY$
determined by the flow of \eqref{eq:vectorfieldyt} defines a canonical
connection $\AA$ on $\cY$. To obtain~\eqref{eq:connectioncY}, let
$b_t$ be a curve on $B_I$ with $\dot{b}_0=v$ and $g\in\Gamma$ such
that $g\cdot b=b_0$. By definition, the horizontal lift $g_t$ of $b_t$
through $g$ is the flow of~\eqref{eq:vectorfieldyt} with $g_0 = g$
(recall that it exists because $y_t$ is $G^c$-invariant). Hence
\[
\theta_\AA^\perp(g,v) = \frac{d}{dt}_{|t=0}g_tg^{-1}g = (R_g)_*(I\zeta_I({v})).
\]
To check~\eqref{eq:curvaturecY}, given $y \in \LieGamma$ we denote by
$\XX_y$ the associated left-invariant vector field on $\Gamma$,
given by
\begin{equation}
\label{eq:leftinvariantyen}
\XX_{y|g} \defeq (L_g)_*y.
\end{equation}
Since $\cbI$ is right invariant, $[\XX_y,\cbI\cdot] =
\cbI[\XX_y,\cdot]$ for any $y\in \Lie\Gamma$, which implies that
$$
[\cbI \XX_{y_0},\cbI \XX_{y_1}]_{|1} = - [y_0,y_1]_\Gamma
$$
for any $y_0,y_1 \in \Lie\Gamma$, by
Proposition~\ref{prop:Gammacomplex}. Note also that
$$
\theta_\AA^\perp(g,v) = \cbI(L_g)_*(\Ad(g^{-1})\zeta_I(v)) = \cbI\XX_{\Ad(g^{-1})\zeta_I(v)|g} = ((R_g)_*(\cbI\XX_{\zeta_I(v)})_{|g},
$$
for any $g \in \cY$ and $v \in T_{gb}B_I$. Hence given $v_0, v_1 \in
T_{gb}B_I$,
\begin{align*}
F_\AA(v_0,v_1) & = -\theta_\AA (R_g)_*[\cbI\XX_{\zeta_I(v_0)},\cbI\XX_{\zeta_I(v_1)}]_{|1}\\
& = \theta_\AA (R_g)_* [\zeta_I(v_0),\zeta_I(v_1)]_\Gamma = (R_g)_* [\zeta_I(v_0),\zeta_I(v_1)]_\Gamma,
\end{align*}
where the first equality follows
from~\eqref{eq:cov-derivative-commutators} and the third because
\begin{equation*}
(R_g)_* [\zeta_I(v_0),\zeta_I(v_1)]_\Gamma =
(L_g)_*\Ad(g^{-1})[\zeta_I(v_0),\zeta_I(v_1)]_\Gamma
\end{equation*}
is a vertical vector field on $\cY$.
\end{proof}

Given $b,b'\in B_I$, $b'=g\cdot b$ for any $g$ in the fibre of
$\cY_{b,I}$ over $b'$, by Proposition~\ref{prop:4-tuples-C1-4}. Then
we have an isomorphism of principal bundles
\begin{equation}\label{eq:pfb-Y-dependence}
\begin{gathered}
\xymatrix @R=0ex @C=2.5ex{
**[l] \cY_{b,I}\ar[r]^-{\cong}&\cY_{b',I}\\
**[l] g'\ar@{|->}[r] & **[r] g'g^{-1},
}\end{gathered}\end{equation}
with corresponding isomorphism $\Gamma_b\lra{\cong}\Gamma_{b'} \colon
g'\longmapsto \Ad(g)g'$ between their structure groups. It follows
from the definition of the canonical connection in terms
of~\eqref{eq:vectorfieldyt}, or from~\eqref{eq:connectioncY}, that
this isomorphism exchanges the canonical connections on these
principal
bundles.

We are now in a position to construct the promised canonical structure
of `$\Gamma$-invariant symmetric space fibration' on $\pi_\cZ\colon
\cC \to \cZ$. Observe first that the
connection~\eqref{eq:connectioncY} induces a canonical affine
connection
\begin{equation}
\label{eq:connectionBI}
  \nabla\colon \Omega^0_B(TB) \lto \Omega^1_B(TB)
\end{equation}
on $B_I$, obtained using the canonical isomorphism
\begin{equation}
\label{eq:TBIassociated}
TB_I \cong \cY \times_{\Gamma_b}\LieX_b \subset \ad \cY,
\end{equation}
which follows from the canonical isomorphism $TB_I \cong
\pi^*TB_I/\Gamma_b$ and the $\Gamma_b$-equivariant isomorphism of vector
bundles
\begin{equation}\label{eq:TBIpullback}
\begin{gathered}
\xymatrix @R=0ex @C=-5.5ex{
**[l]\cY \times  \Lie \cX_b \ar[r]^-{\cong} & **[r]\pi^*T B_I\\
**[l](g,\zeta)\ar@{|->}[r] & **[r] \(g,Y_{I\Ad(g)\zeta|g\cdot b}\).
}\end{gathered}
\end{equation}
%
Note also that the parallel transport $\tau_t(v)$ of a tangent vector
$v\in T_{b_0}B_I$ along a curve $b_t$ on $B_I$, and hence the affine
connection $\nabla$, do not depend on the choice of the base point
$b\in\cB_I$ used implicitly in the right-hand side
of~\eqref{eq:TBIassociated}. In fact,
it is given by the curve on $TB_I$ defined as
\begin{equation}
\label{eq:parallel-trans-canonical-connection}
\tau_t(v)=Y_{I\zeta_t|b_t},
\text{ where }
\zeta_t\defeq \Ad(g_t)\zeta_I(v).
\end{equation}
Here, $g_t$ is the flow of~\eqref{eq:vectorfieldyt} with $g_0=1$. This follows
from~\eqref{eq:TBIassociated},~\eqref{eq:TBIpullback} and standard properties
about horizontal lifts~\cite[p. 114]{KNI}.

Note that the canonical connections~\eqref{eq:connectioncY}
and~\eqref{eq:connectionBI} are constructed exactly as for any
finite-dimensional symmetric space (cf. e.g.~\cite[Ch XI, \S 3]{KNII}) and
that they are exchanged by the $\Gamma$-actions. In fact, our next result
shows that $(B_I,\nabla)$ is a symmetric space, in a similar sense to~\cite[\S
4, Proposition~2]{D6}.

\begin{theorem}
\label{thm:symm-spc}
Let $I\in\cZ$ be such that $B_I$ is non-empty. Then $B_I$ is a symmetric
space, i.e. it has a torsion-free affine connection $\nabla$, with
holonomy group contained in $\cX_b$ and covariantly constant curvature
$R_\nabla$, given by
\begin{equation}
\label{eq:curvatureBI}
\zeta_I(R_\nabla(v_0,v_1)v_2) =
[[\zeta_I(v_0),\zeta_I(v_1)]_{\Gamma},\zeta_I(v_2)]_{\Gamma},
\end{equation}
for any $b \in B_I$ and $v_0,v_1,v_2 \in T_bB_I$.
\end{theorem}

\begin{proof}
To prove this, we relate the torsion $T_\nabla$ of $\nabla$ with the Nijenhuis
tensor $N_\cbI$ of $(\Gamma,\cbI)$ and its curvature $R_\nabla$
with the curvature $F_\AA$ of $\AA$.

Let $V_1$ and $V_2$ be two vector fields on $B_I$. Then
\[
  T_\nabla(V_1,V_2) \defeq \nabla_{V_1} V_2 - \nabla_{V_2} V_2 - [V_1,V_2].
\]
Consider the principal $\Gamma_b$-bundle $\pi\colon \cY \to B_I$
associated to a fixed $b \in B_I$. By~\eqref{eq:TBIassociated}, $TB_I$
is a subbundle of $\ad \cY$, so $V_j$ induces a $\Gamma_b$-invariant
vertical vector fields $\hat V_j$ on $\cY$, given by
\[
\hat V_j(g) = (R_g)_*\zeta_I(V_j(gb)), 
\]
for $g\in\cY$, $j = 0,1$. We claim that
\begin{equation}
\label{eq:TorsionNI}
T_\nabla(V_1,V_2) = - d\pi(N_\cbI(\hat V_1,\hat V_2)),
\end{equation}
and so $T_\nabla = 0$ by Proposition~\ref{prop:Gammacomplex}. To see this,
note first that
\[
\cbI\hat V_j = \theta_\AA^\perp V_j \quad \text{and} \quad F_\AA(V_1,V_2) = - [\hat V_1, \hat V_2],
\]
by~\eqref{eq:connectioncY} and~\eqref{eq:curvaturecY}. Moreover, by the
construction of $\nabla$ and the definition of the covariant
derivative $d_\AA$ induced by $\AA$ on $\ad \cY$
(see~\eqref{eq:cov-derivative-commutators}),
\[
\widehat{\nabla V_j} = d_\AA\hat V_j \defeq [\theta_\AA^\perp(\cdot),\hat V_j] = [\cbI\widehat{(\cdot)},\hat V_j].
\]
It follows then that
\begin{align*}
N_\cbI(\hat V_1,\hat V_2) & \defeq [\cbI\hat V_1,\cbI\hat V_2] - \cbI[\cbI \hat V_1,\hat V_2] - \cbI[\hat V_1,\cbI \hat V_2] -  [\hat V_1,\hat V_2]\\
& = [\theta_\AA^\perp V_1,\theta_\AA^\perp V_2] - \cbI\widehat{\nabla_{V_1}V_2} + \cbI\widehat{\nabla_{V_2}V_1} + F_\AA(V_1,V_2)\\
& = \theta_\AA^\perp([V_1,V_2] - \nabla_{V_1}V_2 + \nabla_{V_2}V_1)\\
& = - \theta_\AA^\perp T_\nabla(V_1,V_2),
\end{align*}
and so~\eqref{eq:TorsionNI} holds.

Since the curvature $R_\nabla$ is induced by $F_{\mathbb{A}}$ via the
adjoint representation, it follows from~\eqref{eq:curvaturecY},
\eqref{eq:TBIassociated} and the fact that $\LieX_b\subset\LieGamma$
is a Lie subalgebra, that
\begin{equation}
  \label{eq:CurvatureNI}
R_\nabla(v_0,v_1)v_2 = Y_{I[[\zeta_I(v_0),\zeta_I(v_1)]_{\Gamma},\zeta_I(v_2)]_{\Gamma}|b},
\end{equation}
for $v_0,v_1,v_2 \in T_b B_I$, which implies~\eqref{eq:curvatureBI},
by condition ($\star$). Hence, since the group $\cX_b$ is normal in
$\Gamma_b$ and $B_I$ is contractible, it follows
from~\eqref{eq:CurvatureNI} that the holonomy group of $\nabla$ is
contained in $\cX_b$ (see~\cite[Theorem~8.1]{KNI}).
Using~\eqref{eq:CurvatureNI} and the
formula~\eqref{eq:parallel-trans-canonical-connection} for the
parallel transport $\tau_t$ of a curve on $B_I$, it is now
straighforward that $\tau_t^*R_\nabla=R_\nabla$, so $\nabla R_\nabla =
0$.
\end{proof}

\begin{remark}
When $H^1(X,\RR) = 0$, so $\Lie \cX_b = \LieGamma_b$, it follows from Proposition~\ref{prop:4-tuples-C1-4}(1) that the bundle $\cY$, endowed
with the restriction of the formally integrable almost complex structure of
Proposition~\ref{prop:Gammacomplex}, is an infinitesimal complexification of
$\Gamma_b$ in the sense of Donaldson~\cite[\S 4]{D6}.
If in addition $\Gamma_I$ is trivial, then there is an alternative
proof of Theorem~\ref{thm:symm-spc} which does not use
Proposition~\ref{prop:Gammacomplex}. In this case, the almost complex structure on
$\cY_{b,I}$ can be defined as the pull-back of the formally integrable
almost complex structure on $Z_b$ by the holomorphic map
\begin{equation}\label{eq:cxaction}
\begin{gathered}
\xymatrix @R=0ex @C=2ex{
**[l]\cY_{b,I} \ar[r] & **[r] Z_b\\
**[l]g\ar@{|->}[r] & **[r] g^{-1}I.
}
\end{gathered}
\end{equation}
\end{remark}

\subsection{The uniqueness and existence problem for the coupled equations}
\label{sub:ANStability}

We apply now the framework of~\secrefs{sub:concrete-setup},
\ref{sub:ANsymm-spc} to construct obstructions to the existence of
solutions to the coupled equations~\eqref{eq:CYMeq}.

Fix coupling constants $\alpha_0, \alpha_1>0$. It follows from
Proposition~\ref{mm-equations}, Lemma~\ref{lem:pairsinvacs2} and the
construction of $\mu_b$ in Theorem~\ref{prop:Kahlerfibration} for each
$b=(\omega,H)\in\cB$, that the existence of a solution $(J,A)\in\cP_b$
of the coupled equations~\eqref{eq:CYMeq} (for the symplectic manifold
$(X,\omega)$ and the principal $G$-bundle $E_H$) is equivalent to the
condition $\mu_b(I)=0$ for some $I\in Z_b$. By the
equivariance~\eqref{eq:mmapBequiv} of $\mu_\cC$, this is equivalent to
the condition
\begin{equation}
\label{eq:slogan}
\pi_\cZ^{-1}(\Gamma\cdot I) \cap \mu_\cC^{-1}(0) \neq \emptyset,
\end{equation}
where $\Gamma\cdot I\subset \cZ$ is the orbit of $I$. Given such an orbit,
in~\secref{sub:ANStability} we construct a complex character $\cF_I$
of the complex Lie algebra $\LieGamma_I$, which vanishes
when~\eqref{eq:slogan} is satisfied, and an `integral of the moment
map' $\cM_I\colon B_I\to \RR$, which is bounded from below
when~\eqref{eq:slogan} is satisfied, provided that the symmetric space
$B_I$ is geodesically convex. Furthermore, we motivate a definition of
`geodesic stability' of the orbit $\Gamma\cdot I$ and conjecture a
link with~\eqref{eq:slogan} when $\Gamma_I$ is finite.

We first reformulate condition~\eqref{eq:slogan} in terms of a
$\Gamma$-invariant family $\sigma$ of $1$-forms $\sigma_I$ on the
fibres $B_I$ of $\pi_\cZ\colon \cC\to \cZ$, defined by the formula
\begin{equation}
\label{eq:sigmaC}
\sigma_I(v) \defeq - \langle\mu_b(I),\zeta_I(v)\rangle,
\end{equation}
for all $(b,I) \in \cC$, $v \in T_bB_I$, with $\zeta_I$ defined as
in~\eqref{eq:vectorfield-zeta}. Here, the $\Gamma$-invariance of
$\sigma$ means
\begin{equation}
\label{eq:Gamma-invariance-sigma}
  \sigma_{g\cdot I}(gv)= \sigma_I(v)
\end{equation}
for all $(b,I) \in \cC$, $v \in T_bB_I$, $g\in\Gamma$. Note
that~\eqref{eq:Gamma-invariance-sigma} follows
from~\eqref{eq:mmapBequiv} and the fact that
\[
\Ad(g) \zeta_I(v) = \zeta_{g I}(v)
\]
for all $g \in \Gamma$, which is immediate from the definition of
$\zeta_I$. Observe also that
\begin{equation}
\label{eq:equivalent-slogan}
\pi_\cZ^{-1}(\Gamma\cdot I) \cap \mu_\cC^{-1}(0) \neq \emptyset \text{ }
\Longleftrightarrow
\text{ $\sigma_I \in \Omega^1(B_I)$ has a zero.}
\end{equation}

Now, since $B_I$ is contractible (see~\eqref{eq:B-product-K-R}), it
suffices to study $\sigma_I$ along curves on $B_I$. Let $V_t$ be a
vector field on $B_I$ along a curve $b_t$ on $B_I$, i.e. a curve on
$TB_I$ with $V_t\in T_{b_t}B_I$ for all $t$. We use the standard
notation $\nabla_{\dot b_t}V_t$ for the covariant derivative of $V_t$
in the direction of $\dot{b}_t$ on the symmetric space $(B_I,\nabla)$
(see~\eqref{eq:connectionBI} and Theorem~\ref{thm:symm-spc}).

\begin{proposition}
\label{propo:sigmaIalongcurve}
\hspace{0.3cm}
\begin{enumerate}
  \item[\textup{(1)}] $\frac{d}{dt}\sigma_I(V_t) = \omega_{b_t}(Y_{\zeta_I(V_t)|I},\mathbf{I}Y_{\zeta_I(\dot b_t)|I}) + \sigma_I(\nabla_{\dot b_t}V_t)$.
  \item[\textup{(2)}]
    $\sigma_I$ is closed.
  \end{enumerate}
\end{proposition}

\begin{proof}
To prove (1), let $g_t$ the horizontal lift of $b_t$ to $\cY_{b_0,I}$
prescribed by the connection~\eqref{eq:connectioncY}, with
$g_0=1$. Then $b_t = g_t\cdot b$ (see Proposition~\ref{prop:4-tuples-C1-4}),
so~\eqref{eq:mmapBequiv} implies
\begin{equation}
\label{eq:propo:sigmaIalongcurve0}
\sigma_I(V_t) = -\langle\mu_b(I_t),\zeta_t\rangle,
\end{equation}
where $I_t \defeq g_t^{-1}\cdot I$ and $\zeta_t \defeq
\Ad(g_t)^{-1}\zeta_I(V_t)$. Using~\eqref{eq:compatible-Y-imag_I}, we
obtain
\[
\dot I_t = - g_t^{-1}\dot g_t g_t^{-1}I = - g_t^{-1}
Y_{I\zeta_I(\dot b_t)|I} = - g_t^{-1} \mathbf{I}Y_{\zeta_I(\dot b_t)|I},
\]
so using formula~\eqref{eq:parallel-trans-canonical-connection} for the
parallel transport $\tau_{t,s}\colon T_{b_t}B_I \to T_{b_s}B_I$ and
the definition of covariant derivative (see e.g.~\cite[p. 114]{KNI}),
\begin{align*}
\nabla_{\dot b_t}V_t \defeq & \frac{d}{ds}_{|s = t}\tau_{t,s}^{-1}(V_t) =
\frac{d}{ds}_{|s = t} Y_{I \Ad(g_s g_t^{-1})^{-1}\zeta_I(V_t)|b_t}\\
= & Y_{I\Ad(g_t)\dot\zeta_t|b_t} = g_t Y_{I_t \dot\zeta_t|b}.
\end{align*}
Formula (1) follows now from this equation and the $\Gamma$-invariance
of $\sigma_I$, as they imply $\sigma_{I_t}(Y_{I_t \dot\zeta_t|b}) =
\sigma_I(\nabla_{\dot b_t}V_t)$, that combined with
~\eqref{eq:propo:sigmaIalongcurve0} imply
\begin{align*}
\frac{d}{dt}\sigma_I(V_t)  =& - \langle d\mu_b(\dot I_t),\zeta_t)\rangle
- \langle\mu_b(I_t),\dot \zeta_t\rangle\\
= & \omega_{b}(g_t^{-1}Y_{\zeta_I(V_t)|I},g_t^{-1}\mathbf{I}Y_{\zeta_I(\dot b_t)|I}) + \sigma_{I_t}(Y_{I_t \dot\zeta_t|b})\\
= & \omega_{b_t}(Y_{\zeta_I(V_t)|I},\mathbf{I}Y_{\zeta_I(\dot b_t)|I}) +
\sigma_I(\nabla_{\dot b_t}V_t),
\end{align*}
since $\mu_b$ is a moment map and $\omega_\cC$ is $\Gamma$-invariant.

To prove (2), let $V_1$ and $V_2$ be two vector fields on $B_I$. Then
\begin{equation}
\label{eq:propo:sigmaIalongcurve2}
d \sigma_I(V_1,V_2) = V_1(\sigma_I(V_2)) - V_2(\sigma_I(V_1)) -
\sigma_I([V_1,V_2]),
\end{equation}
so, using~(1) and the fact that $\omega_\cB(\cdot,\mathbf{I}\cdot)$ is
a family of symmetric bilinear forms, we see that
\[
d \sigma_I(V_1,V_2) = \sigma_I(T_\nabla(V_1,V_2)),
\]
which vanishes because $\nabla$ is torsion-free, by Theorem~\ref{thm:symm-spc}.
\end{proof}

To define our first obstruction to~\eqref{eq:slogan}, note that
$\LieGamma_I$ is a complex Lie algebra for all $I \in \cZ$,
by~\eqref{eq:compatible-Y-imag_I} and the equivariance
of~\eqref{eq:imag_ICeq}. Given $I \in \cZ$ and $b\in B_I$, combining
the 1-form $\sigma_I$ and~\eqref{eq:imag_ICeq}, we obtain a
$\CC$-linear map
\begin{equation}
\label{eq:Futakiformal}
\begin{gathered}
\xymatrix @R=0ex @C=-13ex{
**[l]\cF_I\colon\;\LieGamma_I\ar[r] & **[r] \CC\\ **[l]
\zeta\ar@{|->}[r] & **[r] \langle\cF_I,\zeta\rangle\defeq
\imag\sigma_I(Y_{\zeta|b})+\sigma_I(Y_{I \zeta|b}).
}
\end{gathered}\end{equation}
By the $\Gamma$-invariance of $\sigma$
(see~\eqref{eq:Gamma-invariance-sigma}), this map is also
$\Gamma$-invariant, i.e.
\begin{equation}
\label{eq:Futakiformal-invariant}
\langle\cF_{g\cdot I},\Ad(g)\zeta\rangle=\langle\cF_I,\zeta\rangle,
\end{equation}
for all $\zeta\in\LieGamma_I$, $g\in\Gamma$.

\begin{theorem}
\label{thm:futakiobstructionformal}
The map~\eqref{eq:Futakiformal} is independent of $b \in B_I$. It defines a
character
\[
\cF_I \colon \Lie \Gamma_I \lto \CC
\]
of $\Lie \Gamma_I$ that vanishes if $\sigma_I$ has a zero.
\end{theorem}

\begin{proof}
The proof essentially follows a previous one by Bourguignon~\cite{Bo}.
For the first part, it is enough to prove that $\sigma_I(Y_{\zeta})$
is a constant function on $B_I$, for all $\zeta \in \Lie
\Gamma_I$. Now, $\sigma_I\in \Omega^1(B_I)$ is closed (by
Proposition~\eqref{propo:sigmaIalongcurve}) and $\Gamma_I$-invariant
(since $\sigma$ is $\Gamma$-invariant), so
\[
d(\sigma_I(Y_{\zeta})) = -Y_{\zeta} \lrcorner d\sigma_I +
L_{Y_{\zeta}} \sigma_I = 0,
\]
and hence $\sigma_I(Y_{\zeta})$ is constant, because $B_I$ is
contractible. The second part follows because $\cF_I$ is $\CC$-linear and
$\Gamma_I$-invariant, by~\eqref{eq:Futakiformal-invariant}.
\end{proof}

To obtain the second obstruction, note that, by
Proposition~\ref{propo:sigmaIalongcurve} and the contractibility of $B_I$,
$\sigma_I$ is exact and so there exists a functional
\begin{equation}
\label{eq:integralMmap}
\cM_I \colon B_I\times B_I\to\RR
\end{equation}
such that $d \cM_I(\cdot,b) = \sigma_I$ and $\cM_I(b,b) = 0$ for all
$b\in B_I$. Along a curve $b_t$ on $B_I$,
\begin{equation}
\label{eq:integralMmap-along-curve}
  \cM_I(b_t,b) = \cM_I(b_0,b) + \int_0^t \sigma_I(\dot b_s) ds.
\end{equation}
Moreover, the $\Gamma$-invariance of $\sigma$ implies that
\begin{equation}
\label{eq:Mcocycle}
\cM_I(gb',b) = \cM_{g^{-1}I}(b',b) + \cM_I(b',b),
\end{equation}
for all $g \in \Gamma$ such that $gb' \in B_I$ (i.e. $g \in \cY_{b',I}$).

\begin{proposition}
\label{propo:ANuniqueness}
The functional $\cM_I(\cdot,b)\colon B_I \to \RR$ is convex along geodesics on
$(B_I,\nabla)$. If $B_I$ is geodesically convex and $\sigma_I$ has a zero, then
$\cM_I(\cdot,b)$ is bounded from below, for all $b\in B_I$.
\end{proposition}

\begin{proof}
The first part follows because~\eqref{eq:integralMmap-along-curve} and
Proposition~\ref{propo:sigmaIalongcurve}(1) imply
\begin{equation}
\label{eq:sigmaIincreasing}
\frac{d^2}{dt^2}\cM_I(b_t,b)=\frac{d}{dt}\sigma_I(\dot b_t) = \left\|Y_{\zeta_I(\dot b_t)|I}\right\|^2
\geq 0,
\end{equation}
for any geodesic $b_t$ on $(B_I,\nabla)$, where $\|\cdot\|$ is the $L^2$-norm
with respect to the metric on $Z_{b_t}$.

For the second part, suppose $b' \in B_I$ is a zero of $\sigma_I$. We can
suppose $b' = b$, because using~\eqref{eq:integralMmap-along-curve} along a
curve joining $b$ and $b'$, we see that
\[
  \cM_I(\cdot,b')=\cM_I(b,b')+\cM_I(\cdot,b).
\]
Now, given $b'' \in B_I$, by hypothesis there exists a geodesic $b_t$ with
$b_0 = b$ and $b_1 = b''$. Hence
\[
\cM_I(b'',b)
= \int_0^1\int_0^t\left\|Y_{\zeta_I(\dot b_s)|I}\right\|^2 ds \wedge dt \geq 0,
\]
and so $\cM_I(\cdot,b)$ is bounded from below by $0 \in \RR$.
\end{proof}

\begin{corollary}
\label{cor:ANuniqueness2}
If $B_I$ is geodesically convex, then $\sigma_I$ has at most one zero on $B_I$
modulo the action of $\Gamma_I$.
\end{corollary}

\begin{proof}
Given zeros $b,b' \in B_I$ of $\sigma_I$, let $b_t$ a geodesic joining
them. Then
\[
\left\|Y_{\zeta_I(\dot b_t)|I}\right\|^2 = 0
\]
for all $t$, because~\eqref{eq:sigmaIincreasing} implies that
\begin{equation}
\label{eq:increasing-function}
\begin{gathered}
\xymatrix @R=0ex @C=3ex{
**[l]\RR\ar[r] & **[r]\RR\\
t\ar@{|->}[r] & **[r] \sigma(\dot{b}_t)
}
\end{gathered}
\end{equation}
is an increasing function which vanishes for $t = 0$ and $t=1$. Hence the flow
$g_t$ of $I\zeta_I(\dot b_t)$ lies in $\Gamma_I$ for all $t$ and $g_t b =
b_t$. In particular, $g_1b = b'$.
\end{proof}

\begin{remark}
\label{rem:sign-alpha-sigma-F}
Proposition~\ref{propo:sigmaIalongcurve} and
Theorem~\ref{thm:futakiobstructionformal} hold even when $\alpha_0, \alpha_1$
are not positive (their proofs depend only on the condition that $\omega_b$ is
of type (1,1) with respect to $\mathbf{I}$). In~\secref{chap:Deformation}, we
will use these facts about $\sigma_I$ and $\cF_I$ for arbitrary $\alpha_0,
\alpha_1$. However, Proposition~\ref{propo:ANuniqueness},
Corollary~\ref{cor:ANuniqueness2} and the remainder
of~\secref{sub:ANStability} depend on the assumption that $\alpha_0, \alpha_1$
are positive, although Proposition~\ref{propo:HKpartialConv} also holds in the
degenerate case $\alpha_0\alpha_1=0$.
\end{remark}

If $Z_b$ and $\Gamma_b$ are finite-dimensional manifolds and $\cX_b=\Gamma_b$
is compact, there is a well-known numerical condition, called the
Hilbert--Mumford criterion, which characterizes~\eqref{eq:slogan} (see the
example at the end of~\secref{sub:ANStability}).  In this case, the principal
bundle $\cY$ of Proposition~\ref{prop:4-tuples-C1-4} is the complexification
of $\Gamma_b$ (by the observations about infinitesimal complexifications at
the end of~\secref{sub:ANsymm-spc}, as formally integrable almost complex
structures are integrable in finite dimensions), and the criterion is
formulated in terms of 1-parameter subgroups of $\cY$. In the generality
of~\secref{chap:analytic}, the Lie group $\Gamma_b$ may have no
complexification, but the geodesics of the symmetric space $(B_I,\nabla)$ are
a substitute for the 1-parameter subgroups, and we have the following
generalization of this condition (cf.~\cite{Ch2}, \cite[\S 8]{D6}).

\begin{definition}
\label{def:stability}
A point $I \in \cZ$ is \emph{geodesically semistable} if
\begin{equation}
\label{eq:stability}
\lim_{t \to\infty}\sigma_I(\dot b_t) \geq 0
\end{equation}
for any infinite geodesic ray $b_t$, $t\in [0,\infty[$, in $(B_I,\nabla)$. It is
\emph{geodesically stable} if the inequality~\eqref{eq:stability} is strict
whenever $b_t$ is non-constant.
\end{definition}

Observe that the limit~\eqref{eq:stability} always exists,
because~\eqref{eq:increasing-function} is an increasing function for
geodesic rays, by~\eqref{eq:sigmaIincreasing}. Note also that the geodesic
stability and semistability conditions only depend on the $\Gamma$-orbit of $I
\in \cZ$, because $\sigma$ is $\Gamma$-invariant and the connections on the
fibres of $\pi_\cZ$ are exchanged by the $\Gamma$-action and hence so are
their geodesic rays.

In the finite-dimensional case, by the Kempf--Ness Theorem~\cite{KN}, an orbit
$\Gamma\cdot I\in\cZ$ is geodesically stable if and only if~\eqref{eq:slogan}
holds and $\Gamma_I$ is finite (see the example at the end
of~\secref{sub:ANStability}). The following result provides some evidence that
a sensible question is whether this equivalence also holds in the generality
of~\secref{chap:analytic}, at least when $B_I$ is geodesically convex.

\begin{proposition}
\label{propo:HKpartialConv}
Let $(b,I) \in \cC$. Then
\begin{enumerate}
\item[\textup{(1)}]
  If $\Gamma_{b,I} \defeq \Gamma_b \cap \Gamma_I$ is not finite,
  then $\Gamma \cdot I$ is not geodesically stable.
\item[\textup{(2)}]
  Suppose that $B_I$ is geodesically convex. If~\eqref{eq:slogan} is satisfied,
  then $\Gamma \cdot I$ is geodesically semistable.
\end{enumerate}
\end{proposition}

\begin{proof}
For part (1), let $\zeta \in \LieGamma_{I}$ be non-zero. Let $g_t$ the flow
of $I\zeta$. Then $b_t = g_tb$ is an infinite geodesic ray starting at $b$, because
\[
\dot b_t = g_t Y_{I\zeta|b} = Y_{\Ad(g_t)I\zeta|b_t} = Y_{I\Ad(g_t)\zeta|b_t}
= \tau_t \dot b_0.
\]
Furthermore, if $\zeta \in \LieGamma_{b,I}$, then $\dot b_0 =
Y_{I\zeta|b}\neq 0$, by~\eqref{eq:infactcB-2}, so $b_t$ is non-constant.
Then
\[
\sigma_I(\dot b_t) = \sigma_I(g_t\dot b_0) = \sigma_{g_t^{-1}}(\dot b_0) = \sigma_I(\dot b_0)
\]
and so
\[
\lim_{t \to \infty} \sigma_I(\dot b_t) = \sigma_I(\dot b_0) =
\sigma_I(Y_{I\zeta|b}) = \langle\cF_I,\zeta\rangle.
\]
There are three possibilities. If $\langle\cF_I,\zeta\rangle < 0$, then part (1) is
obvious. The case $\langle\cF_I,\zeta\rangle > 0$ reduces to the previous one by taking the
non-trivial geodesic corresponding to $-I\zeta$. Finally, if $\langle\cF_I,\zeta\rangle =
0$, since $b_t$ is non-trivial, then by definition $I$ is not geodesically
stable.

For part (2), suppose that $B_I$ is geodesically convex and $\Gamma \cdot I$
is not geodesically semistable. Then there exists an infinite geodesic ray
$b_t$ such that
\[
C\defeq \lim_{t \to\infty}\sigma_I(\dot b_t) < 0,
\]
where $\sigma_I(\dot{\sigma}_t) \leq C$ for all $t$,
as~\eqref{eq:increasing-function} is an increasing function, so
$\cM_I(b_t,b_0) \leq Ct$, by~\eqref{eq:integralMmap-along-curve}.  Therefore
$\cM_I(\cdot,b_0)$ is not bounded from below, so~\eqref{eq:slogan} cannot be
satisfied, by Proposition~\ref{propo:ANuniqueness}.
\end{proof}

We would like to point out that the framework developed
in~\secrefs{sub:ANsymm-spc}, \ref{sub:ANStability} is rather general, as it
relies only on formal properties of the double
fibration~\eqref{eq:fibration-compatible-pairs}, and may be applied to other
situations (in particular, to equations with a further coupling with Higgs
fields). The basic ingredients are a real Lie group $\Gamma$, a
$\Gamma$-equivariant double fibration~\eqref{eq:fibration-compatible-pairs},
where $(\cZ,\mathbf{I})$ is an almost complex manifold, and a
$\Gamma$-equivariant map~\eqref{eq:imag_ICeq}
satisfying~\eqref{eq:compatible-Y-imag_I}. It is
crucial that $\pi_\cZ$ satisfies condition ($\star$)
of~\secref{sub:ANsymm-spc}, all its fibres are contractible and $\pi_\cB$
satisfies the properties of Theorem~\ref{prop:Kahlerfibration} for a fibration
of normal subgroups as in~\eqref{eq:Lie-group-bundles} (note that the formal
integrability of the almost complex structures on the fibres of $\pi_\cB$ was
never used).

To see how this general framework works, we
conclude~\secref{sub:ANStability} by explaining how it applies to the
standard theory of finite-dimensional K\"ahler quotients (as
presented e.g. in~\cite[\S 5]{MR}) and its relation with Geomeric
Invariant Theory (GIT). Suppose that $\cZ$ is a finite-dimensional
K\"ahler manifold with a left action of a complex reductive Lie
group $G^c$ preserving its complex structure. Suppose also that this
action restricts to a Hamiltonian action of a maximal compact subgroup
$G\subset G^c$, with $G$-equivariant moment map
\[
  \mu\colon \cZ\lto\mathfrak{g}^*,
\]
where $\mathfrak{g}$ is the Lie algebra of $G$. To compare
with~\secrefs{sub:ANsymm-spc}, \ref{sub:ANStability}, we define:
\begin{itemize}
\item
  $\Gamma$ is the real Lie group underlying $G^c$,
\item
  $\cB=G^c/G$ is the orbit space for the action by right multiplication of $G$
  on $G^c$,
\item
the map $\cZ\to W$ of~\eqref{eq:imag_ICeq} is the constant map given by the
complex structure on the Lie algebra $\mathfrak{g}^c$ of $G^c$,
\item
$\cC=\cB\times\cZ$ and $\cX_\cB=\Gamma_\cB$.
\end{itemize}
Then the isotropy
group of any $G$-orbit $b=[g]\defeq gG \in\cB$ is
\[
\cX_b=\Gamma_b=\Ad(g)G
\]
and the fibre of the morphism~\eqref{eq:infactcB-2} over a point $b=[g]$ is
\begin{equation}\label{eq:infactcB-2-KN}
\begin{gathered}
\xymatrix @R=0ex @C=-3ex{
**[l]\Ad(g)\mathfrak{g}\ar[r]^{\cong} & **[r]T_gG^c/T_g(gG)\\
\zeta\ar@{|->}[r] & **[r][(R_g)_*(\imag\zeta)].
}
\end{gathered}
\end{equation}
Therefore~\eqref{eq:infactcB-2} is an isomorphism and condition
($\star$) of~\secref{sub:ANsymm-spc} is satisfied, and hence so are
the conclusions of Proposition~\ref{sub:ANsymm-spc} and
Theorem~\ref{sub:ANStability}.
In this finite-dimensional case, the construction of the
connections~\eqref{eq:connectioncY} and~\eqref{eq:connectionBI} reduce
to the classical constructions of the canonical connections on
finite-dimensional symmetric spaces (see e.g.~\cite[Ch XI, \S
3]{KNII}).
Hence, by~\cite[Ch XI, Theorem~3.2(3)]{KNII}, the infinite geodesic
rays on $G^c/G$ starting at $[g]$ are the curves
\begin{equation}\label{eq:geodesic-rays-KN}\begin{gathered}
\xymatrix @R=0ex @C=0ex{
**[l][0,\infty[\ar[r] & **[r]G^c/G\\ **[l]
t\ar@{|->}[r] & **[r][e^{t\imag \zeta}g],
}
\end{gathered}\end{equation}
with $\zeta \in \Ad(g)\mathfrak{g}$.
Note that the canonical projection
\[
  \pi_\cB\colon G^c/G\times\cZ\lto G^c/G
\]
is a `trivial' $G^c$-invariant complex fibration. However, since $G^c$
does not necessarily preserve $\omega_\cZ$, to view $\pi_\cB$ as a
$G^c$-invariant Hamiltonian fibration, we endow this map with the
non-trivial family $\omega_\cC$ of symplectic 2-forms $\omega_b\defeq
g_*\omega_\cZ$ on the fibres $Z_b=\cZ$, for $b=[g]\in G^c/G$. Indeed,
the isotropy group $\Ad(g)G$ preserves $\omega_b$ and has moment map
given by
\begin{equation}
\label{eq:mmapB-KN}
\langle \mu_b(I),\zeta\rangle
\defeq \langle\mu(g^{-1}I), \Ad(g^{-1})\zeta\rangle,
\end{equation}
for $b=[g]\in G^c/G$, and~\eqref{eq:mmapB-KN} defines the morphism~\eqref{eq:mmapB} of fibrations over $G^c/G$. Using
the isomorphism~\eqref{eq:infactcB-2-KN}, we obtain the formula
\[
  \langle\cF_I,\zeta\rangle=-\langle \mu(I), \zeta_0\rangle - \imag \langle \mu(I),
  \zeta_1\rangle,
\]
for all $I\in\cZ$, $\zeta=\zeta_0 + \imag \zeta_1
\in\mathfrak{g}^c_I$, where $\zeta_0,\zeta_1\in\mathfrak{g}^c$. Hence
Theorem~\ref{thm:futakiobstructionformal} reduces
to~\cite[Proposition~6 and Corollary~8]{W}.

Suppose now that $\cZ$ is a $G^c$-linearised projective manifold,
i.e. there is a $G^c$-equivariant closed embedding
$\cZ\subset\CC\PP^N$ and $\omega_\cZ$ is the restriction of the
Fubini--Study K\"ahler form. Then geodesic stability/semistability
coincide with GIT stability/semistability, by the Hilbert--Mumford
criterion. This essentially follows because any 1-parameter subgroup
\[
  \lambda\colon \CC^*\to G^c
\]
restricts to a group homomorphism $\lambda\colon S^1\to \Ad(g)G$ for
some $g\in G^c$, which induces an infinite geodesic
ray~\eqref{eq:geodesic-rays-KN} starting at $[g]$ and because the
Hilbert--Mumford weight for $\lambda$ at a point $I\in\cZ$ is
precisely the left-hand side of~\eqref{eq:stability}.
Furthermore, the functional~\eqref{eq:integralMmap} is the Kempf--Ness
functional~\cite{KN}, which provides the key tool to prove the Kempf--Ness theorem relating
the symplectic and GIT quotients:
\[
  \mu^{-1}(0)/G\cong \cZ/\!\!/G^c.
\]
Finally, we should remark that this theorem has been extended to
non-projective manifolds (see e.g.~\cite[\S 5]{MR}, \cite{Te}). In
this case, the functional~\eqref{eq:integralMmap} is the integral of
the moment map in~\cite[\S 5]{MR} and geodesic stability coincides
with analytic stability (by \cite[Corollary~5.3]{MR}).

\subsection{The $\alpha$-Futaki character, the $\alpha$-K-energy and the geodesic equation}
\label{sub:ANCeq}

We now prove that condition ($\star$) of~\secref{sub:ANsymm-spc} is
satisfied and give explicit formulae for the character $\cF_I$, the
functional $\cM_I$ and the geodesic equation on $B_I$ introduced
in~\secref{sub:ANStability}.

Fix a complex structure on $X$ for which $\Omega\in H^2(X,\RR)$ is a
K\"ahler class (i.e. it contains a K\"{a}hler form) and a holomorphic
structure on the principal $G^c$-bundle $\pi\colon E^c \to X$. These
data determine a point $I\in\cZ$. As explained
in~\secref{sub:ANStability}, condition~\eqref{eq:slogan} for the orbit
$\Gamma\cdot I$ is equivalent to the existence of a pair
$b=(\omega,H)\in B_I$ such that the point $(J,A)\in\cP_b$
corresponding to $I$ via Lemma~\ref{lem:pairsinvacs2} satisfies the
coupled equations~\eqref{eq:CYMeq}. In other words,
condition~\eqref{eq:slogan} for the orbit $\Gamma\cdot I$ is
equivalent to the existence of a solution $b=(\omega,H)\in B_I$ to the
following coupled equations, where $S_\omega$ is the scalar curvature
of the Riemannian metric $g_{\check{I}}=\omega(\cdot,\check{I}\cdot)$
and $F_H$ is the curvature of the Chern connection of $H$ and $I$:
\begin{equation}
\label{eq:CYMeq2} \left. \begin{array}{l}
\Lambda_\omega F_H = z\\
\alpha_0 S_{\omega} \; + \; \alpha_1 \Lambda^2_\omega (F_H \wedge F_H) = c
\end{array}\right \}.
\end{equation}
By~\eqref{eq:equivalent-slogan}, these equations are satisfied if and only if the 1-form
$\sigma_I$ on $B_I$ has a zero. Now, the definition of $\sigma_I$
in~\eqref{eq:sigmaC}, and in fact the whole
of~\secrefs{sub:ANsymm-spc}, \ref{sub:ANStability}, depend on
condition ($\star$) of~\secref{sub:ANsymm-spc}.
To prove this condition, note first that by~\eqref{eq:B-product-K-R},
there is a canonical isomorphism
\begin{equation}
\label{eq:symmetricspace0}
T_b B_I \cong dd^cC^{\infty}(X) \oplus \imag\Omega^0(\ad E_H),
\end{equation}
for all $b=(\omega,H)\in B_I$, obtained from the $dd^c$-lemma and from the pointwise isomorphism
$\imag\Lie G\cong G^c/G$ induced by the exponential.
Define now $\LieG^c=\Omega^0(\ad E^c)$ and $\LieG_H=\Omega^0(\ad E_H)$ as
the Lie algebras of the gauge group $\cG^c$ of $E^c$ and the gauge
group $\cG_H$ of $E_H$, respectively. Consider the projection maps onto the
real and imaginary parts associated to $H\in\cR$,
\begin{equation}
\label{eq:pi1map}
\Re_H, \Im_H \colon \LieG^c \lto \LieG_H,
\end{equation}
defined by $y=\Re_Hy+\imag\Im_Hy$ for all $y\in\LieG^c$, where we are
using the canonical isomorphism
\[
\LieG^c\cong \LieG_H\oplus \operatorname{\mathbf{i}} \LieG_H
\]
induced by~\eqref{eq:E^c-vs-E_H} and $\mathfrak{g}^c = \mathfrak{g}
\oplus \imag\mathfrak{g}$.

\begin{lemma}
\label{lem:infactR}
  The infinitesimal action of $y\in\Lie(\Aut E^c)$ on $H\in\cR$ is
  \[
   Y_{y|H} = \operatorname{\mathbf{i}} \Im_H (\theta_A y),
  \]
  where $\theta_A\colon TE^c\to VE^c$ is the vertical projection
  induced by any connection $A$ on $E_H$.
\end{lemma}

\begin{proof}
Using the maps $\theta_A, \theta_A^\perp$
in~\eqref{eq:induced-theta}, any $y\in\Lie(\Aut E^c)$ can be
decomposed as
\[
y = \imag \Im_H (\theta_A y) + \Re_H (\theta_A y) + \theta_A^\perp \check y,
\]
where $\check y$ is the vector field on $X$ covered by $y$. Hence the
flow $g_{-t}$ of $-y$ can be written as
\[
g_{-t} = f_{-t} \circ s_t
\]
where $f_t \in \Aut E_H$ is the flow of $\Re_H(\theta_A y) +
\theta^\perp \check y$ and $s_t$ is the flow of the time-dependent
vector field $- f_{t*}(\imag \Im_H(\theta_A y))$. Therefore, using the
isomorphism $T_H\cR\cong\imag\Omega^0(\ad E_H)$ (also used
in~\eqref{eq:symmetricspace0}) and the fact that $f_t^{-1}$ preserves
$H$, we see that the flow $g_t$ of $y$ satisfies
\[
Y_{y|H} = \frac{d}{dt}_{|t=0} g_t \cdot H = \frac{d}{dt}_{|t=0} s_t^{-1} \cdot H = f_{t*}\(\imag \Im_H(\theta_A y)\)_{|t =0} = \imag \Im_H(\theta_A y).
\qedhere
\]
\end{proof}

\begin{proposition}
\label{lem:isom-LieX=TB_I}
Condition \textup{($\star$)} is satisfied. The inverse
of~\eqref{eq:infactcB-2} is given by
\begin{equation}
\label{eq:vectorfieldzeta}
\zeta_I(v) = -\operatorname{\mathbf{i}}\dot{H} - \theta_H^\perp\eta_\phi \in \LieX_b,
\end{equation}
where $v\in T_b B_I$ corresponds to $(dd^c\phi,\dot{H})\in
dd^cC^{\infty}(X) \oplus \operatorname{\mathbf{i}}\Omega^0(\ad E_H)$
via~\eqref{eq:symmetricspace0}.
\end{proposition}

\begin{proof}
Fix $b=(\omega,H)\in B_I$. Given $\zeta \in \LieGamma$ covering a vector field
$\check{\zeta}$ on $X$, we have
\begin{equation}
\label{eq:infactcBiLiebeth}
  Y_{I \zeta|b}  = (-d(\check I \check \zeta \lrcorner
  \omega),\imag\Re_H(\theta_H \zeta))
\end{equation}
by Lemma~\ref{lem:infactR}, as $I\zeta$ covers $\check I \check \zeta$
and $\theta_H \circ I = \imag \theta_H$, where $\theta_H\colon TE^c\to VE^c$ is the
vertical projection in~\eqref{eq:induced-theta} induced by the Chern
connection of $I$ on $E_H$. In particular, when $\zeta\in \LieX_b$,
$\check\zeta=\eta_\phi$ is the Hamiltonian vector field of some $\phi
\in C^\infty(X)$ and~\eqref{eq:infactcBiLiebeth} becomes
\begin{equation}
\label{eq:infactcBiLieX}
Y_{I \zeta|b}= (-dd^c\phi, \imag\theta_H \zeta).
\end{equation}
Hence, by~\eqref{eq:symmetricspace0} the infinitesimal
action~\eqref{eq:infactcBiLieX} is in $T_b B_I$ and so the
morphism~\eqref{eq:infactcB-2} is well
defined. Furthermore,~\eqref{eq:infactcBiLieX} easily implies
that~\eqref{eq:infactcB-2} has an inverse given
by~\eqref{eq:vectorfieldzeta}.
\end{proof}

Using the formula~\eqref{eq:vectorfieldzeta}, the 1-form $\sigma_I$ on
$B_I$ is given by
\begin{equation}
\label{eq:sigmaC2}
\sigma_I(v) = 4 \alpha_1 \int_X
\imag\dot{H}\wedge(\Lambda_\omega F_H - z) \omega^{[n]} + \int_X\phi S_\alpha(b,I) \omega^{[n]},
\end{equation}
for all $v=(dd^c\phi,\dot{H})\in T_bB_I$, where $\phi\omega^{[n]}$ has
zero integral on $X$ and $S_\alpha(b,I)$ is given
by~\eqref{eq:Salpha-bI}.

The complex character $\cF_I\colon \LieGamma_I \to \CC$ defined
by~\eqref{eq:Futakiformal} provides our first obstruction to the
existence of solutions to~\eqref{eq:CYMeq2}. To give an explicit
expression for $\cF_I$, note first that
\[
\LieGamma_I = \Lie \Aut(E^c,I)
\]
is the Lie algebra of the automorphism group of the holomorphic bundle
$(E^c,I)$, so each $\zeta \in \LieGamma_I$ covers a real holomorphic
vector field $\check \zeta$ on $(X,\check I)$. Now, we can write
\[
\check \zeta = \eta_{\phi_1} + \check I \eta_{\phi_2} + \beta,
\]
for any given a K\"ahler form $\omega \in \cK_{\check I}$, where
$\eta_{\phi_j}$ is the Hamiltonian vector field of $\phi_j \in
C^\infty_0(X)$ on $(X,\omega)$, for $j = 1,2$, and $\beta$ is the dual
of a $1$-form which is harmonic with respect to the K\"ahler metric
$\omega(\cdot,\check{I}\cdot)$ (see e.g.~\cite{LS1}). Using this
decomposition in~\eqref{eq:infactcBiLiebeth}, we see that the
infinitesimal action of $\zeta \in \Lie\Aut(E^c,I)$ on $b=(\omega,H)
\in B_I$ is
\[
Y_{\zeta|b}= (-dd^c\phi_2,\imag \Im_H\theta_H\zeta), 
\]
hence defining the complex-valued function $\phi \defeq \phi_1 + \imag
\phi_2$,
\begin{equation}
\label{eq:alphafutakismooth}
\langle\cF_I,\zeta\rangle
= - 4 \alpha_1 \int_X \theta_H\zeta\wedge(\Lambda_\omega F_H - z) \omega^{[n]}
- \int_X \phi S_\alpha(b,I) \omega^{[n]},
\end{equation}
which must vanish if~\eqref{eq:CYMeq2} has a solution, by
Theorem~\ref{thm:futakiobstructionformal}.

It is now clear from formula~\eqref{eq:alphafutakismooth} that for
trivial $G^c$, $\cF_I$ is the Futaki invariant~\cite{Ft0} of the
K\"ahler class $\Omega$ on $(X,\check I)$, up to a multiplicative
factor.
For non-trivial $G^c$ and $\alpha_0 = 0$, the character $\cF_I$,
restricted to the Lie subalgebra of $\LieGamma_I$ consisting of vector
fields covering holomorphic complex Hamiltonian vector fields
(i.e. vector fields that vanish somewhere on $X$), has already been
constructed by Futaki (see \cite[Theorem~1.1]{Ft1}).

Using now~\eqref{eq:integralMmap-along-curve}, the $\alpha$-K-energy
can be written explicitly along a curve $b_t = (\omega_t,H_t)$ on
$B_I$, with $\omega_t = \omega_0 + dd^c\phi_t$ and
$\dot{\phi}_t\omega_t^{[n]}$ with zero integral on $X$, as
\begin{align}
\notag
\cM_I(b_t,b) =
& \cM_I(b_0,b) +  4 \alpha_1 \int_0^t\int_X \imag\dot{H}_s\wedge(\Lambda_{\omega_s} F_{H_s} - z) \omega_s^{[n]} \wedge ds\\
\label{eq:cM-explicit}
& + \int_0^t\int_X\dot{\phi}_s S_\alpha(b_s,I) \omega^{[n]}\wedge ds.
\end{align}
By Proposition~\ref{propo:ANuniqueness}, $\cM_I(\cdot,b)$ is convex
along geodesics on the symmetric space $(B_I,\nabla)$. The explicit
expression of the affine connection $\nabla$ and its geodesic equation
in the coordinates provided by the canonical
isomorphism~\eqref{eq:symmetricspace0} are the content of the
following proposition.

For the next result, given $b=(\omega,H) \in B_I$, we denote by
$(\cdot,\cdot)_\omega$ the metric on $T^*X$ associated to
$\omega(\cdot,\check{I}\cdot)$ and by $d_H$ the covariant derivative
associated to the Chern connection of $H$ and $I$.

\begin{proposition}
\label{thm:symm-spcCeq}
\begin{enumerate}
\item[\textup{(1)}]
The Christoffel symbol $\Gamma\colon T_bB_I\times
T_bB_I\to T_bB_I$ is
\[
\qquad
\Gamma(\dot{b}_1,\dot{b}_2) =
\(-  dd^c(d\phi_1,d\phi_2)_\omega,- \check I \eta_{\phi_1} \lrcorner d_H\dot H_2- \check I \eta_{\phi_2} \lrcorner d_H\dot H_1 + \operatorname{\mathbf{i}} F_H(\eta_{\phi_1},\check I \eta_{\phi_2})\),
\]
for all $\dot b_j = (dd^c \phi_j,\dot H_j) \in T_b B_I$,
with $j = 1,2$.
\item[\textup{(2)}]
A curve $b_t = (\omega_t,H_t)$ on $B_I$, with $\omega_t = \omega +
dd^c\phi_t$, is a geodesic if and only if
\begin{equation}
\label{eq:geodesicequation}
\left. \begin{array}{l}
dd^c(\ddot \phi_t - (d\dot\phi_t,d\dot\phi_t)_{\omega_t}) = 0\\
\ddot H_t - 2\check I \eta_{\dot \phi_t} \lrcorner d_{H_t}\dot H_t +
\operatorname{\mathbf{i}} F_{H_t}(\eta_{\dot \phi_t},\check I \eta_{\dot\phi_t})= 0
\end{array}\right \},
\end{equation}
where $\eta_{\dot \phi_t}$ is the Hamiltonian vector field of
$\dot\phi_t$ over $(X,\omega_t)$.
\end{enumerate}
\end{proposition}

\begin{proof}
The proof of part (1) is a computation of the covariant derivative of
a vector field $v_t = (dd^c\psi_t, \operatorname{\mathbf{i}}\xi_t)$
along a curve $b_t= (\omega_t,H_t)$ on $B_I$, i.e. a curve $v_t$ on
$TB_I$ with $v_t\in T_{b_t}B_I$ for all $t$. Recall that the covariant
derivative of $v_t$ along $b_t$ is (see e.g.~\cite[p. 114]{KNI})
\begin{equation}\label{eq:covariantdervt}
  \nabla_{\dot{b}_t}v_t \defeq \frac{d}{ds}_{|s=t} \tau_{t,s}^{-1}(v_s) = \dot{v}_t + \Gamma(\dot{b}_t,v_t),
\end{equation}
where $\dot{v}_t=(dd^c\dot{\psi}_t, \operatorname{\mathbf{i}}
\dot{\xi}_t)$ and $\tau_{t,s}\colon T_{b_t}B_I \to T_{b_s}B_I$ is the parallel transport along $b_t$. To calculate~\eqref{eq:covariantdervt} we compute the parallel transport $\tau_{0,t}(v)$ of any
$v=(dd^c\phi,\dot H)\in T_bB_I$ along $b_t$
using~\eqref{eq:parallel-trans-canonical-connection}. Let $\zeta_t=
\Ad(g_t)\zeta_I(v)$, where $g_t$ is the horizontal lift of $b_t$ to
$\cY_{b_0,I}$ (i.e. the flow of $I\zeta_I(\dot b_t)$) with $g_0=1$
(see~\eqref{eq:principal-bundle-Y} and
Proposition~\ref{prop:4-tuples-C1-4}). By~\eqref{eq:vectorfieldzeta},
\begin{align*}
\zeta_t & 
= (g_t)_*(-\imag\dot{H} - \theta_{H,I}^\perp\eta_\phi)
= -\imag(g_t)_*\dot{H} - \theta_{H_t,I_t}^\perp((\check g_t)_*\eta_\phi)
= -\imag(g_t)_*\dot{H} - \theta_{H_t,I_t}^\perp\eta_t,
\end{align*}
where $\eta_t$ is the Hamiltonian vector field of $\phi \circ \check
g_t^{-1}$ over $(X,\omega_t)$, so
by~\eqref{eq:parallel-trans-canonical-connection}
and~\eqref{eq:infactcBiLiebeth},
\begin{align}
\tau_{0,t}(\dot{b}) &= Y_{I\zeta_t|b_t}= (dd^c(\phi \circ \check
g_t^{-1}),\imag
\operatorname{Re}_{H_t}\theta_{H_t,I}(-\imag(g_t)_*\dot{H} -
\theta_{H_t,I_t}^\perp\eta_t))
\notag\\
& = (dd^c(\phi \circ \check g_t^{-1}),\operatorname{Im}_{H_t}((g_t)_*\dot{H}) + \eta_t \lrcorner \imag (\theta_{H_t,I_t} - \theta_{H_t,I})).
\label{eq:parallel-trans-explicit}
\end{align}
Hence we obtain
\[
  \tau_{t,s}^{-1}(v_s)=(dd^c(\psi_s \circ \check g_s),\operatorname{Im}_{H_t}(\imag(g_{t,s})^*\xi_s) + \eta_{t,s} \lrcorner \imag (\theta_{H_t,I_s} - \theta_{H_t,I})),
\]
where $v_s = (dd^c\psi_s,\imag\xi_s)\in T_{b_s}B_I$, $g_{t,s}$ is the
flow of $I\zeta_I(\dot b_s)$ with $g_{t,t}=1$, $\eta_{t,s}$ is the
Hamiltonian vector field of $\psi_s \circ \check g_{t,s}$ over
$(X,\omega_t)$ and $I_s = g_{t,s}^{-1}\cdot I$. Thus denoting
$\nabla_{\dot{b}_t} v_t=(w^1_t,w^2_t)$, we conclude that
\begin{align*}
 w^1_t = & dd^c(\dot \psi_t - d\psi_t(\check I \eta_{\dot \phi_t,\omega_t})) = 
  dd^c\dot \psi_t -  dd^c(d\psi_t,d\dot\phi_t)_{\omega_t},
\\
w_t^2 = & \operatorname{Im}_{H_t}(\imag \dot \xi_t + [I \zeta_I(\dot b_t),\imag \xi_t]) - \imag \theta_{H_t,I}[I\zeta_I(\dot b_t),\theta_{H_t,I}^\perp\eta_{\psi_t}]\\
  = & \imag \dot \xi_t + \operatorname{Im}_{H_t}[I \zeta_I(\dot b_t),\imag \xi_t] + \imag \theta_{H_t,I}[\zeta_I(\dot b_t),\theta_{H_t,I}^\perp(\check I \eta_{\psi_t})]\\
= & \dot v^2_t + \operatorname{Im}_{H_t}[\dot H_t,\imag \xi_t] - [\theta_{H_t,I}^\perp(\check I \eta_{\dot \phi_t}),\imag \xi_t] - [\theta_{H_t,I}^\perp(\check I \eta_{\psi_t}),\dot H_t]\\
& - \imag \theta_{H_t,I}[\theta_{H_t,I}^\perp\eta_{\dot \phi_t},\theta_{H_t,I}^\perp(\check I \eta_{\psi_t})]\\
= & \imag\dot\xi_t - \check I \eta_{\dot \phi_t} \lrcorner d_{H_t}(\imag \xi_t)- \check I \eta_{\psi_t} \lrcorner d_{H_t}\dot H_t + \imag F_{H_t}(\eta_{\dot \phi_t},\check I \eta_{\psi_t}).
\end{align*}
This proves (1). Note that $F_H$ and $\omega$ are of type $(1,1)$, so
the torsion is $T_\nabla = 0$ (cf. Theorem~\ref{thm:symm-spc}) and the
geodesic equation is~\eqref{eq:geodesicequation}. This proves (2).
\end{proof}

\begin{remark}
\label{rem:cov-derivative-canonical-connetion2}
When $G^c$ is the trivial group, so $E^c=X$,
Theorem~\ref{thm:symm-spc} and Proposition~\ref{thm:symm-spcCeq}
reduce to the corresponding results for the space of K\"ahler metrics
$\cK_{\check I}$ already studied by Mabuchi~\cite{Mab2,Mab1} and
Donaldson~\cite{D6}. More precisely, we recover the Levi--Civita
connection of the Mabuchi metric on the Riemannian symmetric space
$\cK_{\check I}$, the functional $\cM_I(\cdot,b)$ is the Mabuchi
K-energy~\cite{Mab2,Mab1} on the space of K\"ahler metrics, by
formula~\eqref{eq:cM-explicit}, and~\eqref{eq:geodesicequation}
reduces to the geodesic equation on the space of K\"ahler
metrics~\cite{Mab1}
\begin{equation}
\label{eq:geodesicequationK}
\ddot \phi_t - (d\dot\phi_t,d\dot\phi_t)_{\omega_t} = 0.
\end{equation}
It seems plausible that the methods used by Chen \& Tian~\cite{Ch1,
  ChT} in their study of~\eqref{eq:geodesicequationK} could be adapted
to equation~\eqref{eq:geodesicequation} and to the existence and
uniqueness problem for the coupled equations. As in the case
of~\eqref{eq:geodesicequationK}, this would require a reformulation
of~\eqref{eq:geodesicequation} as a complex Monge--Amp\`ere equation.
\end{remark}

Note that the explicit formula for the Christoffel symbols in
Proposition~\ref{thm:symm-spcCeq} provides a direct proof of the
vanishing of the torsion $T_\nabla$
(cf. Theorem~\ref{thm:symm-spc}). Observe also that the two factors of
$B_I = \cK_{\check I} \times \cR$ are Riemannian symmetric spaces with
holonomy groups contained in $\cH_\omega$ (see~\cite{Mab1}
and~\cite[\S 4]{D6}) and $\cG_H$, and that the holonomy group of $B_I$
is contained in their group extension $\cX_b$
(see~\eqref{eq:coupling-term-moment-map-1}). Here, the structure of Riemannian symmetric space on $\cR$ depends on the choice of an element $\omega \in \cK_{\check I}$. However,
Proposition~\ref{thm:symm-spcCeq} implies that the symmetric space
structure of $B_I$ is not the product structure. In fact, it is an
open question whether $B_I$ carries a Riemannian metric compatible
with $\nabla$ (see Remark~\ref{rem:riemannian-symm-spc-v1} for
details).

\section{Extremal pairs and deformation of solutions}
\label{chap:Deformation}

Following the approach of LeBrun \& Simanca~\cite{LS2, LS1}, in this
section we define an extremality condition for pairs $(\omega,H)$ and
also an extremality condition in the weak coupling limit
(see~\eqref{eq:extremalpair} and~\eqref{eq:extremalpairsplit0}). We
establish existence results for extremal pairs $(\omega,H)$ near
solutions to the coupled equations under deformations of the coupling
constants and the K\"ahler class (Theorems~\ref{thm:DeformationCYMeq1}
and~\ref{thm:DeformationCYMeq3}) and find sufficient conditions for
the existence of solutions to the coupled equations
(Theorems~\ref{thm:DeformationCYMeq2}
and~\ref{thm:DeformationCYMeq4}).

In~\secref{chap:Deformation} we fix a complex reductive Lie group
$G^c$, an $n$-dimensional compact complex manifold $(X,J)$, with
underlying real manifold $X$ and complex structure $J$, and a
holomorphic principal $G^c$-bundle $(E^c,I)$ over $(X,J)$ with
underlying real principal $G^c$-bundle $E^c$ and complex structure $I$
(so $\check{I}=J$ in the notation of~\secref{sub:concrete-setup}). We
also fix a maximal compact Lie subgroup $G\subset G^c$. The Lie
algebras of $G\subset G^c$ are denoted $\mathfrak{g}\subset
\mathfrak{g}^c$, respectively. As in
Theorem~\ref{prop:Kahlerfibration}, we fix a $G^c$-invariant symmetric
bilinear form $(\cdot,\cdot) \colon \mathfrak{g}^c \otimes
\mathfrak{g}^c \to \CC$ which restrics to a $G$-invariant positive
definite inner product on $\mathfrak{g}$. Finally, $\mathfrak{z} =
\mathfrak{g}^G$ and $\mathfrak{z}^c = (\mathfrak{g^c})^{G^c}$ denote
the subsets of elements of $\mathfrak{g}$ and $\mathfrak{g}^c$ which
are invariant under the adjoint actions of $G$ and $G^c$, respectively
(cf.~\eqref{eq:centre-z}).

\subsection{Extremal pairs}
\label{sec:Defextremalholomorphic}

We start studying an extremality condition which will be useful to
prove Theorem~\ref{thm:DeformationCYMeq2}.
Throughout~\secref{sec:Defextremalholomorphic}, we fix
$\alpha=(\alpha_0,\alpha_1)\in\RR^2$ such that $\alpha_1\neq 0$, and a
K\"{a}hler class $\Omega$ on $(X,J)$. Note that we will not assume
$\alpha_0, \alpha_1>0$, but that we can still apply
Proposition~\ref{propo:sigmaIalongcurve} and
Theorem~\ref{thm:futakiobstructionformal} (see
Remark~\ref{rem:sign-alpha-sigma-F}). We define
\[
  B_\Omega\defeq\cK_\Omega\times\cR,
\]
where $\cK_\Omega$ is the space of K\"ahler forms in $\Omega$ compatible with the complex
structure $J$ and $\cR=\Omega^0(E^c/G)$ (cf.~\eqref{eq:B-product-K-R}).

The following definition is closely related to the vanishing condition
for the linearisation at a solution of the coupled equations (see
Proposition~\ref{propo:DefTalpha} and Lemma~\ref{lemma:PTalpha}).

\begin{definition}
\label{def:ext-pair}
A pair $b = (\omega,H)\in B_\Omega$ is \emph{extremal} if it satisfies the equations
\begin{equation}
\label{eq:extremalpair} \left. \begin{array}{l}
4\alpha_1d_H\Lambda_\omega F_H + \eta_{\alpha}(b) \lrcorner F_H = 0\\
L_{\eta_{\alpha}(b)} J = 0
\end{array}\right \},
\end{equation}
where $\eta_{\alpha}(b)$ is the Hamiltonian vector field on
$(X,\omega)$ of the function
\begin{equation}
\label{eq:Salpha}
S_\alpha(b) \defeq -\alpha_0 S_\omega
-\alpha_1\Lambda^2_\omega\(F_H\wedge F_H\) + 4\alpha_1\Lambda_\omega F_H \wedge z_\Omega
\in C^\infty(X),
\end{equation}
(cf.~\eqref{eq:Salpha-bI}). Here $S_\omega$ is the scalar curvature
of the metric $\omega(\cdot,J\cdot)$ and $z_\Omega$ is the element of $\mathfrak{z}\subset\mathfrak{g}$
given by~\eqref{eq:z(Omega,E)}.
\end{definition}

Extremal K\"ahler metrics in $\Omega$, introduced by Calabi in~\cite[\S 1]{Ca}, can be characterized as those $\omega \in
\cK_\Omega$ such that the Hamiltonian vector field of
$S_\omega$ over $(X,\omega)$ is in $\LieH_J$~\cite[\S
2]{Ca}. In particular, all cscK metrics are extremal. Similarly, extremal pairs admit a
description in terms of real-holomorphic vector fields on the total
space of $E^c$. To see this, recall that each $H\in\cR$ induces a reduction of $E^c$ to a principal
$G$-bundle $E_H\subset E^c$ (see \secref{sub:concrete-setup}) and each
$b=(\omega,H)\in B_\Omega$ determines a short exact sequence of Lie
groups (see~\secref{sec:Ceqham-act-ext-grp})
\begin{equation}
\label{eq:ses-G_H-X_b-H_omega}
  1 \to \cG \lto \cX \lra{p} \cH\to 1,
\end{equation}
given from left to right by the gauge group of $E_H$, the extended
gauge group of $E_H$ over $(X,\omega)$ and the group of Hamiltonian
symplectomorphisms of $(X,\omega)$. This exact sequence induces
another one
\begin{equation}
\label{eq:Ext-Lie-group-holmorphic}
1 \to \cG_I \lto \cX_I \lra{p} \cH_J,
\end{equation}
where $\cG_I=\cG\cap\Aut(E^c,I)$, $\cX_I=\cX\cap\Aut(E^c,I)$
and $\cH_J =\cH\cap\Aut(X,J)$ are finite dimensional complex
Lie groups (see e.g.~\cite[\S 2.120]{Be}). Note that the Lie algebra $\LieX_I$ is given by $G^c$-invariant
real-holomorphic vector fields on the total space of $(E^c,I)$
covering Hamiltonian (real-holomorphic) vector fields on
$(X,J,\omega)$.

Using the horizontal lift $\theta_H^\perp\colon
\LieH\to \LieX_b$, of the Chern connection associated to $H$ and
$I$ (cf.~\eqref{eq:induced-theta}), we define
\begin{equation}
\label{eq:zetab}
\zeta_{\alpha}(b) \defeq - 4\alpha_1(\Lambda_\omega F_H - z_\Omega) - \theta_H^\perp\eta_{\alpha}(b)
\in \LieX_b,
\end{equation}
for each $b=(H,\omega)\in B_\Omega$. Then it follows
from~\eqref{eq:infinit-action-connections} and~\eqref{eq:pairsinvacs2}
that
\begin{equation}
\label{eq:extrpair-holomvectfield}
\text{$b\in B_\Omega$ is extremal
$\Longleftrightarrow$ $\zeta_{\alpha}(b) \in \LieX_I$.}
\end{equation}

The following link between extremal pairs and the coupled equations is
a generalization of the corresponding link for K\"{a}hler metrics (see
e.g.~\cite[Lemma~1]{LS1}). To establish this, note that each $\omega\in\cK_\Omega$ induces $L^2$-inner products on
$C^\infty(X)$ and $\Omega^0(\ad E_H)$, given by
\begin{subequations}
\label{eq:pairings}
\begin{align}
\label{eq:pairingH}
\langle\phi_0,\phi_1\rangle_\omega & \defeq \int_X \phi_0 \phi_1 \omega^{[n]},
\\
\label{eq:LieG-dual}
\langle\xi_0,\xi_1\rangle_\omega & \defeq \int_X (\xi_0 \wedge \xi_1)
\omega^{[n]},
\end{align}
\end{subequations}
for $\phi_j \in C^\infty(X)$, $\xi_j \in \Omega^0(\ad E_H)$
($j=0,1$). Their associated $L^2$-norms are denoted
$\|\cdot\|_\omega$.

\begin{proposition}
\label{prop:extremal-solution}
A pair $b\in B_\Omega$ is a
solution to the coupled equations~\eqref{eq:CYMeq2} if and only if it is an extremal pair and $\cF_{\alpha,\Omega}=0$.
\end{proposition}

\begin{proof}
If there exists a solution $b=(\omega,H)\in B_\Omega$ to the
coupled equations~\eqref{eq:CYMeq2}, then $\cF_{\alpha,\Omega}=0$, by ~\eqref{eq:equivalent-slogan}
(or~\eqref{eq:alphafutakismooth}) and
Theorem~\ref{thm:futakiobstructionformal}, and futhermore, $b$ is
obviously an extremal pair, since
$d_H\Lambda_\omega F_H=0$ and $\eta_{\alpha}(b)=0$.
Conversely, if $b=(\omega,H)$ is extremal, i.e. $\zeta_{\alpha}(b) \in \LieX_I$, then
\[
  \cF_{\alpha,\Omega}(\zeta_{\alpha}(b)) = \|S_\alpha(b)-\hat S_\alpha\|_\omega^2 + 16 \alpha^2_1\|\Lambda_\omega F_H - z_\Omega\|_\omega^2 \geq 0,
\]
by~\eqref{eq:alphafutakismooth}, where $\hat S_\alpha = \int_X S_\alpha(b)\omega^{[n]}/\Vol_\Omega$, so $\cF_{\alpha,\Omega}=0$ implies
that $b$ satisfies~\eqref{eq:CYMeq2}.
\end{proof}

Extremal pairs enjoy good regularity properties, similar to those of extremal K\"{a}hler metrics~\cite[Proposition
4]{LS2}:

\begin{lemma}
\label{lem:reg-extremalpairs}
Let $(\omega,H)$ be an extremal pair such that $\omega$
is a K\"{a}hler form of class $C^2$ on $(X,J)$ and $H$ is a section of
$E^c/G$ of class $C^4$. Then both $\omega$ and $H$ are smooth.
\end{lemma}

\begin{proof}
We will show by induction on $l \in \mathbb{N}$ that $\omega$ and $H$
are H\"older of class $C^{2l-1,\beta}$ and $C^{2l+1,\beta}$
respectively, for all $\beta \in (0,1)$ and $l \in
\mathbb{N}$. By
assumption, $\omega$ and $H$ are of class $C^{1,\beta}$ and
$C^{3,\beta}$, respectively. Suppose now that $\omega$ and $H$ are of
class $C^{2l-1,\beta}$ and $C^{2l+1,\beta}$, respectively. As
$\eta_{\alpha}(b)$ is a real holomorphic vector field
by~\eqref{eq:extrpair-holomvectfield}, it is real analytic, so
$dS_\alpha(b)$ is of class $C^{2l-1,\beta}$, i.e. $S_\alpha(b) \in
C^{2l,\beta}$, and hence it follows from~\eqref{eq:Salpha} that the
scalar curvature $S_\omega$ is of class $C^{2l-1,\beta}$, because
\[
\Lambda^2_\omega (F_H\wedge F_H) - 4\Lambda_\omega F_H \wedge z_\Omega
\]
is of class $C^{2l-1,\beta}$. Arguing as in \cite[Proposition~4]{LS2},
it follows from the regularity theory for the Laplacian and for the
Monge--Amp\`ere equation that $\omega$ is of class $C^{2l+1,\beta}$
(recall that the scalar curvature can be written in holomorphic
coordinates as $\Delta_\omega \log \det(\omega)$). Since
$\zeta_{\alpha}(b)$, defined by~\eqref{eq:zetab}, is a real-holomorphic
vector field on $(E^c,I)$ by assumption
(see~\eqref{eq:extrpair-holomvectfield}), it is real analytic and so
$\Lambda_\omega F_H$ is of class $C^{2l,\beta}$, because $\alpha_1\neq
0$. Identifying $H$ locally with a
function on the base with values in $\exp(\imag\mathfrak{g})\subset
G^c$ and using holomorphic coordinates for the bundle $E^c$, we can
write
\begin{equation}
\label{eq:proporeg2}
\Delta_{\dbar}H = H(\Lambda_\omega F_H - \Lambda_\omega (\dbar(H^{-1})\wedge \partial H)),
\end{equation}
where the right-hand side is of class $C^{2l,\beta}$ and
$\Delta_{\dbar}$ is elliptic with $C^{2l+1,\beta}$ coefficients. By
the regularity theory of linear elliptic differential operators, $H$
is of class $C^{2l+2,\beta}$ (see
e.g.~\cite[Theorem~3.55]{Au}). Applying this argument again
to~\eqref{eq:zetab}, we see that the right-hand side of
\eqref{eq:proporeg2} is of class $C^{2l+1,\beta}$ and so $H$ is of
class $C^{2l+3,\beta}$, as required.
\end{proof}

\begin{remark}
\label{rem:riemannian-symm-spc-v1}
Note that $B_\Omega=\cK_\Omega\times\cR$ has a Riemannian metric
$g_\Omega$, given by
\begin{equation}
\label{eq:pairingscX}
g_\Omega(v_0,v_1) = \int_X \phi_0 \phi_1 \omega^{[n]} + \int_X (\xi_0 \wedge \xi_1) \omega^{[n]}
\end{equation}
for $b=(\omega,H)$, $v_j=(\phi_j,\xi_j)\in T_bB_\Omega\cong
C_0^\infty(X) \times\Omega^0(\ad E_H)$
(cf.~\eqref{eq:pairing-product}), with $\dot{H}=\imag\xi$ in the
notation of~\secref{sub:ANCeq}.
Although this metric is rather canonical, it does not endow the
symmetric space $(B_\Omega,\nabla)$ of Theorem~\ref{thm:symm-spc} with
a structure of Riemannian symmetric space, since $g_\Omega$ is not
preserved in general by the canonical affine connection $\nabla$ on
$B_\Omega$ constructed in~\secref{sub:ANsymm-spc}. In fact, by a
straightforward calculation using
formula~\eqref{eq:parallel-trans-explicit} for the parallel transport,
\begin{align*}
(\nabla_{v_0} g_\Omega)(v_1,v_2) & = - \int_X \bigg{(}\dot{\xi}_1
\wedge(\theta_{H}[\zeta_I(v_0),\theta_{H}^\perp J \eta_{\phi_2}])\\
& + (\theta_{H}[\zeta_I(v_0),\theta_{H}^\perp J
\eta_{\phi_1}])\wedge\dot{\xi}_2\bigg{)} \; \omega^{[n]}.
\end{align*}
However, if the group $G^c$ is trivial, so $B_\Omega=\cK_\Omega$, then
$g_\Omega$ is precisely the Mabuchi metric and $\nabla g_\Omega=0$, by
the previous formula, so we recover the known fact~\cite{D6,Mab1} that
$\cK_\Omega$ is a Riemannian symmetric space with Levi--Civita
connection $\nabla$, by Theorem~\ref{thm:symm-spc}.
\end{remark}

\subsection{Holomorphic vector fields on the principal bundle}
\label{sub:hol-vect-fields}

Given $b=(\omega,H)\in B_\Omega$, we now relate the Lie algebra
$\LieX_I$ (see~\eqref{eq:Ext-Lie-group-holmorphic}) to the space
of solutions to a fourth-order elliptic differential equation which is
closely related to the linearization of our coupled equations. We will use the inner product on $C^\infty(X)\times\Omega^0(\ad E_H)$
induced by~\eqref{eq:pairings}, i.e. given by
\begin{equation}
\label{eq:pairing-product}
\langle v_0,v_1\rangle_\omega \defeq \langle\phi_0,\phi_1\rangle_\omega + \langle\xi_0,\xi_1\rangle_\omega,
\end{equation}
for $v_j=(\phi_j,\xi_j)\in C^\infty(X)\times\Omega^0(\ad E_H)$
($j=0,1$).

We define an operator
\begin{equation}\label{eq:P_omega}
\begin{gathered}
\xymatrix @R=0ex @C=-7ex{
**[l]\operatorname{P}=\operatorname{P}_\omega\colon\; C^{\infty}(X)\ar[r] & **[r]\Omega^0(\End TX)\\
**[l]\phi\ar@{|->}[r] & **[r] -L_{\eta_\phi}J.
}
\end{gathered}
\end{equation}
In other words, $\operatorname{P}$ is induced by the infinitesimal
action of $\LieH$ on $\cJ_\omega$. Let $\operatorname{P}^*$ be the
formal adjoint of $\operatorname{P}$ with respect to the $L^2$-inner
products on $C^\infty(X)$ and $\Omega^0(\End TX)$ induced by
$\omega(\cdot,J\cdot)$, with the $L^2$-inner product on $\Omega^0(\End
TX)$ multiplied by a factor of $1/2$, so that its restriction to
$T_J\cJ$ coincides with $\omega_\cJ(\cdot,\mathbf{J}\cdot)$ (defined
by~\eqref{eq:SympJ}). Then $\operatorname{P}^*\operatorname{P}$ is, up
to a multiplicative constant factor, the Lichnerowicz operator of the
compact K\"ahler manifold $(X,J,\omega)$. This is an elliptic
self-adjoint semipositive differential operator of order 4, whose
kernel is the set of functions $\phi$ such that $\eta_\phi\in\LieH_J$,
and which may be interpreted as the linearization of the cscK equation
at $\omega$ (see e.g.~\cite{LS2}).

We define now an operator which is closely related to the
linearization of the coupled equations (see
Proposition~\ref{propo:DefTalpha}) and which will play the role of the
Lichnerowicz operator in our study. The operator
is
\begin{equation}
\label{eq:Lalphaoperator}
\mathbf{L}_{\alpha,b}= (\mathbf{L}_{\alpha,b}^0,\mathbf{L}_{\alpha,b}^1)
\colon C^\infty(X) \times \Omega^0(\ad E_H) \lto C^\infty(X) \times \Omega^0(\ad E_H),
\end{equation}
where $\mathbf{L}_{\alpha,b}^0 $ and $\mathbf{L}_{\alpha,b}^1$ are
defined by
\begin{equation}
\label{eq:Dalphapaoperator}
\begin{split}
\mathbf{L}_{\alpha,b}^0(\phi,\xi) & = \alpha_0 \operatorname{P}^*\operatorname{P}
\phi - 2\alpha_1 \Lambda^2_\omega (F_H\wedge d_H J(d_H \xi +
\eta_\phi\lrcorner F_H)) + \mathbf{L}_{\alpha,b}^1(\phi,\xi) \wedge z_\Omega,\\
\mathbf{L}_{\alpha,b}^1(\phi,\xi) & = 4\alpha_1\Lambda_\omega d_H J (d_H \xi +
\eta_\phi\lrcorner F_H) = 4\alpha_1d_H^*(d_H \xi + \eta_\phi\lrcorner F_H).
\end{split}
\end{equation}
Here, $J$ is the endomorphism of $\Omega^1(\ad E_H)$ induced by the
complex structure $J$ (see~\cite[(2.8)]{Be}), $d_H\colon
\Omega^0(\ad E_H) \to \Omega^1(\ad E_H)$ is the covariant derivative
of the Chern connection of $H$ and $I$ and, by the K\"{a}hler
identities, $d_H^*=\Lambda_\omega d_H J$ is its formal adjoint.

Recall that the Chern connection associated to any
$H\in\cR$ and $I$ induces a vector space isomorphism
(see~\secref{sec:Ceqcoupling-term})
\begin{equation}
\label{eq:zetaphixi}
C^\infty(X)/\RR \times \Omega^0(\ad E_H) \lra{\cong} \LieX_b
\colon ([\phi],\xi) \longmapsto \xi + \theta_H^\perp\eta_\phi.
\end{equation}

Let $(Z_b,\textbf{I}_b,\omega_b)$ be the K\"{a}hler manifold
constructed in Lemma~\ref{lem:pairsinvacs2} and
Theorem~\ref{prop:Kahlerfibration}.

\begin{lemma}
\label{lem:bra-L-ket}
Let $v_j=(\phi_j,\xi_j)\in C^\infty(X)\times\Omega^0(\ad E_H)$,
for $j=0,1$. Then
\begin{equation}
\label{eq:Dalphaselfadjoint}
  \langle v_0,\mathbf{L}_{\alpha,b} v_1\rangle_\omega
  = \omega_b(Y_{\zeta_0|I},\mathbf{I} Y_{\zeta_1|I})
    + 4\alpha_1 \langle J\eta_{\phi_0}\lrcorner(d_H \xi_1 +
    \eta_{\phi_1}\lrcorner F_H),\Lambda_\omega F_H - z_\Omega\rangle_\omega,
\end{equation}
where $Y_{\zeta_j}$ is the infinitesimal action of $\zeta_j = \xi_j +
\theta_H^\perp\eta_{\phi_j} \in \LieX_b$ on $Z_b$, for $j=0,1$.
\end{lemma}

\begin{proof}
By the moment map interpretation of scalar curvature of the K\"{a}hler
metric $\omega(\cdot,J\cdot)$, its derivative $\delta_JS_\omega\colon
T_J\cJ_\omega\to C^\infty(X)$ with respect to $J\in\cJ_\omega$
satisfies
\begin{equation}
\label{eq:deltaJS}
\delta_J S_\omega\circ\mathbf{J}\circ\operatorname{P} = - \operatorname{P}^*\operatorname{P}
\end{equation}
(see~\secref{section:cscKmmap} and~\cite[equation (26)]{D7}).
By~\eqref{eq:compatible-Y-imag_I} and Lemma~\ref{lem:pairsinvacs2},
$\mathbf{I} Y_{\zeta_0|I}$ is identified with the infinitesimal
action of $I\zeta_0$ on $(J,A)\in\cP_b$, where $A$ is the Chern
connection of $H$ and $I$, so
\begin{equation}
\label{eq:IYzeta}
  \mathbf{I} Y_{\zeta_1|I} = - L_{I\zeta_1}I
  = (- JL_{\eta_{\phi_1}}J,J(d_H\xi_1 + \eta_{\phi_1} \lrcorner F_H)),
\end{equation}
by
Lemma~\ref{lem:infinit-action-connections}. Hence~\eqref{eq:Dalphaselfadjoint}
follows from formula~\eqref{eq:thm-muX} for the moment map $\mu_b$.
\end{proof}

Given a pair $b=(\omega,H)\in B_\Omega$, an element $\zeta\in\LieX_b$
is in $\LieX_I$ if and only if $Y_{\zeta|I}=0$.
Using~\eqref{eq:IYzeta}, we see that $\LieX_I
\subset \LieX_b $ is the subset of elements $\zeta=\xi +
\theta_H^\perp\eta_\phi$ such that
\begin{equation}
\label{eq:infinit-action-connections2}
\operatorname{P}\phi = 0, \quad d_H \xi + \eta_\phi \lrcorner F_H = 0.
\end{equation}
Hence if $(\phi,\xi) \in C^\infty(X)\times\Omega^0(\ad E_H)$ satisfies
$\xi + \theta_H^\perp\eta_\phi\in \LieX_I$, then it is in
$\ker\mathbf{L}_{\alpha,b}$ (see~\eqref{eq:Dalphapaoperator}). We
provide now sufficient conditions to obtain the converse implication.

Given a pair $b=(\omega,H)\in B_\Omega$, $H$ is a
\emph{Hermitian--Yang--Mills} reduction (HYM) on $(E^c,I)$ with
respect to $\omega$ if it satisfies
\begin{equation}
\label{eq:HMY-reduction}
\Lambda_\omega F_H=z_\Omega.
\end{equation}
Note that if $H$ is HYM on $(E^c,I)$
with respect to $\omega$, then~\eqref{eq:Dalphaselfadjoint} becomes
simply
\begin{equation}
\label{eq:bra-L-ket-HYM}
\langle v_0, \mathbf{L}_{\alpha,b} v_1\rangle_\omega
= \omega_b(Y_{\zeta_0|I},\mathbf{I} Y_{\zeta_1|I}).
\end{equation}

\begin{proposition}
\label{propo:impfunc}
The operator $\mathbf{L}_{\alpha,b}$ is elliptic. If $H$ is
HYM 
with respect to $\omega$, then $\mathbf{L}_{\alpha,b}$ is also
self-adjoint. If furthermore $\alpha_0\alpha_1 > 0$, then
\begin{equation}
\label{eq:ker-L}
\ker\mathbf{L}_{\alpha,b} = \{(\phi,\xi) \in C^\infty(X)\times\Omega^0(\ad E_H)
\,|\, \xi + \theta_H^\perp\eta_\phi\in \LieX_I\}.
\end{equation}
\end{proposition}

\begin{proof}
The operator $\mathbf{L}_{\alpha,b}$ is elliptic because so are
$\operatorname{P}^*\operatorname{P}$ and
$d_H^*d_H$. If $H$ is
HYM, then we can apply~\eqref{eq:bra-L-ket-HYM}, where
$\omega_b(\cdot, \mathbf{I}\cdot)$ is symmetric, so $\mathbf{L}_{\alpha,b}$ is
self-adjoint. We have already seen that the right-hand side
of~\eqref{eq:ker-L} is contained in $\ker\mathbf{L}_{\alpha,b}$. If
in addition $\alpha_0\alpha_1 > 0$, then $\omega_b$ is compatible with
either $\mathbf{I}$ or $-\mathbf{I}$
(see~\secref{sec:Ceqcoupled-equations}), so if $v=(\phi,\xi)$
satisfies $\zeta\defeq \xi+\theta_H^\perp\eta_\phi\notin\LieX_I$,
then $\langle v, \mathbf{L}_{\alpha,b} v\rangle_\omega =
\omega_b(Y_{\zeta|I},\mathbf{I} Y_{\zeta|I}) \neq 0$
by~\eqref{eq:bra-L-ket-HYM}, and hence
$v\notin\ker\mathbf{L}_{\alpha,b}$. This implies~\eqref{eq:ker-L}.
\end{proof}

Observe that, although $\mathbf{L}_{\alpha,b}$ is an analogue in our
context of the Lichnerowicz operator, there is an important difference
between these two operarors, since by Proposition~\ref{propo:impfunc},
we can ensure that $\mathbf{L}_{\alpha,b}$ is self-adjoint and its
kernel corresponds to $\LieX_I$ via~\eqref{eq:zetaphixi} only when
$b=(\omega,H)$ satisfies the Hermitian--Yang--Mills
equation~\eqref{eq:HMY-reduction} and $\alpha_0\alpha_1 > 0$.

\subsection{The linearized coupled equations}
\label{sec:Defmmapoperator}

Throughout~\secref{sec:Defmmapoperator}, we fix a coupling constant
$\alpha\in\RR^2$, a holomorphic structure $I$ on $E^c$ over $(X,J)$, a
K\"{a}hler class $\Omega$ on $(X,J)$ and $b=(\omega,H)\in
B_\Omega$. Let $H^{1,1}(X,\RR)\subset H^2(X,\RR)$ be the vector subspace of those
de Rham classes which are representable by real closed $(1,1)$-forms
on $(X,J)$. Recall that $H^{1,1}(X,\RR)$ is identified by Hodge theory
with the space $\cH^{1,1}(X)$ of real harmonic $(1,1)$-forms on
$(X,J,\omega)$.

In~\secref{sec:Defmmapoperator}, we will compute the first-order
deformations of the moment map $\mu_b$ constructed in
Theorem~\ref{prop:Kahlerfibration} under deformations given by a new
K\"{a}hler form $\widetilde{\omega}$ and a new holomorphic structure
$\widetilde{I}$ on the principal bundle $E^c$ over $(X,J)$, given by
\begin{subequations}
\label{eq:omega-deformations}
\begin{align}
\label{eq:omegatilde}
\widetilde{\omega} & \defeq \omega + \gamma + dd^c\phi,
\\
\label{eq:Htilde}
\widetilde{I} & \defeq e^{\imag\xi} \cdot I,
\end{align}
\end{subequations}
parametrised by a triple
\[
  (\gamma, \phi, \xi) \in \cH^{1,1}(X)\times C^\infty(X)\times \Omega^0(\ad E_H).
\]
We will also consider the deformed pair
\begin{equation}
\label{eq:btilde}
\widetilde{b}=(\widetilde{\omega},\widetilde{H})\in B_{\widetilde{\Omega}},
\text{ with } \widetilde{H}\defeq e^{-\imag\xi} \cdot H\in\cR,
\end{equation}
where $\widetilde{\Omega}$ is the cohomology class of
$\widetilde{\omega}$. Note that~\eqref{eq:changecurvature} implies
\begin{equation}
\label{eq:Ftilde}
F_{H,  \widetilde{I}} = e^{\imag\xi}\cdot F_{\widetilde{H},I},
\end{equation}
where $F_{H,I}$ is the curvature of the Chern connection
$\theta_H=\theta_{H,I}$ associated to $H$ and $I$.

In fact, to prove Theorems~\ref{thm:DeformationCYMeq2}
and~\ref{thm:DeformationCYMeq4}, we will need to apply the implicit
function theorem, so we will work in Sobolev spaces. Let $L^2_k(X)$
and $L^2_k(\ad E_H)$ be the Sobolev spaces of real-valued functions on
$X$ and sections of the bundle $\ad E_H$, respectively, whose
distributional derivatives up to order $k$ are square
integrable. These are real Hilbert spaces that, by the Sobolev
embedding theorem, have natural bounded inclusion maps
$L^2_k(X)\subset C^l(X)$ and $L^2_k(\ad E_H)\subset C^l(\ad E_H)$ into
the Banach spaces of $l$-times continously differentiable functions
and sections of $\ad E_H$, respectively, provided $k>n+l$. Moreover,
if $k>n$, then $L^2_k(X)$ is a Banach algebra.
Fix $k>n$. Let
\begin{equation}
\label{eq:subset-U}
  \mathcal{U}=\hat{\mathcal{U}}\times L^2_{k+4}(\ad E_H) \subset \cH^{1,1}(X) \times L^2_{k+4}(X) \times L^2_{k+4}(\ad E_H),
\end{equation}
where $\hat{\mathcal{U}}\subset\cH^{1,1}(X)\times L^2_{k+4}(X)$ is the
open neighbourhood of $(0,0)$ consisting of pairs $(\gamma, \phi)$
such that $\widetilde{\omega}(\cdot,J\cdot)$ is a K\"{a}hler metric of
class $C^2$, with $\widetilde{\omega}$ defined
by~\eqref{eq:omegatilde}. Define the \emph{moment map operator}
\begin{equation}
\label{eq:Talphamap}
\begin{gathered}
\xymatrix @R=0ex @C=-16ex {**[l]
\operatorname{T}_\alpha=(\operatorname{T}_\alpha^0,\operatorname{T}_\alpha^1) \colon
\; \mathcal{U} \ar[r] & **[r] L^2_{k}(X) \times L^2_{k+2}(\ad E_H)\\
**[l](\gamma,\phi,\xi) \ar@{|->}[r] &  **[r]
\(S_\alpha(\widetilde{b},I), 
4\alpha_1(\Lambda_{\widetilde{\omega}}F_{H,\widetilde{I}} - z_{\widetilde{\Omega}})\),\\
}
\end{gathered}
\end{equation}
where $\widetilde{\omega}, \widetilde{I}$ and $\widetilde{b}$ are
defined by~\eqref{eq:omega-deformations} and~\eqref{eq:btilde}, while
$S_\alpha(\widetilde{b},I)$ and $z_{\widetilde{\Omega}}$ are given by
the formulae~\eqref{eq:Salpha} and~\eqref{eq:HMY-reduction}, using the
K\"{a}hler class
\[
  \widetilde{\Omega}\defeq [\widetilde{\omega}]\in H^{1,1}(X,\RR).
\]
Observe that $\operatorname{T}_\alpha$ is a variant for Sobolev
spaces of the families of moment maps $\mu_b$.

The following proposition can be compared with~\cite[Proposition~5]{LS1}.

\begin{proposition}
\label{propo:DefTalpha}
For $k>n$, $\operatorname{T}_\alpha$ is a well-defined $C^1$ map
whose Fr\'echet derivative $\delta \operatorname{T}_\alpha$ at the
origin $(0,0,0)$ is given by
\begin{equation}
\label{eq:Talphamapder}
\begin{split}
\delta\operatorname{T}_\alpha(\dot{\gamma},\dot{\phi},\dot{\xi}) = \mathbf{L}_{\alpha,b}(\dot{\phi},\dot{\xi}) & + \((d(S_\alpha(b,I)),d\dot{\phi})_\omega,4\alpha_1J\eta_{\dot{\phi}}\lrcorner d_H\Lambda_\omega F_H\)\\
 & + \delta_{\dot{\gamma}}\operatorname{T}_\alpha,\\
\end{split}
\end{equation}
for all $(\dot{\gamma},\dot{\phi},\dot{\xi})\in \cH^{1,1}(X) \times
L^2_{k+4}(X) \times L^2_k(\ad E_H)$, where
\[
\mathbf{L}_{\alpha,b}\colon L^2_{k+4}(X) \times L^2_{k+4}(\ad E_H) \lto
L^2_{k}(X) \times L^2_{k+2}(\ad E_H)
\]
is given by~\eqref{eq:Dalphapaoperator}, $(\cdot,\cdot)_\omega$ is the
inner product on $T^*X$ induced by $\omega(\cdot,J\cdot)$, $\eta_{\dot{\phi}}$ is the Hamiltonian vector
field of $\dot{\phi}$ on $(X,\omega)$ and
$\delta_{\dot{\gamma}}\operatorname{T}_\alpha$ is the
directional derivative of $\operatorname{T}_\alpha$ at the
origin in the direction $(\dot{\gamma},0,0)$.
\end{proposition}

\begin{proof}
The operator $\operatorname{T}_\alpha$ is well-defined because
$L^2_k(X)$ is a Banach algebra for $k>n$,
$\operatorname{T}_\alpha^0$ is a non-linear differential
operator of order 4 in $\phi$ and order 2 in $\gamma$ and $\xi$, while
$\operatorname{T}_\alpha^1$ is a non-linear differential
operator of order 2 in $\phi$ and $\xi$ and order $0$ in $\gamma$.

To prove that $\operatorname{T}_\alpha$ is $C^1$, we will
calculate its directional derivatives
$\delta_{(\dot{\phi},\dot{\xi})}\operatorname{T}_\alpha(\gamma,\phi,\xi)$
and $\delta_{\dot{\gamma}}\operatorname{T}_\alpha(\gamma,\phi,\xi)$
at $(\gamma,\phi,\xi)$ in the directions $(0,\dot{\phi},\dot{\xi})$
and $(\dot{\gamma},0,0)$, respectively, for $(\gamma,\phi,\xi) \in
\mathcal{U}$, $(\dot{\gamma},\dot{\phi},\dot{\xi})\in \cH^{1,1}(X)
\times L^2_{k+4}(X) \times L^2_k(\ad E_H)$.

To compute $\delta_{(\dot{\phi},\dot{\xi})}\operatorname{T}_\alpha(\gamma,\phi,\xi)$,
we define a curve (on an appropriate Sobolev completion of $B_\Omega$ and for $|t|$ small),
given by
\[
b_t = (\widetilde{\omega}_t,\widetilde{H}_t)\defeq
(\widetilde{\omega}+tdd^c\dot{\phi},e^{-\imag(\xi+t\dot{\xi})}\cdot H).
\]
Let $\eta_t$ be the Hamiltonian vector field of $\dot{\phi}$ over
$(X,\widetilde{\omega}_t)$ and $g_t$ the flow of
\begin{equation}
\label{eq:DefTalpha-y_t}
y_t\defeq I\zeta_I(\dot b_t) = - I(\dot{\xi} +
\theta_{H_t}^\perp\eta_t),
\end{equation}
i.e. the curve of $G^c$-equivariant automorphisms of $E^c$ satisfying
$\dot g_t\cdot g_t^{-1} =y_t$, with initial condition $g_0=\Id$. Since
the K\"ahler class $\widetilde{\Omega}$ of $\widetilde{\omega}_t$ is
constant along the curve $b_t$, we can apply the constructions in the
proof of Proposition~\ref{prop:4-tuples-C1-4}(1), so the flow $g_t$
exists and satisfies
\begin{equation}
\label{eq:flow-DefTalpha}
b_t = g_t\cdot\tilde{b}
\end{equation}
(as $b_0=\tilde{b}$). Note that the identity~\eqref{eq:flow-DefTalpha}
holds in a strong sense, as $k > n$, so the K\"ahler metrics
$\widetilde{\omega}_t$ are of class $C^2$ and the $G$-reductions
$\widetilde{H}_t$ are of class $C^4$. Define another curve
\[
  I_t \defeq g^{-1}_t\cdot I
\]
in (an appropriate Sobolev completion of) the space $Z_{\widetilde{b}}$ of holomorphic structures on the
principal $G^c$-bundle $E^c$ which are compabible with $\widetilde{b}$
(see~\secref{sub:concrete-setup}). Using
the dependence of $S_\alpha(b_t,I)$ on the holomorphic structure $I$
on $E^c$, we obtain
\[
\operatorname{T}^0_t
\defeq\operatorname{T}_\alpha^0(\gamma,\phi+t\dot{\phi},\xi+t\dot{\xi})
=S_\alpha(b_t,I) 
=S_\alpha(\widetilde{b},I_t)\circ\check{g}_t^{-1} 
\]
by~\eqref{eq:Salpha-bI-g} and~\eqref{eq:flow-DefTalpha}. Since
$\frac{d}{dt}_{|t=0} I_t = L_{y_0}I$, this implies
\begin{equation}
\label{eq:directional-deriv-1-Talpha^0}
\begin{split}
\delta_{(\dot{\phi},\dot{\xi})}\operatorname{T}_\alpha^0(\gamma,\phi,\xi) & =
\frac{d}{dt}_{|t=0}\operatorname{T}^0_t = (\delta_I S_\alpha)_{|(\tilde{b},I)}(L_{y_0}I) + J\eta_{\dot{\phi}} \lrcorner d(S_\alpha(\tilde{b},I))\\
& = (\delta_I S_\alpha)_{|(\tilde{b},I)}(L_{y_0}I) + (d(S_\alpha(\tilde{b},I)),d\dot{\phi})_{\widetilde{\omega}},
\end{split}
\end{equation}
where $\delta_I S_\alpha\colon T_I Z_{\widetilde{b}}\to C^\infty(X)$ is the
derivative of $S_\alpha$ with respect to $I$. Now, by~\eqref{eq:IYzeta}
\[
L_{y_0}I = (- JL_{\eta_{\dot{\phi}}}J,
J(d_{\widetilde{H}}\dot{\xi} + \eta_{\dot{\phi}} \lrcorner F_{\widetilde{H}}))
\]
and from this formula,~\eqref{eq:deltaJS}
and~\eqref{eq:Dalphapaoperator}, we obtain
\begin{align*}
(\delta_I S_\alpha)_{|(\widetilde{b},I)}(\mathbf{I}Y_{\zeta_{\dot{b}}}) & =
\alpha_0 \operatorname{P}^*\operatorname{P} \dot{\phi}
- 2\alpha_1 \Lambda^2_{\widetilde{\omega}} (F_{\widetilde{H}}\wedge d_{\widetilde{H}}J(d_{\widetilde{H}}\dot{\xi} + \eta_{\dot{\phi}} \lrcorner F_{\widetilde{H}}) + \mathbf{L}_{\alpha,\widetilde{b}}^1(\dot{\phi},\dot{\xi}) \wedge z_{\widetilde{\Omega}}
\\
& = \mathbf{L}_{\alpha,\widetilde{b}}^0(\dot{\phi},\dot{\xi}),
\end{align*}
where $z_{\widetilde{\Omega}}$ is defined as
in~\eqref{eq:Salpha} using the K\"{a}hler class $\widetilde{\Omega}$, so the right hand side
of~\eqref{eq:directional-deriv-1-Talpha^0} is
\begin{equation}
\label{eq:directional-deriv-Talpha^0}
\delta_{(\gamma,\phi,\xi)}\operatorname{T}_\alpha^0(0,\dot{\phi},\dot{\xi})
= \mathbf{L}_{\alpha,\widetilde{b}}^0
(\dot{\phi},\dot{\xi}) + (d(S_\alpha(\widetilde{b},I)),d\dot{\phi})_{\widetilde{\omega}}.
\end{equation}
By~\eqref{eq:changecurvature}, we also have
\begin{align*}
\operatorname{T}^1_t
& \defeq \operatorname{T}_\alpha^1(\gamma,\phi + \dot{\phi}_t,\xi + t \dot{\xi})\\
& = 4\alpha_1\(e^{\imag (\xi + t \dot{\xi})}g_t\)\cdot
(\Lambda_{\widetilde{\omega}_t} F_{H,I_t} - z_{\widetilde{\Omega}}),
\end{align*}
and a straightforward calculation shows that
\begin{align}
\notag
\delta_{(\dot{\phi},\dot{\xi})}\operatorname{T}_\alpha^1(\gamma,\phi,\xi) = &
\frac{d}{dt}_{|t=0}\operatorname{T}^1_t
= 4\alpha_1\Lambda_{\widetilde{\omega}} d_{\widetilde{H}}J(d_{\widetilde{H}}\dot{\xi} + \eta_{\dot{\phi}} \lrcorner F_{\widetilde{H}}) + 4\alpha_1J\eta_{\dot{\phi}} \lrcorner d_{\widetilde{H}}\Lambda_{\widetilde{\omega}} F_{\widetilde{H}}
\\
\label{eq:directional-deriv-1-Talpha^1}
= &
\mathbf{L}_{\alpha,b}^1(\dot{\phi},\dot{\xi})
+ 4\alpha_1J\eta_{\dot{\phi}} \lrcorner d_{\widetilde{H}}\Lambda_{\widetilde{\omega}} F_{\widetilde{H}}.
\end{align}
To compute $\delta_{\dot{\gamma}}\operatorname{T}_\alpha(\gamma,\phi,\xi)$,
for $(\gamma,\phi,\xi)\in\mathcal{U}$ and $\dot{\gamma}\in\cH^{1,1}(X)$,
we define a curve
\[
  b_t=(\omega_t,\widetilde{H})=(\widetilde{\omega}+t\dot{\gamma},\widetilde{H})
\]
(for $t\in\RR$ small). Let
\begin{align*}
\operatorname{T}_t^0 \defeq &
\operatorname{T}_\alpha^0(\gamma+t\dot{\gamma},\phi,\xi)=S_\alpha(b_t,I)
\\
= & - \alpha_0 S_{\omega_t} - \alpha_1 \Lambda^2_{\omega_t}(F_{\widetilde{H}} \wedge F_{\widetilde{H}}) + 4\alpha_1\Lambda_{\omega_t}F_{\widetilde{H}}\wedge z_{\Omega_t} , 
\\
\operatorname{T}_t^1 \defeq & \operatorname{T}_\alpha^1(\gamma+t\dot{\gamma},\phi,\xi)=4\alpha_1(\Lambda_{\omega_t}F_{H,\widetilde{I}} - z_{\Omega_t}),
\end{align*}
where $\Omega_t=[\omega_t]\in H^{1,1}(X,\RR)$. 
As shown by LeBrun \& Simanca
(see~\cite[Proposition~5]{LS1} and~\cite[Proposition~6]{LS2}), the
derivative of the first term of $\operatorname{T}_t^0$ is given by
\[
\delta_{\dot{\gamma}}S_{\widetilde{\omega}}
\defeq\frac{d}{dt}_{|t=0} S_{\omega_t}
=\Delta_{\widetilde{\omega}}(\widetilde{\omega},\dot{\gamma})_{\widetilde{\omega}}
-2(\rho_{\widetilde{\omega}},\dot{\gamma})_{\widetilde{\omega}},
\]
where $\Delta_{\widetilde{\omega}}$ and $\rho_{\widetilde{\omega}}$
are the Laplacian and the Ricci curvature of
$\widetilde{\omega}(\cdot, J\cdot)$, respectively. To calculate the
derivatives of $\operatorname{T}_t^1$ and of the second term of
$\operatorname{T}_t^0$, we use the equality
\[
\frac{d}{dt}_{|t=0} \omega^{[n]}_t
=\dot{\gamma}\wedge\widetilde{\omega}^{[n-1]}
=(\Lambda_{\widetilde{\omega}}\dot{\gamma}) \widetilde{\omega}^{[n]}
\]
and the following computations:
\begin{align*}
\frac{d}{dt}_{|t=0} & \(\Lambda_{\omega_t}F_{H,\widetilde{I}}\omega_t^{[n]}\)
\\ = &
\frac{d}{dt}_{|t=0}\(F_{H,\widetilde{I}}\wedge \omega_t^{[n-1]}\) = F_{H,\widetilde{I}} \wedge \dot{\gamma} \wedge \tilde{\omega}^{[n-2]},
\\
\frac{d}{dt}_{|t=0}  & \(\(\Lambda_{\omega_t}^2(F_{\widetilde{H}}\wedge
F_{\widetilde{H}}) - 4\Lambda_{\omega_t}F_{\widetilde{H}}\wedge z_{\Omega_t}\) \omega_t^{[n]}\)
\\ = &
\frac{d}{dt}_{|t=0} \(2F_{\widetilde{H}}\wedge
F_{\widetilde{H}}\wedge \omega_t^{[n-2]} - 4F_{\widetilde{H}}\wedge z_{\Omega_t} \omega_t^{[n-1]}\)
\\ = & 2 F_{\widetilde{H}}\wedge F_{\widetilde{H}} \wedge \dot{\gamma} \wedge \widetilde{\omega}^{[n-3]}
- 4 F_{\widetilde{H}} \wedge \(z_{\widetilde{\Omega}} \dot{\gamma} + \frac{\delta_{\dot{\gamma}}z_{\widetilde{\Omega}} \widetilde{\omega}}{n-1}\) \wedge \widetilde{\omega}^{[n-2]}.
\end{align*}
Here,~\eqref{eq:HMY-reduction} implies
\begin{equation}
\label{eq:delta-z}
  \delta_{\dot{\gamma}}z_{\widetilde{\Omega}} \defeq \frac{d}{dt}_{|t=0} z_{\Omega_t}
  = \sum_j \beta_j z_j
\end{equation}
for an orthonormal basis $\{z_j\}$ of $\mathfrak{z}$, with
\begin{align*}
\beta_j \defeq & \frac{d}{dt}_{|t=0} \frac{\langle z_j(E)\cup \Omega_t^{[n-1]},[X]\rangle}{\Vol_{\Omega_t}}\\
= & \frac{\langle z_j(E)\cup [\dot{\gamma}] \cup \widetilde{\Omega}^{[n-2]},[X]\rangle}{\Vol_{\widetilde{\Omega}}} - \frac{\langle z_j(E)\cup \widetilde{\Omega}^{[n-1]},[X]\rangle \langle[\dot{\gamma}]\cup \widetilde{\Omega}^{[n-1]},[X]\rangle}{\Vol_{\widetilde{\Omega}}^2}.
\end{align*}
From these equalities, we obtain the directional
derivatives
\begin{subequations}
\label{eq:directional-deriv-2-Talpha}
\begin{align}
\label{eq:directional-deriv-2-Talpha^0}
\delta_{\dot{\gamma}} \operatorname{T}_\alpha^0  & (\gamma,\phi,\xi)
= 
\frac{d}{dt}_{|t=0}\operatorname{T}_t^0
= \alpha_0\(2(\rho_{\widetilde{\omega}},\dot{\gamma})_{\widetilde{\omega}} - \Delta_{\widetilde{\omega}} (\widetilde{\omega},\dot{\gamma})_{\widetilde{\omega}}\)
\\ \notag &
- \frac{\alpha_1}{3}\Lambda_{\widetilde{\omega}}^3\(F_{H,\widetilde{I}}\wedge F_{H,\widetilde{I}}\wedge \dot{\gamma}\)
 + 2\alpha_1\Lambda_{\widetilde{\omega}}^2\(F_{H,\widetilde{I}} \wedge \(z_{\widetilde{\Omega}} \dot{\gamma} + \frac{\delta_{\dot{\gamma}}z_{\widetilde{\Omega}} \widetilde{\omega}}{n-1}\)\)
\\ \notag
& + \alpha_1\(\Lambda_{\widetilde{\omega}}^2(F_{H,\widetilde{I}}\wedge F_{H,\widetilde{I}}) - 4\Lambda_{\widetilde{\omega}}F_{\widetilde{H}}\wedge z_{\widetilde{\Omega}}\)(\Lambda_{\widetilde{\omega}}\dot{\gamma}),\\
\label{eq:directional-deriv-2-Talpha^1}
\delta_{\dot{\gamma}} \operatorname{T}_\alpha^1  & (\gamma,\phi,\xi)
= \frac{d}{dt}_{|t=0}\operatorname{T}_t^1 = 4\alpha_1\((F_{H,\widetilde{I}},\dot{\gamma})_{\widetilde{\omega}} - \delta_{\dot{\gamma}}z_{\widetilde{\Omega}}\).
\end{align}
\end{subequations}
It now follows
from~\eqref{eq:directional-deriv-1-Talpha^0},~\eqref{eq:directional-deriv-1-Talpha^1},~\eqref{eq:directional-deriv-2-Talpha}
and the formula~\eqref{eq:delta-z} 
for $\delta_{\dot{\gamma}}z_{\widetilde{\Omega}}$ 
that the directional derivatives are continuous.
Therefore, $\operatorname{T}_\alpha$ is $C^1$ and its Fr\'echet
derivative given by~\eqref{eq:Talphamapder}
(by~\eqref{eq:directional-deriv-1-Talpha^0}
and~\eqref{eq:directional-deriv-1-Talpha^1}).
\end{proof}

Note that an explicit formula for the directional derivative
$\delta_{\dot{\gamma}} \operatorname{T}_\alpha$ has been
calculated in~\eqref{eq:directional-deriv-2-Talpha},
although it has not been recorded in~\eqref{eq:Talphamapder}, as it is
not needed in this paper.

\subsection{Deformation of solutions}
\label{sec:ANdeformingsolutions}

As in~\secref{sec:Defmmapoperator}, we now fix a holomorphic structure $I$ on $E^c$ over $(X,J)$, a
K\"{a}hler class $\Omega$ on $(X,J)$ and $b=(\omega,H)\in
B_\Omega$. Note that $\cH_J$ acts trivally on the space
$\cH^{1,1}(X)\subset\Omega^2(X)$ of real harmonic $(1,1)$-forms for
the metric $\omega(\cdot,J\cdot)$.
Let
\begin{equation}
\label{eq:inv-Sobolev-spc}
L^2_k(X)^{\cH_J} \subset L^2_k(X)
\quad \textrm{ and }
\quad L^2_{k}(\ad E_H)^{\cX_I} \subset L^2_k(\ad E_H)
\end{equation}
be the closed subspaces of $\cH_J$-invariant functions and
$\cX_I$-invariant sections, respectively. Let
$\hat{\mathcal{V}}=\hat{\mathcal{U}}\cap\(\cH^{1,1}(X)\times
L^2_k(X)^{\cH_J}\)$ and
\begin{equation}
\label{eq:subset-V}
\mathcal{V} = \hat{\mathcal{V}}\times L^2_{k}(\ad E_H)^{\cX_I}
= \mathcal{U} \cap \(\mathcal{H}^{1,1}(X) \times L^2_{k+4}(X)^{\cH_J} \times L^2_{k+4}(\ad E_H)^{\cX_I}\).
\end{equation}
Given coupling constants $\alpha\in\RR^2$, by restriction of the maps of Proposition~\ref{propo:DefTalpha}, for
$k>n$, we obtain well-defined maps
\begin{subequations}
\label{eq:Talphamap-Lalpha-invariant}
\begin{align}
\label{eq:Talphamapinvariant}
\hat{\operatorname{T}}_\alpha \colon &
\mathcal{V} \lto L^2_{k}(X)^{\cH_J} \times L^2_{k+2}(\ad E_H)^{\cX_I},
\\
\label{eq:Dalphapaoperatorinvariant}
\hat{\mathbf{L}}_{\alpha,b} \colon &
L^2_{k+4}(X)^{\cH_J} \times L^2_{k+4}(\ad
E_H)^{\cX_I} \lto L^2_{k}(X)^{\cH_J} \times L^2_{k+2}(\ad E_H)^{\cX_I}
\end{align}
\end{subequations}
(cf.~\cite[(5.1)]{LS2}), where
$\hat{\operatorname{T}}_\alpha$ is $C^1$ with Fr\'echet
derivative given by~\eqref{eq:Talphamapder}, and
$\hat{\mathbf{L}}_{\alpha,b}$ is a linear elliptic operator.

Let $d^*$ and $\mathbf{G}$ be the formal adjoint of the de Rham
differential and the Green operator of the Laplacian for the fixed
metric $\omega(\cdot,J\cdot)$, respectively. Then for any symplectic
form $\widetilde{\omega}$ and any $\widetilde{\eta}$ in the Lie algebra
$\LieH_{\widetilde{\omega}}$ of Hamiltonian vector fields over
$(X,\widetilde{\omega})$ we have
\begin{equation}
\label{eq:Hamiltonianfunction=Greenfunction}
d(\mathbf{G}d^*(\widetilde{\eta}\lrcorner\widetilde{\omega})) =\widetilde{\eta}\lrcorner\widetilde{\omega}.
\end{equation}
As the image of the Green operator is perpendicular to the constants,
the Hamiltonian function $f =
\mathbf{G}d^*(\widetilde{\eta}\lrcorner\widetilde{\omega})$ is
`normalized' for the volume form $\omega^{[n]}$, that is, $\int_X f
{\omega}^{[n]}=0$.

For each $(\gamma,\phi,\xi)\in\cV$, we define a linear map
\begin{equation}
\label{eq:Pbold}
\begin{gathered}
\xymatrix @R=0ex @C=-22ex {**[l]
\mathbf{P}_{(\gamma,\phi,\xi)}=(\mathbf{P}^0_{(\gamma,\phi)},\mathbf{P}^1_{\xi})\colon\;\;
\RR \times \mathfrak{z}(\LieX_I) \ar[r] & **[r]
L^2_{k}(X)^{\cH_J} \times L^2_{k+2}(\ad E_H)^{\cX_I}\\
**[l](t,v) \ar@{|->}[r] &  **[r]
\(\mathbf{G}d^*(p(v) \lrcorner \widetilde{\omega}) + t,\theta_{H,\widetilde{I}}v\),\\
}
\end{gathered}
\end{equation}
where
\[
  \mathfrak{z}(\LieX_I) \defeq(\LieX_I)^{\cX_I}
\]
is the centre of $\LieX_I$ (cf.~\eqref{eq:centre-z}) and $p \colon
\cX_I \to \cH_J$ is the map in~\eqref{eq:Ext-Lie-group-holmorphic},
while $\widetilde{\omega}$ and $\widetilde{H}$ are defined
by~\eqref{eq:omegatilde} and~\eqref{eq:btilde}.
The map $\mathbf{P}_{(\gamma,\phi,\xi)}$ attaches to a vector field
$v\in\mathfrak{z}(\LieX_I)$ its vertical part
$\theta_{H,\widetilde{I}}v$, calculates the normalized Hamiltonian
function of the vector field $p(v)$ over $(X,\widetilde{\omega})$, and
adds an extra parameter $t$ which accounts for the fact that
Hamiltonian functions are only determined up to a constant
(cf.~\eqref{eq:zetaphixi},~\cite[\S
5]{LS2},~\cite[Proposition~2]{LS1}).

Here is the key link between extremal pairs and the linarization of
the coupled equations.

\begin{lemma}
\label{lemma:PTalpha}
Let $(\gamma,\phi,\xi) \in \mathcal{V}$.
\begin{enumerate}
\item[\textup{(1)}]
   $\mathbf{P}_{(\gamma,\phi,\xi)}$ is injective.
\item[\textup{(2)}]
  If $\hat{\operatorname{T}}_\alpha(\gamma,\phi,\xi)\in\Im\mathbf{P}_{(\gamma,\phi,\xi)}$,
  then $\widetilde{b}=(\widetilde{\omega},\widetilde{H})$ is an
  extremal pair.
\item[\textup{(3)}]
  $\Im\mathbf{P}_{0}\subset\ker\hat{\mathbf{L}}_{\alpha,b}$, with
  equality if $\alpha_0\alpha_1 > 0$ and $H$ is HYM with respect to
  $\omega$.
\end{enumerate}
\end{lemma}

\begin{proof}
We first prove that, given $(t,v)\in\RR\times\mathfrak{z}(\LieX_I)$
and $(f,\chi)\defeq\mathbf{P}_{(\gamma,\phi,\xi)}(t,v)$, we have
\begin{equation}\label{eq:lemmaPTalpha1}
t = \int_X f\omega^{[n]}/\Vol_\Omega, \qquad v = \chi + \theta^\perp_{H,\widetilde{I}}\widetilde{\eta}_f,
\end{equation}
where $\widetilde{\eta}_f$ is the Hamiltonian vector field associated
to $f \in C^{\infty}(X)$ and $\widetilde{\omega}$. To see this, note
that, since $p(v)$ is holomorphic and preserves $\widetilde{\omega}$,
it can be written as
\[
p(v) = \widetilde{\eta}_\psi + \beta,
\]
where $\widetilde{\eta}_\psi$ is the real-holomorphic Hamiltonian
vector field associated to $\psi \in C^{\infty}(X)$ and
$\widetilde{\omega}$ and $\beta$ is a parallel vector field with
respect to $\widetilde{\omega}$ (see~\cite[\S 2]{LS1}). Then, since
$p(v)$ and $\widetilde{\eta}_\psi$ vanish somewhere on $X$, we have
that $\beta = 0$ and therefore
\[
d \psi = p(v)\lrcorner\widetilde{\omega} = df.
\]
Formula~\eqref{eq:lemmaPTalpha1} follows from the decomposition of $v$
into its vertical and horizontal parts with respect to
$\theta_{H,\tilde{I}}$.

Now, (1) follows from~\eqref{eq:lemmaPTalpha1}. To prove (2), suppose
$\hat{\operatorname{T}}_\alpha(\gamma,\phi,\xi)\in\Im\mathbf{P}_{(\gamma,\phi,\xi)}$,
i.e.
\begin{equation}
\label{eq:lemmaPTalpha2}
  f=S_\alpha(\widetilde{b}), \quad
  \chi=\theta_{h,\widetilde{I}}v
  =4\alpha_1(\Lambda_{\widetilde{\omega}}F_{H,\widetilde{I}}-z_{\widetilde{\Omega}}).
\end{equation}
From~\eqref{eq:lemmaPTalpha1}, it follows that
\begin{equation}
\label{eq:lemmaPTalpha-ext}
4\alpha_1d_{\widetilde{H}}\Lambda_{\widetilde{\omega}}F_{\widetilde{H}}
= -\widetilde{\eta}_f\lrcorner F_{\widetilde{H}},
\quad
\operatorname{P}_{\widetilde{\omega}}f=-L_{\widetilde{\eta}_f}J = 0,
\end{equation}
where we have used~\eqref{eq:changecurvature} to obtain the first
equation, while the other identity follows because
$\widetilde{\eta}_\psi = \widetilde{\eta}_f$ is
real-holomorphic. Therefore
$\widetilde{b}=(\widetilde{\omega},\widetilde{H})$ is an extremal
pair.

To prove (3), note first that the inclusion
$\Im\mathbf{P}_{0}\subset\ker\hat{\mathbf{L}}_{\alpha,b}$ is a
straightforward consequence of~\eqref{eq:lemmaPTalpha1}. Suppose now
that $\alpha_0\alpha_1>0$ and $H$ is HYM with respect to $\omega$. Let
$(f,\chi)\in\ker\hat{\mathbf{L}}_{\alpha,b}$. By
Proposition~\ref{propo:impfunc},
$v\defeq\chi+\theta^\perp_{H,I}\eta_f$ is in $\LieX_I$. In fact,
$v\in\mathfrak{z}(\LieX_I)$, as $f$ is $\cH_J$-invariant and $\chi$ is
$\cX_I$-invariant by assumption
(see~\eqref{eq:Dalphapaoperatorinvariant}). Therefore
$\mathbf{P}_{0}(v,t) = (f,\chi)$, where
$t\defeq\int_Xf\omega^{[n]}/\Vol_{\Omega}$.
\end{proof}

Let $\langle\cdot,\cdot\rangle_\omega$ be the $L^2$-inner product on $
L^2_{k}(X)^{\cH_J} \times L^2_{k+2}(\ad E_H)^{\cX_I}$
given by~\eqref{eq:pairing-product}. We claim that the orthogonal
projectors onto $\Im\mathbf{P}_{(\gamma,\phi,\xi)}$, denoted
\[
\Pi_{(\gamma,\phi,\xi)} \colon L^2_{k}(X)^{\cH_J} \times L^2_{k+2}(\ad E_H)^{\cX_I}
\lto L^2_{k}(X)^{\cH_J} \times L^2_{k+2}(\ad E_H)^{\cX_I}
\]
vary smoothly with $(\gamma,\phi,\xi) \in \mathcal{V}$. 
To prove this, note that the map
\[
\mathbf{P}\colon \cV\times \RR \times \mathfrak{z}(\LieX_I)\lto
L^2_{k}(X)^{\cH_J} \times L^2_{k+2}(\ad E_H)^{\cX_I}
\colon (\gamma,\phi,\xi,t,v)\longmapsto\mathbf{P}_{(\gamma,\phi,\xi)}(t,v)
\]
is $C^1$, as $\mathbf{P}^0_{(\gamma,\phi)}(t,v)$ is linear in
$(\gamma,\phi,t,v)$ and $\mathbf{P}^1_\xi(v)$ depends linearly on $v$
and smoothly on $\xi$.
Moreover, $\mathbf{P}_{(\gamma,\phi,\xi)}$ is an isomorphism onto its
image for all $(\gamma,\phi,\xi) \in \mathcal{V}$, by
Lemma~\ref{lemma:PTalpha}.
Let $\{w_j\}$ be a basis of the vector space
$\RR\oplus\mathfrak{z}(\LieX_I)$ and $\{\zeta_j(\gamma,\phi,\xi)\}$ be
the orthonormal basis of $\Im\mathbf{P}_{(\gamma,\phi,\xi)}$ extracted
from $\{\mathbf{P}_{(\gamma,\phi,\xi)} w_j\}$ by the Gram-Schmidt
orthogonalization process. Then the claim follows
by the above
observations and the fact that
\begin{equation}
\label{eq:Piexplicit}
\Pi_{(\gamma,\phi,\xi)} =
\sum_j \langle\zeta_j(\gamma,\phi,\xi),\cdot\rangle_{\omega} \zeta_j.
\end{equation}
Furthermore, since $\langle\zeta_j,\zeta_k\rangle_\omega$ are
continuous functions on $\cV$, the origin has an open neighbourhood
$\mathcal{V}_0 \subset \mathcal{V}$ such that for all
$(\gamma,\phi,\xi) \in \mathcal{V}_0$, the following holds
(cf.~\cite[(5.3)]{LS1}):
\begin{equation}
\label{eq:P0Pphi}
\ker(\Id - \Pi_{(\gamma,\phi,\xi)}) = \ker(\Id - \Pi_0)\circ(\Id - \Pi_{(\gamma,\phi,\xi)}).
\end{equation}

For any pair of non-negative integers $(l,m)$, let $I_{l,m}\subset
L^2_{l}(X)^{\cH_J} \times L^2_{m}(\ad E_H)^{\cX_I}$ be the orthogonal
complement of $\Im\mathbf{P}_{0}$. Define
\[
\mathcal{W} = \mathcal{V}_0 \cap (\mathcal{H}^{1,1}(X) \times I_{k+4,k+4}).
\]
Note that, under the assumptions in the last part of
Lemma~\ref{lemma:PTalpha}, the subspace $\mathcal{W}$ is perpendicular
to $\ker\mathbf{L}_{\alpha,b}$. We will use this fact to obtain
existence results about deformations of extremal pairs.
Define a LeBrun--Simanca map~\cite[\S 5]{LS2}
\begin{equation}
\label{eq:Tmap}
\begin{gathered}
\xymatrix @R=0ex @C=-15.5ex {
**[l] \mathbf{T}_\alpha \colon\;\;\;\;\;\;\; \mathcal{W} \ar[r] & **[r] I_{k,k+2}\\
**[l] (\gamma,\phi,\xi) \ar@{|->}[r] &  **[r]
(\Id - \Pi_0)\circ(\Id - \Pi_{(\gamma,\phi,\xi)})\circ \hat{\operatorname{T}}_\alpha(\gamma,\phi,\xi).\\
}
\end{gathered}
\end{equation}
Then $\mathbf{T}_\alpha$ is
$C^1$, because it is the composition of $C^1$ maps.

Given $(\dot{\phi},\dot{\xi}) \in I_{k+4,k+4}$, to calculate the
directional derivative $\delta_{(\dot{\phi},\dot{\xi})}
\mathbf{T}_\alpha$ of $\mathbf{T}_\alpha$ at the origin in
the direction $(0,\dot{\phi},\dot{\xi})$, we define the curve $b_t =
(0,t\dot{\phi},t\dot{\xi})$. Using~\eqref{eq:Talphamapder}, we obtain
\begin{align*}
\delta_{(\dot{\phi},\dot{\xi})}\mathbf{T}_\alpha =
\frac{d}{dt} &\mathbf{T}_\alpha(b_t)_{|t=0}  = (\Id - \Pi_0)\mathbf{L}_{\alpha,b}(\dot{\phi},\dot{\xi})\\
& + (\Id - \Pi_0)\((d(S_\alpha(b)),d\dot{\phi})_\omega,4\alpha_1J\eta_{\dot{\phi}}\lrcorner d_H(\Lambda_\omega F_H)\)\\
& - (\Id - \Pi_0)\frac{d}{dt}\(\Pi_{b_t}\operatorname{T}_\alpha(0)\)_{|t= 0}.
\end{align*}
Now, if $b=(\omega,H)$ is a solution to the coupled
equations~\eqref{eq:CYMeq2}, then the second summand of the
right-hand side vanishes and
$\Pi_{b_t}\operatorname{T}_\alpha(0) =
\operatorname{T}_\alpha(0)$ for all $t$, so the third summand of
the right-hand side vanishes too and hence, under this assumption, we
conclude that
\begin{equation}
\label{eq:Talphamapder2}
\delta_{(\dot{\phi},\dot{\xi})}\mathbf{T}_\alpha=(\Id - \Pi_0) \circ \hat{\mathbf{L}}_{\alpha,b}(\dot{\phi},\dot{\xi}).
\end{equation}

\begin{remark}
It is at this point that one runs into technical difficulties if one
attempts to apply the approach of LeBrun \& Simanca~\cite{LS2} to
obtain deformations of an extremal pair which is not a solution of the
coupled equations. The problem is that for an arbitrary extremal pair
$b=(\omega,H)$, if one proceeds as in~\cite[Lemma~1]{LS2}, then one
obtains
\[
\delta_{(\dot{\phi},\dot{\xi})}\mathbf{T}_\alpha
= (\Id - \Pi_0)\(\hat{\mathbf{L}}_{\alpha,b}(\dot{\phi},\dot{\xi})
+ (0,-J\eta_\alpha(b,I) \lrcorner (d_H\dot{\xi}
+ \eta_{\dot{\phi}}\lrcorner F_H))\),
\]
and to construct deformations of $b$ which are also extremal pairs
using the approach of~\cite{LS2}, we need know
that~\eqref{eq:Talphamapder2} is satisfied. A natural condition which
implies that~\eqref{eq:Talphamapder2} holds is that
$S_\alpha(\omega,H)$ is constant. Furthermore, in the approach
of~\cite{LS2}, we need to know that
$\hat{\mathbf{L}}_{\alpha,b}$ is self-adjoint, with kernel
$\Im\mathbf{P}_0$, so another natural condition is that the
Hermitian--Yang--Mills equation is satisfied, by
Proposition~\ref{propo:impfunc} and Lemma~\ref{lemma:PTalpha}(2). In
other words, to get a direct generalization of the method
of~\cite{LS2}, it is natural to impose the condition that $b$ is a
solution of the coupled equations, as we will do below.
\end{remark}

We can now prove the two main results
of~\secref{sec:ANdeformingsolutions}. For this, given $\alpha\in\RR^2$, we call $b \in B_\Omega$ an \emph{extremal pair with coupling constants $\alpha$} if it satisfies~\eqref{eq:extremalpair}.

\begin{theorem}
\label{thm:DeformationCYMeq1}
Suppose $(\omega,H)$ is a solution to the coupled
equations~\eqref{eq:CYMeq2} with coupling constant $\alpha$ and
$[\omega]=\Omega$, where $\alpha=(\alpha_0,\alpha_1)\in\RR^2$
satisfies $\alpha_0\alpha_1>0$. Then $(\alpha,\Omega)$ has an open
neighbourhood $U\subset\RR^2\times H^{1,1}(X,\RR)$ such that for all
$(\widetilde{\alpha},\widetilde{\Omega})\in U$ there exists an
extremal pair $(\widetilde{\omega},\widetilde{H})$ with coupling
constants $\widetilde{\alpha}$ and such that $[\widetilde{\omega}]=\widetilde{\Omega}$.
\end{theorem}

\begin{proof}
Note that $\eta_\alpha(b)=0$, as $b=(\omega,H)$ is a solution of the
coupled equations~\eqref{eq:CYMeq2}. Since the map
$\mathbf{T}_\alpha$ depends linearly on $\alpha =
(\alpha_0,\alpha_1)$, it can be viewed as a $C^1$ map
$\mathbf{T}\colon \RR^2 \times \mathcal{W} \to I_{k,k+2}$, whose the
Fr\'echet derivative at the origin with respect to $\phi$ and $\xi$ is
$\delta\operatorname{T}_\alpha = (\Id-\Pi_0)\circ
\hat{\mathbf{L}}_{\alpha,b}$, by~\eqref{eq:Talphamapder2}. Since $H$
is HYM with respect to $\omega$ and $\alpha_0\alpha_1>0$,
Lemma~\ref{lemma:PTalpha} applies and $(\Id - \Pi_0)\circ
\hat{\mathbf{L}}_{\alpha,b}$ is an isomorphism. Therefore, by the
implicit function theorem, there exists an open neighbourhood
$U\subset \RR^2\times\mathcal{H}^{1,1}(X)$ of $(\alpha,\Omega)$ such
that for all $(\widetilde{\alpha},\gamma) \in U$ there exists a pair
$(\phi,\xi) \in I_{k+4,k+4}$ such that
\[
\operatorname{T}_{\widetilde{\alpha}}(\gamma,\phi,\xi)\in\ker\((\Id-\Pi_0)(\Id-\Pi_{\gamma,\phi,\xi})\),
\]
so
$\operatorname{T}_{\widetilde{\alpha}}(\gamma,\phi,\xi)\in\Im\mathbf{P}_{(\gamma,\phi,\xi)}$ by~\eqref{eq:P0Pphi}. Hence the pair
$(\widetilde{\omega},\widetilde{H})$ determined by $(\gamma,\phi,\xi)$
is extremal with coupling constant $\widetilde{\alpha}$, by
Lemma~\ref{lemma:PTalpha}(1), and smooth by
Lemma~\ref{lem:reg-extremalpairs}.
\end{proof}

Let $H^{1,1}(X,\RR)^+\subset H^{1,1}(X,\RR)$ denote the `K\"{a}hler
cone' of $(X,J)$, i.e. the open subset of elements $\Omega\in
H^{1,1}(X,\RR)$ such that $\cK_\Omega$ is non-empty. Given
$(\alpha,\Omega) \in \RR^2_{>0}\times H^{1,1}(X,\RR)^+$, consider the
$\alpha$-Futaki character
$\cF_{\alpha,\Omega}\colon\LieGamma_I\lto\CC$ defined
in~\eqref{eq:Futakiformal} (or~\eqref{eq:alphafutakismooth}). Denote
\[
V(\cF)\defeq \{ (\alpha,\Omega) \,|\, \cF_{\alpha,\Omega}=0 \}
\subset \RR_{>0}^2 \times H^{1,1}(X,\RR)^+.
\]

\begin{theorem}
\label{thm:DeformationCYMeq2}
Let $S$ be the set of pairs $(\alpha,\Omega) \in \RR_{>0}^2 \times
H^{1,1}(X,\RR)^+$ for which there exists a solution $(\omega,H)\in
B_\Omega$ to the coupled equations~\eqref{eq:CYMeq2}.
\begin{enumerate}
\item[\textup{(1)}]
  Then $S \cap V(\cF)$ is open in $V(\cF)$.
\item[\textup{(2)}]
  If $\Aut(E^c,I)$ is finite, then $S \subset \RR^2 \times
  H^{1,1}(X,\RR)$ is open.
\end{enumerate}
\end{theorem}

\begin{proof}
The proof is immediate from Theorem~\ref{thm:DeformationCYMeq1},
together with Proposition~\ref{prop:extremal-solution} for part (1)
and~\eqref{eq:extrpair-holomvectfield} for part (2).
\end{proof}

\subsection{Deformations of solutions in the weak coupling limit}
\label{sec:Defweakcoupling}

We will obtain now solutions to the coupled
equations~\eqref{eq:CYMeq2} in `weak coupling limit'
$0<\lvert\alpha_1/\alpha_0\rvert\ll 1$ by deforming solutions
$(\omega,H)\in B_\Omega$ with coupling constants $\alpha_0\neq 0,
\alpha_1=0$.
Since we will study these equations for coupling constants in a small
open neighbourhood of a pair $(\alpha_0,\alpha_1)\in\RR^2$ satisfying
$\alpha_0\neq 0,\alpha_1=0$, we can divide the second equation
in~\eqref{eq:CYMeq2} by $\alpha_0$. Hence in the sequel we will
normalize to $\alpha_0=1$ and $\alpha\defeq\alpha_1$ will be called
\emph{the} coupling constant.

Note that for $\alpha=0$, the coupled equations~\eqref{eq:CYMeq2} are
the condition that $\omega$ is a cscK metric on $(X,J)$ and $H$ is a
Hermitian--Yang--Mills reduction of $(E^c,I)$ with respect to $\omega$,
so in particular the pair $(\omega,H)$ satisfies the following
equations:
\begin{equation}
\label{eq:extremalpairsplit} \left. \begin{array}{l}
d_H^* F_H=0\\
L_{\eta_{S_\omega}} J=0
\end{array}\right \}
\end{equation}
Here, $d_H^*F_H=0$ is the Yang--Mills equation, which is equivalent to
\begin{equation}
\label{eq:dHlambdaF}
d_H\Lambda_{\omega}F_H=0
\end{equation}
by the K\"{a}hler identities (see e.g.~\cite[Proposition~3]{D3}), and
$\eta_{S_\omega}$ is the Hamiltonian vector field of the scalar
curvature $S_\omega$ over $(X,\omega)$, so $L_{\eta_{S_\omega}} J=0$
is the condition that $\omega$ is an extremal metric on $(X,J)$.

If one attempts to generalize Theorem~\ref{thm:DeformationCYMeq1} to
the weak coupling limit, one observes that
Proposition~\ref{prop:extremal-solution} cannot be used for
$\alpha=0$, but the system of equations~\eqref{eq:extremalpairsplit}
can be viewed as an adiabatic limit of equation~\eqref{eq:extremalpair}. In fact, a pair
$b_\lambda\defeq(\lambda\omega,H)$ satisfies~\eqref{eq:extremalpair}
with coupling constant $\alpha$, for a real number $\lambda>0$, if and
only if
\begin{equation}
\label{eq:extremalpairalpha}
4\alpha d_H\Lambda_{\omega}F_H + \lambda^{-1}\eta_\lambda \lrcorner F_H = 0,
\quad
L_{\eta_\lambda} J = 0,
\end{equation}
where $\eta_\lambda$ is the Hamiltonian vector field of
$S_{\alpha/\lambda}(\omega,H)$
over $(X,\omega)$, and~\eqref{eq:extremalpairsplit} is the formal
limit of~\eqref{eq:extremalpairalpha} when
$\lambda\to\infty$. Hence a strategy to obtain a solution to the
coupled equations~\eqref{eq:CYMeq2} for
$0<\lvert\alpha_1/\alpha_0\rvert\ll 1$ (equivalently, for $\lambda \gg
0$) could be to deform a solution to~\eqref{eq:CYMeq2} for $\alpha=0$
(which is therefore a solution to~\eqref{eq:extremalpairsplit}) to
obtain a solution of~\eqref{eq:extremalpairalpha}. The problem is that
the kernel of the operator $\mathbf{L}_{\alpha,b_\lambda}$ determined
by a solution $b_\lambda$ to the coupled equations~\eqref{eq:CYMeq2}
has a discontinuity in the limit $\lambda \to \infty$. More precisely,
this kernel for finite $\lambda>0$ can be identified with $\LieX_I$ (see
Proposition~\ref{propo:impfunc}), whereas the kernel of
$\mathbf{L}_{\alpha,b}$ in the limit $\lambda \to \infty$ is
\begin{equation}
\label{eq:LalphaKernelgeneral}
\{(\phi,\xi) \in C^\infty(X)\times\Omega^0(\ad E_H)  \,|
\eta_\phi\in\LieH_J,\, d_H^*(d_H \xi + \eta_\phi \lrcorner F_H) = 0\}
\end{equation}
(this follows directly from~\eqref{eq:Dalphapaoperator}). This
discontinuity causes serious technical problems when one attempts to
use this this strategy within the approach of LeBrun \& Simanca.

The source of this difficulty is related to the vanishing of the
factor $4\alpha_1$ multiplying the HYM term in the moment maps $\mu_b$
when $\alpha_1=0$ (see~\eqref{eq:thm-muX}). One way to get around this
problem is to apply the approach of LeBrun \& Simanca to the operator
obtained by dropping this factor in the moment map operator
$\operatorname{T}_\alpha$.
Fix an integer $k>n$ and keep the notation
of~\secrefs{sec:Defmmapoperator}, \ref{sec:ANdeformingsolutions}.
Then the resulting modified moment map operator is
\begin{equation}
\label{eq:Balphamap}
\begin{gathered}
\xymatrix @R=0ex @C=-9ex {  **[l] 
\operatorname{B}_\alpha\colon\;\;\;\;\;\;
\mathcal{U} \ar[r] & **[r] L^2_{k}(X)\times L^2_{k+2}(\ad E_H)\\ 
**[l](\gamma,\phi,\xi) \ar@{|->}[r] &  **[r]
\(S_\alpha(\widetilde{b}),\Lambda_{\widetilde{\omega}}F_{H,\widetilde{I}} - z_{\widetilde{\Omega}}\),
}
\end{gathered}
\end{equation}
where $\mathcal{U}$ is the open set in~\eqref{eq:subset-U} and
$\widetilde{\omega}, \widetilde{I}$ and $\widetilde{b}$ are given
by~\eqref{eq:omega-deformations} and~\eqref{eq:btilde}.

As we will see below, this modification on the moment map operator
within the approach of LeBrun \& Simanca produces the following
modified extremality condition
(cf.~\eqref{eq:extremalpair}).

\begin{definition}
A pair $b=(\omega,H)\in B_\Omega$ is called \emph{extremal} with
coupling constant $\alpha$ \emph{in the weak coupling limit} if it
satisfies the equations
\begin{equation}
\label{eq:extremalpairsplit0} \left. \begin{array}{l}
d_H^* F_H = 0\\
L_{\eta_\alpha(b)} J = 0
\end{array}\right \},
\end{equation}
where $\eta_\alpha(b)$ is the Hamiltonian vector field of
$S_\alpha(b)$ over $(X,\omega)$. 
\end{definition}

Note that the system of equations~\eqref{eq:extremalpairsplit0}
becomes~\eqref{eq:extremalpairsplit} when $\alpha=0$, while for
arbitrary $\alpha$ any solution to the coupled
equations~\eqref{eq:CYMeq2} is an extremal pair in the weak coupling
limit (see~\eqref{eq:dHlambdaF}). To obtain a partial converse, define
the characters
\begin{equation}
\label{eq:Futaki-0,infty}
\cF_{0,\Omega},\,\cF_{\infty,\Omega}\colon\Lie\Aut(E^c,I)\lto\CC,
\end{equation}
as the $\alpha$-Futaki characters of the K\"{a}hler class $\Omega$ for
$(\alpha_0,\alpha_1)$ equal to $(1,0)$ and $(0,1)$
in~\eqref{eq:alphafutakismooth}, respectively.
By~\eqref{eq:alphafutakismooth}, up to a multiplicative factor,
$\langle\cF_{0,\Omega},\zeta\rangle$ is the Futaki
character~\cite{Ft0} of the K\"ahler class $\Omega$ on $(X,J)$
evaluated at $p(\zeta)$, where $p$ is the map
in~\eqref{eq:Ext-Lie-group-holmorphic}.
It is also clear from~\eqref{eq:alphafutakismooth} that the existence
of a solution to the coupled equations~\eqref{eq:CYMeq2} does not
necessarily imply the vanishing of $\cF_{0,\Omega}$ or
$\cF_{\infty,\Omega}$.

\begin{proposition}
\label{prop:extremal-solution-weak}
A solution $b\in B_\Omega$ of~\eqref{eq:extremalpairsplit0} is a solution to the coupled
equations~\eqref{eq:CYMeq2} if $\cF_{0,\Omega} =
\cF_{\infty,\Omega} =0$ and the vector field $\eta_\alpha(b)$ over $X$
can be lifted to a holomorphic vector field over the total space of
$(E^c,I)$.
\end{proposition}

\begin{proof}
By~\eqref{eq:dHlambdaF}, $\Lambda_{\omega}F_H$ is
a vertical holomorphic vector field on the total space of $(E^c,I)$,
i.e. $\Lambda_{\omega}F_H\in\LieG_I$. Now, if
$\cF_{\infty,\Omega}=0$, then $H$ is HYM with respect to $\omega$,
because in this case, by~\eqref{eq:alphafutakismooth} we obtain
\[
\|\Lambda_{\omega} F_H - z_\Omega\|_\omega^2
= - \langle\cF_{\infty,\Omega},\Lambda_{\omega} F_H - z_\Omega\rangle = 0.
\]
Moreover, if $\cF_{0,\Omega} = \cF_{\infty,\Omega} =0$ and
$\eta_\alpha(b)=p(\zeta)$ for a holomorphic vector field $\zeta$ on
$(E^c,I)$, then by a straightforward computation
using~\eqref{eq:alphafutakismooth}, we obtain
\[
\|S_\alpha(b)-\hat S_\alpha\|_\omega^2 = \langle\cF_{0,\Omega},\eta_\alpha(b)\rangle + \alpha \langle\cF_{\infty,\Omega},\zeta\rangle + \alpha \langle\theta_H \zeta,\Lambda_{\omega} F_H - z_\Omega\rangle = 0,
\]
where $\hat S_\alpha = \int_X S_\alpha(b)\omega^{[n]}/\Vol_\Omega$,
so $b$ is a solution to the coupled equations~\eqref{eq:CYMeq2}.
\end{proof}

Extremal pairs in the weak coupling limit enjoy the same
good regularity properties:

\begin{lemma}
\label{lem:reg-extremalpairs-weak}
Let $(\omega,H)$ be a solution of~\eqref{eq:extremalpairsplit0} such that $\omega$ is a
K\"{a}hler form of class $C^2$ on $(X,J)$ and $H$ is a section of
$E^c/G$ of class $C^4$.
Then both $\omega$ and $H$ are smooth.
\end{lemma}

\begin{proof}
This follows exactly as Lemma~\ref{lem:reg-extremalpairs}.
\end{proof}

We define now a linear differential operator which is closely related
to the linearization of $\operatorname{B}_\alpha$
(see~\eqref{eq:Balphamap}) when $\alpha=0$ and which will play the
role in the weak coupling limit of the Lichnerowicz
operator~\eqref{eq:P_omega} in the study of the cscK equation or the
operator $\mathbf{L}_{\alpha,b}$ defined
in~\secref{sec:ANdeformingsolutions} away from the weak coupling
limit.
This linear differential operator is
\begin{equation}
\label{eq:Coperator}
\begin{gathered}
\xymatrix @R=0ex @C=-18ex {**[l]
\mathbf{C} \colon
L^2_{k+4}(X)\times L^2_{k+4}(\ad E_H) \ar[r] &
**[r] L^2_{k}(X) \times L^2_{k+2}(\ad E_H)
\\ **[l]
(\phi,\xi)  \ar@{|->}[r] & **[r]
\(\operatorname{P}^*\operatorname{P}\phi,d_H^*(d_H\xi + \eta_{\phi}\lrcorner F_H)\),
}
\end{gathered}
\end{equation}
where $\operatorname{P}$ is defined as in~\eqref{eq:P_omega}.
It is easy to see (cf. Proposition~\ref{propo:impfunc}) that the
operator $\mathbf{C}$ is elliptic and self-adjoint with respect to the $L^2$-inner product
$\langle\cdot,\cdot\rangle_\omega$ given
by~\eqref{eq:pairing-product}.

It can be shown as in the proof of Proposition~\ref{propo:DefTalpha}
that $\operatorname{B}_\alpha$ is well-defined and $C^1$ and
that its Fr\'echet derivative at the origin $(0,0,0)$ when $\alpha=0$
is given by
\begin{equation}
\label{eq:Balphamapder}
\delta\operatorname{B}_0(\dot{\gamma},\dot{\phi},\dot{\xi})
=\mathbf{C}(\dot{\phi},\dot{\xi})
+((dS_\omega,d\dot{\phi})_\omega,0)
+ \delta_{\dot{\gamma}}\operatorname{B}_0,
\end{equation}
where $\delta_{\dot{\gamma}}\operatorname{B}_0$ is the directional
derivative of $\operatorname{B}_0$ at the origin in the direction
$(\dot{\gamma},0,0)$ (cf.~\eqref{eq:Talphamapder}).

To proceed as in~\secref{sec:Defweakcoupling} following the approach
of LeBrun \& Simanca, we need to consider the restriction of
$\operatorname{B}_\alpha$ and $\mathbf{C}$ to suitable
subspaces of the Sobolev spaces. Let
\[
L^2_{k}(X)^{\cH_J} \subset L^2_k(X) \textrm{ and }
L^2_{k}(\ad E_H)^{\cG_I}\subset L^2_{k}(\ad E_H)
\]
be the closed subspaces consisting of $\cH_J$-invariant
functions and $\cG_I$-invariant sections, respectively
(cf.~\eqref{eq:inv-Sobolev-spc}) and
\[
\mathcal{V}' = \mathcal{U}\cap
\(\mathcal{H}^{1,1}(X)\times L^2_{k+4}(X)^{\cH_J}\times L^2_{k+4}(\ad E_H)^{\cG_I}\)
\]
(cf.~\eqref{eq:subset-V}). By restriction of~\eqref{eq:Balphamap}
and~\eqref{eq:Coperator}, we obtain well-defined maps
\begin{subequations}
\label{eq:B-C-operators-inv}
\begin{align}
\label{eq:Boperator-inv}
\hat{\operatorname{B}}_\alpha \colon & \mathcal{V}'\lto L^2_{k}(X)^{\cH_J} \times L^2_{k+2}(\ad E_H)^{\cG_I},\\
\label{eq:Coperator-inv}
\hat{\mathbf{C}}\colon &
L^2_{k+4}(X)^{\cH_J}\times L^2_{k+4}(\ad E_H)^{\cG_I}
\lto L^2_{k}(X)^{\cH_J} \times L^2_{k+2}(\ad E_H)^{\cG_I},
\end{align}
\end{subequations}
where $\hat{\operatorname{B}}_\alpha$ is $C^1$ and
$\hat{\mathbf{C}}$ is a linear elliptic operator
(cf.~\eqref{eq:Talphamap-Lalpha-invariant}).

Note that in the construtions~\eqref{eq:B-C-operators-inv} we have
used the subspace $L^2_{k}(\ad E_H)^{\cG_I}\subset L^2_{k}(\ad E_H)$
rather than the possibly smaller subspace $L^2_{k}(\ad E_H)^{\cX_I}$
which appeared in~\eqref{eq:Talphamap-Lalpha-invariant}. In practice,
we could say that the exact
sequence~\eqref{eq:Ext-Lie-group-holmorphic}
in~\secref{sec:ANdeformingsolutions} degenerates to the trivial
extension
\begin{equation}
\label{eq:trivial-ext}
1 \to \cG_I \lto \cH_J\times\cG_I \lto \cH_J \to 1
\end{equation}
in the weak coupling limit $\alpha\to 0$. In particular, the centre
$\mathfrak{z}(\LieX_I)$ of $\LieX_I$ (see~\eqref{eq:Pbold}) is now
replaced by the centre $\mathfrak{z}(\LieH_J)\oplus
\mathfrak{z}(\LieG_I)$ of the Lie algebra of $\cH_J\times\cG_I$
and $\mathbf{P}_{(\gamma,\phi,\xi)}$ (see~\eqref{eq:Pbold}) is
replaced by
\begin{equation}
\label{eq:Qbold}
\begin{gathered}
\xymatrix @R=0ex @C=-22ex {**[l]
\mathbf{Q}_{(\gamma,\phi,\xi)} 
\colon
\RR\times\mathfrak{z}(\LieH_J)\oplus\mathfrak{z}(\LieG_I) \ar[r] & **[r]
L^2_{k+3}(X)^{\cH_J} \times L^2_{k+2}(\ad E_H)^{\cG_I} 
\\**[l]
(t,w,v) \ar@{|->}[r] &  **[r]
\(\mathbf{G}d^*(w\lrcorner\widetilde{\omega})+t,v\),\\
}
\end{gathered}
\end{equation}
with $(\gamma,\phi,\xi)\in\mathcal{V}'$.

\begin{lemma}
\label{lemma:PTalpha2}
Let $(\gamma,\phi,\xi)\in\mathcal{V}'$. If
$\hat{\operatorname{B}}_\alpha(\gamma,\phi,\xi)\in\Im\mathbf{Q}_{(\gamma,\phi,\xi)}$,
then $\widetilde{b}=(\widetilde{\omega},\widetilde{H})$ is a solution of~\eqref{eq:extremalpairsplit0}.
\end{lemma}

\begin{proof}
This follows exactly as part (2) of Lemma~\ref{lemma:PTalpha}.
\end{proof}

Since $\mathbf{C}$ has kernel~\eqref{eq:LalphaKernelgeneral} by
elliptic regularity, part (3) of Lemma~\ref{lemma:PTalpha} has no
direct analogue in the weak coupling limit. Lemma~\ref{lemma:PTalpha3}
will provide a suitable replacement of this part of the
lemma.
Let $\langle\cdot,\cdot\rangle_\omega$ be the $L^2$-inner product on $
L^2_{k}(X)^{\cH_J} \times L^2_{k+2}(\ad E_H)^{\cG_I}$
given by~\eqref{eq:pairing-product}. One can prove as
in~\secref{sec:ANdeformingsolutions} that the orthogonal projector
\[
\Pi'_{(\gamma,\phi,\xi)}\colon
L^2_{k}(X)^{\cH_J}\times L^2_{k+2}(\ad E_H)^{\cG_I}
\lto L^2_{k}(X)^{\cH_J}\times L^2_{k+2}(\ad E_H)^{\cG_I}
\]
onto $\Im\mathbf{Q}_{(\gamma,\phi,\xi)}$ varies smoothly with
$(\gamma,\phi,\xi)\in\mathcal{V}'$ and, by continuity, there exists an
open neighbourhood $\mathcal{V}'_0\subset\mathcal{V}'$ of the origin
such that
\[
\ker(\Id - \Pi'_{(\gamma,\phi,\xi)}) = \ker(\Id-\Pi_0)\circ(\Id-\Pi'_{(\gamma,\phi,\xi)})
\]
for any $(\gamma,\phi,\xi) \in \mathcal{V}'_0$ (cf.~\eqref{eq:P0Pphi}).

For any pair of non-negative integers $(l,m)$, let $I'_{l,m}\subset
L^2_{l}(X)^{\cH_J} \times L^2_{m}(\ad E_H)^{\cG_I}$ be the orthogonal
complement of $\Im\mathbf{Q}_{0}$. Define
\[
\mathcal{W}' = \mathcal{V}'_0 \cap (\mathcal{H}^{1,1}(X) \times I'_{k+4,k+4}).
\]

\begin{lemma}
\label{lemma:PTalpha3}
The induced map $\hat{\mathbf{C}}\colon I'_{k+4,k+4}\lto I'_{k,k+2}$ is
an isomorphism.
\end{lemma}

\begin{proof}
This map is well-defined because $\Im\mathbf{Q}_0\subset
\ker\(\operatorname{P}\oplus d_H\)$.
If $\hat{\mathbf{C}}(\phi,\xi)=0$ for some $(\phi,\xi)\in
I'_{k+4,k+4}$, then $\operatorname{P}^*
\operatorname{P}\phi=0$, so
$\operatorname{P}\phi = 0$, which implies $\phi=0$, and $\hat{\mathbf{C}}(\phi,\xi)=0$
means $d_H^*d_H\xi=0$, so $d_H\xi=0$, which implies $\xi=0$. Thus $\hat{\mathbf{C}}$ is injective.
Finally, $\hat{\mathbf{C}}$ is surjective because so is
$\operatorname{P}^* \operatorname{P}\oplus
d_H^*d_H$.
\end{proof}

Define now a LeBrun--Simanca map~\cite[\S 5]{LS2}
\begin{equation}
\label{eq:Bmap}
\begin{gathered}
\xymatrix @R=0ex @C=-15ex {**[l]
\mathbf{B}_\alpha \colon
\mathcal{W}' \ar[r] & **[r] I'_{k,k+2}
\\**[l]
(\gamma,\phi,\xi) \ar@{|->}[r] &  **[r]
(\Id - \Pi'_0)\circ(\Id - \Pi'_{(\gamma,\phi,\xi)})\circ \hat{\operatorname{B}}_\alpha(\gamma,\phi,\xi).
}
\end{gathered}
\end{equation}
As $\mathbf{B}_\alpha$ is the composition of $C^1$-maps, it is
$C^1$. Using Lemma~\ref{lemma:PTalpha3} and~\cite[Lemma~1]{LS2}, we
can see that its directional derivative at the origin in the direction
$(0,\dot{\phi},\dot{\xi})$ for $\alpha=0$ is
\begin{equation}
\label{eq:Balphamapder2}
\delta_{(\dot{\phi},\dot{\xi})}\mathbf{B}_0
=(\Id - \Pi_0')\hat{\mathbf{C}}(\dot{\phi},\dot{\xi})
=\hat{\mathbf{C}}(\dot{\phi},\dot{\xi}),
\end{equation}
for all $(\dot{\phi},\dot{\xi})\in I'_{k+4,k+4}$.

We can now prove the two main results of~\secref{sec:Defweakcoupling}.

\begin{theorem}
\label{thm:DeformationCYMeq3}
Suppose that $\omega$ is an extremal K\"{a}hler metric on $(X,J)$ with
$\Omega=[\omega]$ and $H$ is a Yang--Mills reduction of $(E^c,I)$ with
respect to $\omega$. Then $(0,\Omega)$ has an open neighbourhood
$U\subset\RR\times H^{1,1}(X,\RR)$ such that for all
$(\widetilde{\alpha},\widetilde{\Omega})\in U$ there exists an
extremal pair $(\widetilde{\omega},\widetilde{H})$ with coupling
constant $\widetilde{\alpha}$ in the weak coupling limit such that
$[\widetilde{\omega}]=\widetilde{\Omega}$.
\end{theorem}

\begin{proof}
This follows as Theorem~\ref{thm:DeformationCYMeq1}, combining~\eqref{eq:Balphamapder2} with Lemma~\ref{lemma:PTalpha3} and the
implicit function theorem, and then using Lemmas~\ref{lemma:PTalpha2}
and~\ref{lem:reg-extremalpairs-weak}.
\end{proof}

In the following theorem, we say that a reduction $H \in \cR$ is
irreducible if its Chern connection is irreducible, that is, if its
isotropy group inside the gauge group $\cG_H$ of $E_H$ is minimal---
the centre of $G$ (see~\secref{sec:CeqscalarKal-K} and also~\cite[\S
4.2.2]{DK}).

\begin{theorem}
\label{thm:DeformationCYMeq4}
Assume that there is a cscK metric $\omega$ on $(X,J)$ with cohomology
class $\Omega$ and there are no non-zero Hamiltonian Killing vector fields
on $X$. Then
\begin{enumerate}
\item[\textup{(1)}]
If $(E^c,I)$ admits an irreducible HYM reduction $H$ with respect to $\omega$, then $(0,\Omega)$ has
an open neighbourhood $U\subset\RR\times H^{1,1}(X,\RR)$ such that for
all $(\widetilde{\alpha}_1,\widetilde{\Omega})\in U$, there exists a
solution $(\widetilde{\omega},\widetilde{H})$ to the coupled
equations~\eqref{eq:CYMeq2} with coupling constant
$\widetilde{\alpha}=(1,\widetilde{\alpha}_1)$ and
$\widetilde{\omega}\in\widetilde{\Omega}$.
\item[\textup{(2)}]
If $(E^c,I)$ admits a HYM reduction $H$ with respect to $\omega$, then there exists
$\epsilon>0$ such that for all $\widetilde{\alpha}_1\in\RR$ with
$-\epsilon<\widetilde{\alpha}_1<\epsilon$, there exists a solution
$(\widetilde{\omega},\widetilde{H})$ to the coupled equations~
\eqref{eq:CYMeq2} with coupling constants $(1,\widetilde{\alpha}_1)$
and $\widetilde{\omega}\in\widetilde{\Omega}$.
\end{enumerate}
\end{theorem}

\begin{proof}
Since HYM reductions are Yang--Mills,
Theorem~\ref{thm:DeformationCYMeq3} implies that for all
$(\widetilde{\alpha},\widetilde{\Omega})$ in a neighbourhood
$U\subset\RR\times H^{1,1}(X,\RR)$ of $(0,\Omega)$, there exists an
extremal pair $(\widetilde{\omega},\widetilde{H})$ with coupling
constant $\widetilde{\alpha}$ in the weak coupling limit with
$[\widetilde{\omega}]=\widetilde{\Omega}$ and $\widetilde{H}$
irreducible.

Part (1) follows now since the function
$S_\alpha(\widetilde{\omega},\widetilde{H})$ defined
by~\eqref{eq:Salpha} is constant on $X$ for any extremal pair
$(\widetilde{\omega},\widetilde{H})$, as $\LieH_J=0$ and, furthermore,
the vertical real-holomorphic vector field on $(E^c,I)$ defined by
$\Lambda_{\widetilde{\omega}}F_{\widetilde{H}}$ is in $\mathfrak{z}$,
as $\widetilde{H}$ is irreducible.

Part (2) follows from Theorem~\ref{thm:DeformationCYMeq3} and
Proposition~\ref{prop:extremal-solution-weak}, because
$\cF_{0,\Omega}=\cF_{\infty,\Omega}=0$
by~\eqref{eq:alphafutakismooth}, as $\LieH_J=0$ and $(E^c,I)$
admits a HYM reduction $H$ with respect to $\omega$.
\end{proof}

\section{Examples and cscK metrics on ruled manifolds}
\label{chap:examples}

This section contains some examples of solutions to the coupled equations
\eqref{eq:CYMeq00}. In \secref{sec:example4} we also discuss how the existence
of solutions in the limit case $\alpha_0 = 0$ can be applied, using results of
Y. J. Hong in \cite{Ho2}, to obtain cscK metrics on ruled manifolds.

\subsection{Projectively flat bundles}
\label{sec:example1}

Let $(E^c,I)$ be a holomorphic principal $G^c$-bundle over a compact
complex manifold $X$. We fix a maximal compact subgroup $G \subset
G^c$ and a $G$-invariant metric $(\cdot,\cdot)$ on
$\mathfrak{g}$. Suppose that there exists a $G$-reduction $H$ on $E^c$
and a K\"ahler metric $\omega$ on $X$ satisfying
\begin{equation}
\label{eq:CYMeqdecoupled}
\left. \begin{array}{l}
F_H  = z \, \omega\\
S_\omega  = \hat S
\end{array}\right \},
\end{equation}
where $F_H$ is the curvature of the Chern connection of $H$, $z$ is
the element of $\mathfrak{z}$ (see~\eqref{eq:centre-z}) given
by~\eqref{eq:z(Omega,E)} and $\hat S \in \RR$. It is then
straightforward that the pair $(\omega,H)$ provides a solution of the
coupled equations \eqref{eq:CYMeq00}. Note that the first equation in
\eqref{eq:CYMeqdecoupled} implies that the $G$-bundle $E_H$
corresponding to $H$ is projectively flat, i.e. it is given by a
representation $\pi_1(X) \to G/Z(G)$, where $Z(G)$ denotes the centre
of $G$. Moreover, it implies the following topological constraint
\begin{equation}
\label{eq:topconstraint}
[z \wedge F_H] = |z|^2[\omega] \in H^2(X,\RR),
\end{equation}
where $[F_H \wedge z]$ is the Chern--Weil class associated to the
$G$-invariant linear form $(\cdot,z)$ on $\mathfrak{g}$. We discuss
now some examples of solutions of \eqref{eq:CYMeqdecoupled}. We apply
Theorem~\ref{thm:DeformationCYMeq4}(1) to perturb the K\"ahler class
of the given solution in order to obtain new solutions that do not
satisfy the topological constraint \eqref{eq:topconstraint}.

\begin{example}
\label{example0}
Let $X$ be a compact Riemann surface. Then the coupled equations
\eqref{eq:CYMeq00}, for a $G$-reduction $H$ on $E^c$ and a K\"ahler
metric $\omega$ on $X$, split into the system in separated variables
\eqref{eq:CYMeqdecoupled}, since $\dim_\CC X = 1$ and the term $(F_H
\wedge F_H)$ vanishes. Then the solutions of the coupled equations
\eqref{eq:CYMeq00} are given by pairs $(\omega,H)$, where $\omega$ is
a cscK metric and $H$ is a $G$-reduction such that its Chern
connection is Hermitian--Yang--Mills \eqref{eq:HYM}. Due to the
Narasimhan and Seshadri Theorem~\cite{D2}, and Ramanathan's
generalization~\cite{ramanathan:1975}, examples of solutions of the
coupled equations~\eqref{eq:CYMeq00} are given by polystable
$G^c$-bundles over $X$.
\end{example}

\begin{remark}
In~\cite{Pd}, Pandharipande used Geometric Invariant Theory to
compactify the moduli space of pairs $(X,F)$ consisting of a smooth
algebraic curve $X$ of genus $g>1$, polarised by a multiple of its
canonical bundle, and a semistable vector bundle $F$ over the
curve. By~\cite[Proposition~8.2.1]{Pd}, such a pair is GIT stable if
and only if $E$ is Mumford stable. An interesting issue is that this
decoupling phenomenon for the stability condition of a pair $(X,F)$ is
reflected in the decoupling of the equations~\eqref{eq:CYMeq00}, as
already observed in Example~\ref{example0}. In fact, combining the
Narasimhan--Seshadri Theorem with the uniformization Theorem on
Riemann surfaces, it follows that any GIT stable pair $(X,F)$ in
Pandharipande's construction, with $X$ smooth, admits one and only one
(irreducible) solution of \eqref{eq:CYMeq00} with K\"ahler class equal
to the class of the polarisation. This gives some evidence to the
claim that a Hitchin--Kobayashi correspondence for
equations~\eqref{eq:CYMeq00} exists in arbitrary dimensions, as
conjectured by the authors~\cite{GF1}. An important difference with
the curve case is that of course in higher dimensions one expects that
the stability condition equivalent to the existence of solutions will
involve conditions on the base manifold as well. In~\cite{GF1}, a new
notion of stability for degree zero bundles and polarised varieties
has been defined. We hope to address the relation between this
stability condition and the existence of solutions
of~\eqref{eq:CYMeq00} in future work.
\end{remark}

Let $(X,L)$ be a compact polarised manifold of complex dimension
$n$. Suppose that there exists a cscK metric
\[
\omega = \frac{\imag}{2\pi}{F_H} \in c_1(L),
\]
where $F_H$ is the curvature of a Hermitian metric $H$ on $L$. Then
$(\omega,H)$ is a solution of \eqref{eq:CYMeqdecoupled}, and hence a
solution of~\eqref{eq:CYMeq00}. Since $H$ is trivially an irreducible
HYM metric with respect to $\omega$, if there are no non-zero
Hamiltonian holomorphic vector fields on $X$, we can apply
Theorem~\ref{thm:DeformationCYMeq4} $(1)$ obtaining solutions
of~\eqref{eq:CYMeq00} with non-zero ratio of the coupling constants
and K\"ahler class close to $[\omega]$ in $H^{1,1}(X,\RR)$.

\begin{example}
\label{example1}
Let $X$ be a degree four hypersurface of $\PP^3$ and set $L =
\cO_X(1)$. Then $X$ is a K3 surface and, by Yau's solution~\cite{Yau}
of the Calabi conjecture (see e.g.~\cite{Be}), there exists a unique
K\"ahler Ricci flat metric $\omega \in c_1(L)$. Since $(X,\omega)$ is
K\"ahler Ricci flat, any holomorphic vector field on $X$ is
$\omega$-parallel and so $\Lie(\Aut X)$ contains no non-zero
Hamiltonian holomorphic vector fields. Therefore, applying
Theorem~\ref{thm:DeformationCYMeq4}(1), we obtain solutions
of~\eqref{eq:CYMeq00} with non-zero ratio of the coupling constants
$\alpha_1/\alpha_0$ and K\"ahler class $\widetilde{\Omega}$ close to
$\Omega=[\omega]$ in $H^{1,1}(X,\RR)$. As the dimension of
$H^{1,1}(X,\RR)$ is $20$, we can assume that $\widetilde{\Omega}$ is
not contained in the real line spanned by $\Omega$, and so it is not
obvious \emph{a priori} that such a K\"ahler class contains a solution
of \eqref{eq:CYMeq00} for our choice of manifold $X$ and bundle $L$.
\end{example}

When $(X,\omega)$ is a flat K\"ahler torus, we can relax condition \eqref{eq:CYMeqdecoupled} and assume that $E$ is an arbitrary projectively flat Hermitian bundle over $X$.

\begin{example}
Let $X \cong \CC^n/\Lambda_X$ be a complex torus given by a lattice
$\Lambda_X$ in $\ZZ^{2n}$ and endowed with a flat K\"ahler metric $\omega$. Examples of holomorphic vector bundles $E$
over $X$ admitting a projectively flat Hermitian metric $H$ are given
by representations of a central extension of $\Lambda_X$ into $\U(r)
\subset \GL(r,\CC)$. 
Suppose that $E$ is given by an irreducible representation of $\Lambda_X$ and take a projectively flat Hermitian metric $H$ on $E$, with curvature $\tau \Id$. By a conformal change on $H$, we can assume that $\tau$ is harmonic, and hence it is constant with respect to the natural coordinates in the torus.
Then $(\omega,H)$ is a solution to the coupled equations \eqref{eq:CYMeq00} for  arbitrary value of the coupling
constants $\alpha_0$ and $\alpha_1$.
\end{example}

\begin{remark}
In~\cite{ST}, Schumacher \& Toma constructed a moduli space of
(non-uniruled) polarised K\"ahler manifolds equipped with stable
vector bundles, using versal deformations. This moduli space is
endowed with a K\"ahler metric, provided that the cohomological
constraint~\eqref{eq:topconstraint} is satisfied, the base manifold
$X$ is K\"ahler--Einstein and the bundle is projectively flat. The
gauge-theoretic equations corresponding to this moduli construction
are therefore equivalent to \eqref{eq:CYMeqdecoupled}, whose solutions
are in particular solutions to the coupled
equations~\eqref{eq:CYMeq00}. Note here that the cscK equation and the
K\"{a}hler--Einstein equation are equivalent, by Hodge theory, if the
class of the polarisation is a multiple of $c_1(X)$.
\end{remark}

In the examples of~\secref{sec:example1}, the coupled
equations~\eqref{eq:CYMeq00} admit decoupled solutions arising from
the system in separated variables \eqref{eq:CYMeqdecoupled}. There is
a geometric interpretation for this in terms of the extended gauge
group $\cX$ in~\eqref{eq:Ext-Lie-groups} associated to a solution
$(\omega,H)$ to $F_H = z\omega$ and the moment map
interpretation of~\eqref{eq:CYMeq00} in~\secref{chap:Ceq}. Namely, the
Chern connection $A$ of $H$ determines a Lie algebra splitting of the
short exact sequence
\[
0 \to \LieG \lto \LieX \lto \LieH \to 0
\]
(see~\eqref{eq:Ext-Lie-alg-2}). The splitting is given by the Lie
algebra homomorphism
\begin{equation}
\label{eq:splittingLiecX}
\xymatrix @R=0ex @C=-9ex { **[l]
\Phi\colon \LieH \cong C^{\infty}_0(X) \ar[r]
&**[r]
\LieX
\\ **[l]
\phi \ar@{|->}[r] 
& **[r] 
\theta_A^{\perp}\eta_\phi - \phi z,
}
\end{equation}
(see~\eqref{eq:LieH}), where $\eta_\phi\lrcorner \omega = d\phi$ and
$\theta_A^\perp$ is the horizontal lift with respect to the connection
$A$. To see this, note that
\begin{align*}
[\Phi(\phi_1),\Phi(\phi_2)] & = [\theta_A^{\perp}\eta_{\phi_1} - \phi_1 z,\theta_A^{\perp}\eta_{\phi_2} - \phi_2 z]\\
& = \theta_A^{\perp}[\eta_{\phi_1},\eta_{\phi_2}] - \{\phi_1,\phi_2\} z + (F_A - z\omega)(\eta_{\phi_1},\eta_{\phi_2})\\
& = \Phi(\{\phi_1,\phi_2\}) + (F_A - \omega z)(\eta_{\phi_1},\eta_{\phi_2}),
\end{align*}
where $\{\phi_1,\phi_2\}$ is the Poisson bracket in $C^{\infty}_0(X)$
given by $\omega$. Note that this homomorphism does not extend in
general to the Lie algebra of the group of diffeomorphisms of
$X$. Therefore, when $\dim_\CC X = 1$ or $E$ is projectively flat, the
coupled system \eqref{eq:CYMeq00} may have `decoupled' solutions due
to the fact that $\LieX$ is a semidirect product of $\LieG$ and
$\LieH$.

\subsection{Homogeneous bundles over homogeneous K\"ahler manifolds}
\label{sec:example2}

For the basic material on this topic we refer to~\cite{Be}
and~\cite{Ko}. Let $X$ be a compact homogeneous K\"ahlerian manifold
(i.e. admitting a K\"ahler metric) of a compact group $G$. In other
words, $X = G/G_o$, for a closed subgroup $G_o \subset G$, equipped
with the canonical $G$-invariant complex structure
(see~\cite[Remark~8.99]{Be}). Then homogeneous holomorphic vector
bundles $E$ of rank $r$ over $X$ are in one-to-one correspondence with
representations of $G_o$ in $\GL(r,\CC)$. For any invariant K\"ahler
metric $\omega$ on $X$, there exists a unique $G$-invariant
Hermitian--Yang--Mills unitary connection $A$, provided that the
representation inducing $E$ is irreducible
(see~\cite[Proposition~6.1]{Ko}). Moreover, for any such choice of
invariant metric and connection, the scalar curvature $S_\omega$ and
the function $\Lambda_\omega^2\tr (F_A \wedge F_A)$ on $X$ are
$G$-invariant and hence constant. It hence turns out that $A$
satisfies the system of equations
\begin{equation}
\label{eq:ext-HYM}
\left. \begin{array}{l}
\Lambda_{\omega} F_A = \imag \lambda \Id\\
\Lambda^2_{\omega} \tr (F_A \wedge F_A) = - \frac{4\hat{c}}{(n-1)!}
\end{array}\right \},
\end{equation}
where $\hat c \in \RR$ is as in \eqref{eq:constant-c-hat} and $\lambda
\in \RR$ is determined by the first Chern class of $E$ and
$[\omega]$.
Equations~\eqref{eq:ext-HYM} corresponds to the limit
\[
  \alpha_0 \to 0
\]
in~\eqref{eq:CYMeq00}. Fix a pair of arbitrary coupling constants
$\alpha_0, \alpha_1 >0$ and a homogeneous holomorphic vector bundle
$E$ over $X$ associated to an irreducible representation. Then any
K\"ahler class on $X$ determines a unique $G$-invariant solution
$(\omega,A)$ to the coupled equations with coupling constants
$\alpha_0$ and $\alpha_1$. To see this, note that each de Rham class
on $X$ (in particular, each K\"ahler class) contains a unique
$G$-invariant representative, obtained from an arbitrary
representative by averaging. Trivially, the scalar curvature of any
$G$-invariant K\"ahler metric is constant. Therefore, the unique
$G$-invariant solution of \eqref{eq:CYMeq00} arises as a simultaneous
solution of the cscK equation and \eqref{eq:ext-HYM}, corresponding to
the limit cases $\alpha_0 = 0$, and $\alpha_1 = 0$.

\begin{example}
Let $(X,\omega)$ be a compact homogeneous K\"ahler--Einstein surface
$G/G_0$. By~\cite[Corollary~8.98]{Be}, this means that $X$ is a
complex torus or it is simply connected. Let $E$ be a homogeneous
vector bundle on $X$ induced by an irreducible representation of $G_o$ in
$\SU(r)$, with induced $G$-invariant Hermitian metric $H$ and
$G$-invariant unitary connection $A$. Then the pair $(\omega,A)$
satisfies the system of equations
\begin{equation}
\label{eq:KEYMeq} \left. \begin{array}{l}
F_A^{+} = 0\\
\alpha_0 (\rho_\omega - c'\omega) = \alpha_1 (2(\Lambda_{\omega} F_A)
\wedge F_A - \Lambda_\omega (F_A \wedge F_A) - c'' \omega)
\end{array}\right \},
\end{equation}
for real numbers $c',c''$, where $\rho_\omega$ is the Ricci form of
$\omega$ and $F_A^+ = 0$ is the Anti-Self-Duality equation for the
connection $A$. To prove this, note that $A$ is HYM and
\begin{align*}
2(\Lambda_{\omega} F_A) \wedge F_A - \Lambda_\omega (F_A \wedge F_A) & = - \Lambda_\omega (F_A \wedge F_A)\\
& = \Lambda_\omega (|F_A|^2 \omega^2)\\
& = |F_A|^2 \omega,
\end{align*}
(see~\eqref{eq:identitysquaredFA}), where $|\cdot|$ is the pointwise
norm with respect to $\omega$. Hence $(\omega,A)$
satisfies~\eqref{eq:KEYMeq} because the function $|F_A|^2$ is constant
over $X$ by invariance.
Observe that the system~\eqref{eq:KEYMeq} is stronger
than~\eqref{eq:CYMeq00}. Indeed, it can be readily checked
from~\cite[Proposition~9.61]{Be} that if $(\omega,A)$ satisfies
\eqref{eq:KEYMeq}, then the associated invariant Riemannian metric on
the total space of the frame $\U(r)$-bundle of $(E,H)$ over $X$,
constructed as in~\secref{sec:CeqscalarKal-K}, is Einstein, and
therefore $(\omega,A)$ satisfies~\eqref{eq:CYMeq00},
by~\eqref{eq:KK-eq}.
\end{example}

\subsection{Stable bundles and cscK manifolds}
\label{sec:example3}

We supply now some cases where Theorem~\ref{thm:DeformationCYMeq4} can
be applied, obtaining examples of solutions with non-zero ratio of the
coupling constants and fixed K\"ahler class. Starting with a cscK
metric, we check that the new K\"ahler metrics that we obtain are not
cscK. Using the contents of~\secref{sub:ANCeq}, we also give an
explicit Example~\ref{ex:obstruction} in which there cannot exist
solutions to the coupled equations.

\begin{example}
\label{ex:ASD-Donaldson}
Let $X$ be a high degree hypersurface of $\mathbb{P}^3$. By theorems
of Aubin and Yau (see e.g.~\cite[Theorem~11.7]{Be}), there exists
a unique K\"ahler--Einstein metric $\omega \in c_1(X)$ with negative
(constant) scalar curvature. Moreover, $c_1(X) < 0$ implies that the
group of automorphisms of the complex manifold $X$ is discrete
(see~\cite[Proposition~2.138]{Be}). Let $E$ be a smooth
$\SU(2)$-principal bundle over $X$ with second Chern number
\[
k = \frac{1}{8\pi^2}\int_X \tr (F_A \wedge F_A) \in \ZZ,
\]
where $A$ is a connection on $E$. When $k$ is sufficiently large, the
moduli space $M_k$ of Anti-Self-Dual (ASD) connections $A$ on $E$ with
respect to $\omega$ is non-empty
(see~\cite[Sec.~10.1.14]{DK}). Moreover, if $k$ is large enough, $M_k$
is non-compact but admits a compactification. Let $A$ be a connection
that determines a point in $M_k$. Then $A$ is irreducible and so we
can apply Theorem~\ref{thm:DeformationCYMeq4}(1), obtaining solutions
$(\omega_\alpha,A_\alpha)$ of~\eqref{eq:CYMeq00} with $[\omega_\alpha]
= [\omega]$, nonzero values of the coupling constants $\alpha_0$,
$\alpha_1$ and small ratio
\[
\alpha = \frac{\alpha_1}{\alpha_0}.
\]
We claim that if the pointwise norm
\begin{equation}
\label{eq:examplesFAnorm}
|F_{A_0}|_{\omega_0}^2\colon X \lto \RR
\end{equation}
of the initial HYM connection $A_0 = A$ with respect to the
K\"{a}hler--Einstein metric $\omega_0 = \omega$ is not constant, then
$\omega_\alpha$ is not cscK for $0 < \alpha \ll 1$. To see this, note
that $(\omega_\alpha,A_\alpha)$ approaches uniformly to
$(\omega_0,A_0)$ as $\alpha\to 0$ (see
Theorem~\ref{thm:DeformationCYMeq3}) and so
\[
\lim_{\alpha \to 0} \left||F_{A_\alpha}|_{\omega_\alpha}^2 -
  |F_{A_0}|_{\omega_0}^2\right|_{L^{\infty}} = 0.
\]
Hence if $\eqref{eq:examplesFAnorm}$ is not constant, then
$|F_{A_\alpha}|_{\omega_\alpha}^2$ is not constant for small $\alpha$,
so the claim follows from
\[
S_{\omega_\alpha} = \frac{c}{\alpha_0} - \alpha
\Lambda^2_{\omega_\alpha}(F_{A_\alpha} \wedge F_{A_\alpha}) =
\frac{c}{\alpha_0} + \alpha |F_{A_\alpha}|_{\omega_\alpha}^2,
\]
where $c \in \RR$. This last equation is satisfied because
$(\omega_\alpha,A_\alpha)$ is a solution to~\eqref{eq:CYMeq00}. To
choose an ASD connection for which \eqref{eq:examplesFAnorm} is not a
constant, we consider a sequence of ASD connections
$\{A^l\}_{l=0}^\infty$ defining points of $M_k$ and approaching a
point on the boundary of the compactification. When $l\gg 0$, the
connections $A_l$ start bubbling. This bubbling is reflected in the
fact that the function \eqref{eq:examplesFAnorm} becomes more and more
concentrated in a finite number of points of the manifold. Therefore,
eventually, we obtain an ASD irreducible connection for which
\eqref{eq:examplesFAnorm} is not a constant.

To be more precise, recall that any point on the boundary of the
compactification of $M_k$ is given by an ideal connection
(see~\cite[Definition~4.4.1]{DK}), i.e. an unordered $d$-tuple $(p_1,
\ldots, p_d)$ of points on $X$ and a connection $A_\infty$ on
$M_{k-d}$, the moduli space of ASD connections on a suitable smooth
$\SU(2)$-bundle $E_{k-d}$ with second Chern number $k-d$. If $[A_l]
\to [A_\infty]$ as $l\to \infty$, then for any continuous function $f$
on $X$ (see~\cite[Theorem~4.4.4]{DK}),
\begin{equation}
\label{eq:convergencemeasure}
\lim_{l \to \infty} \int_X f \tr (F_{A_l} \wedge F_{A_l}) = \int_X f \tr (F_{A_\infty} \wedge F_{A_\infty}) + 8\pi^2 \sum_{m=1}^{d}f(p_m).
\end{equation}
Take $A_\infty$ in $M_{k -d}$ with $d>0$. If $|F_{A_l}|_{\omega}^2$ is
constant for all $l$, using~\eqref{eq:convergencemeasure} and the
equality
\[
|F_{A_l}|_{\omega}^2 \omega^2 = \tr (F_{A_l} \wedge F_{A_l}),
\]
we obtain that $d = 0$ and hence a contradiction (e.g., in
\eqref{eq:convergencemeasure}, take a sequence $\{f_j\}_{j=1}^\infty$
of test functions approaching the delta function of a point $p_i$ on
$X$).
\end{example}

The hypothesis of Theorem~\ref{thm:DeformationCYMeq4} hold in much
more generality. By the Donaldson--Uhlenbeck--Yau
Theorem~\cite{D5,UY}, which admits a generalization to principal
bundles (see~\cite{ABi,RS}), a family of examples generalizing
Example~\ref{ex:ASD-Donaldson} is provided by polystable holomorphic
principal bundles over cscK manifolds with no non-zero Hamiltonian Killing vector fields. Recall that this
theorem states that if a holomorphic principal $G^c$-bundle $(E^c,I)$
is (Mumford--Takemoto) polystable with respect to a K\"ahler class
$\Omega$ on a compact complex manifold $X$, then for any K\"ahler form
$\omega \in \Omega$ there exists a reduction $H$ of $(E^c,I)$ to $G$
which is HYM with respect to $\omega$.

Let $(X,L)$ be a compact polarised manifold whose first Chern class
$c_1(X)$ satisfies
\[
c_1(X) = \lambda c_1(L)
\]
for some $\lambda \in \ZZ$. When $\lambda < 0$ (e.g. if $X$ is a high
degree hypersurface of $\mathbb{P}^m$), $X$ has finite group of
automorphisms and by the above result of Aubin and Yau, there exists a
unique K\"ahler--Einstein metric $\omega \in c_1(L)$. If $\lambda =
0$, then by Yau's a solution to Calabi's Conjecture (see
e.g.~\cite[Theorem~11.7]{Be}), there exists a unique Ricci flat metric
on $c_1(L)$. As the dimension of the group of automorphisms of such
manifolds is equal to its first Betti number (see
\cite[Remark~11.22]{Be}), the simply connected ones (e.g. K3 surfaces)
are complex Ricci flat manifolds with finite group of
automorphisms. If $\lambda > 0$, it has been recently proved \cite{ChDoSun,T5} that $c_1(L)$
admits a K\"ahler--Einstein metric if and only if $(X,L)$ is K-stable. Let us restrict to the
case
\[
X = \mathbb{P}^2 \; \sharp \; m \overline{\mathbb{P}}^2,
\]
the complex surface obtained by blowing up $\mathbb{P}^2$ at $m$
generic points (see \cite{TY}). If we take $m$ such that $3 < m < 8$
then $c_1(X) > 0$, $X$ has finite automorphism group
(see~\cite[Remark~3.12]{T4}) and it was proved in~\cite{TY} that $X$
admits a K\"ahler--Einstein metric.

On the other hand, given a polarised projective manifold $(X,L)$
(without any assumption on $c_1(X)$), an asymptotic result of Maruyama
\cite{Ma} states that there exist $c_1(L)$-stable vector bundles $E$
over $X$ of rank $r$, provided that $r > \dim X > 2$ and
\begin{equation}
\label{eq:stablebundlesufcondition}
c_2(E) \cdot c_1(L)^{n-2} \gg 0.
\end{equation}
If $X$ has finite group of automorphisms and it is endowed with a
K\"ahler--Einstein metric $\omega \in c_1(L)$ as before, then we can
apply Theorem~\ref{thm:DeformationCYMeq4}.

\begin{example}
Let $(X,\omega)$ be a K\"ahler--Einstein manifold. Then $\omega$ is
a cscK metric, which determines a Hermitian--Yang--Mills metric $H$ on
the tangent bundle $E^c=TX$. The pair $(\omega,H)$ is a solution
to~\eqref{eq:CYMeq2} with $\alpha_1 = 0$, but it is not a solution
with $\alpha_1 \neq 0$ unless the Chern connection of $H$ is flat. If
$c_1(X)\leq 0$, then there are no non-zero Hamiltonian holomorphic
vector fields over $X$, so $\cF_{0,\Omega}=\cF_{\infty,\Omega}=0$ and
as in Theorem~\ref{thm:DeformationCYMeq4}, $(0,\Omega)$ has an open
neighbourhood $U\subset \RR \times H^{1,1}(X,\RR)$ such that for all
$(\widetilde{\alpha},\widetilde{\Omega}) \in U$, there exists a
solution $(\widetilde{\omega},\widetilde{H})$ to the coupled
equations~\eqref{eq:CYMeq2} with coupling constants satisfying
$\alpha_1/\alpha_0=\widetilde{\alpha}$ and $[\widetilde{\omega}] =
\widetilde{\Omega}$.
\end{example}

We will now construct an example where the $\alpha$-Futaki character
$\cF_I$ obstructs the existence of solutions to the coupled equations
for small ratio of the coupling constants.

\begin{example}
\label{ex:obstruction}
Let $(X,\omega)$ be a K\"ahler manifold such that $\omega$ is not a
cscK metric but it is extremal (e.g. $\CC\mathbb{P}^2$ blown up at one
point~\cite{Ca}). Recall from~\secref{sec:Defextremalholomorphic} that
the extremality condition is equivalent to the condition that
$S_\omega$ is the Hamiltonian function of a real holomorphic Killing
vector field $\eta$. Since $\omega$ is not a cscK metric, it follows
from~\eqref{eq:alphafutakismooth} and~\eqref{eq:Futaki-0,infty} that
the classical Futaki character of the K\"ahler class $\Omega=[\omega]$
evaluated at $\eta$ is
\[
  \langle\cF_{0,\Omega},\eta\rangle = \int_X(S_\omega-\hat{S})^2\omega^{[n]} > 0.
\]
Note that $\eta$ lifts to a holomorphic vector field $\zeta \in \Lie
\Aut (TX)$ on the holomorphic tangent bundle $E^c=TX$ of $X$. It
follows from~\eqref{eq:alphafutakismooth} that the $\alpha$-Futaki
character $\cF_I$ evaluated at $\zeta$ is positive for sufficiently
small values of $\alpha_1/\alpha_0>0$. Hence the pair $(X,TX)$ does
not admit a solution $(\omega,H)$ to~\eqref{eq:CYMeq2} with
$\omega\in\Omega$ and these values of the coupling constants.

Given an arbitrary holomorphic principal $G^c$-bundle $E^c$ over $X$,
the obstruction to lift a holomorphic vector field on $X$ to a
$G^c$-invariant holomorphic vector field on $E^c$ lies in $H^1(X,\ad
E^c)$ (cf.~\eqref{eq:infinit-action-connections}). Note that when $G^c
= \CC^*$, the previous argument always applies.
\end{example}

Let $E^c$ be a stable holomorphic principal $G^c$-bundle over a polarised manifold $(X,L)$. In this situation, the Donaldson-Uhlembeck-Yau Theorem allows us to think of the coupled equations as a generalization of the constant scalar curvature equation for a K\"ahler metric $\omega \in c_1(L)$. More precisely, given such $\omega$ there exists a unique HYM reduction $H$ on $E^c$ with respect to $\omega$ and therefore \eqref{eq:CYMeq2} can be interpreted as a single scalar equation for the K\"ahler metric. Although this approach may not be very useful in general, it becomes very explicit for the case of a line bundle $E^c$. In this case, a solution of the coupled equations is equivalent to a pair $(\omega,\beta)$, where $\beta$ is a harmonic $(1,1)$-form with $[\beta]/\sqrt 2 = 2\pi c_1(E^c)$ and satisfying
\begin{equation}\label{eq:CYMeqline}
S_\omega - \alpha|\beta|^2_\omega = c'
\end{equation}
for a real constant $c' \in \RR$, where $\alpha =
\alpha_1/\alpha_0$. As for this, we simply note that $H$ is a HYM
Hermitian metric on $E^c$ with respect to $\omega$ if and only if
$iF_H$ is harmonic. Therefore, for line bundles, the coupled equations
provide a deformation of the constant scalar curvature equation by a
harmonic $(1,1)$-form (cf. \cite{Stoppa}). This point of view has been
recently used by Keller and T{\o}nnesen-Friedman to find solutions of
the coupled equations on polarised complex 3-folds that do not admit
any cscK metric \cite{KellerTonnesen}. We should stress that in
general equation \eqref{eq:CYMeqline} is as difficult as the cscK
equation, which has been completely solved only in the
K\"ahler--Einstein case \cite{ChDoSun,T5}. It would be interesting to
study these deformations in terms of K-stability.

\subsection{CscK metrics on ruled manifolds}
\label{sec:example4}

We now briefly discuss the relation between
equation~\eqref{eq:ext-HYM}, given by the limit
\[
\alpha_0 \to 0
\]
in~\eqref{eq:CYMeq00}, and the existence of solutions to the cscK
equation on ruled manifolds. We will use existence results of
Y. J. Hong~\cite{Ho1,Ho2}.

Let $(X,J,\omega)$ be a compact K\"ahler manifold with constant scalar
curvature and $E$ a holomorphic stable vector bundle of degree zero
over $X$ (examples of this type were already provided
in~\secref{sec:example3}). Let $H$ be a Hermitian metric on $E$ whose
Chern connection $A$ is HYM (it exists by the
Donaldson--Uhlenbeck--Yau Theorem~\cite{D5, UY}). Let $L$ be the
tautological bundle over the projectivised bundle $\mathbb{P}(E)$ of
$E$ and $F_{\textrm{A}_{L^{\ast}}}$ the curvature of the connection
induced by $A$ on $L^{\ast}$. Then the $2$-form
\[
\frac{i}{2\pi} F_{\textrm{A}_{L^{\ast}}}
\]
is non-degenerate on the fibres and in fact it induces the
Fubini--Study metric, so
\[
\widehat{\omega}_k
= \frac{i}{2\pi} F_{\textrm{A}_{L^{\ast}}} + k \pi^{\ast}\omega
\]
is a K\"ahler metric on $\mathbb{P}(E)$ for $k$ large enough. When the
automorphism group of $(X,J)$ is finite, Y.J. Hong~\cite{Ho1} used a
deformation argument to prove that the cohomology class
$[\widehat{\omega}_k]$ contains a cscK metric for $k \gg 0$. Let $\cX$
be the extended gauge group of the frame $\PU(r)$-bundle of the
Hermitian vector bundle $(E,H)$ and $\cX_I\subset\cX$ the stabilizer
of the connection $A$. The assumption on $\Aut X$ was removed
in~\cite{Ho2} (see~\cite[Definition~I.A]{Ho2}), under the additional
conditions that the subgroup
\[
\cX_I \subset \Aut \mathbb{P}(E)
\]
is finite and
\begin{equation}
\label{eq:Hongcondition}
\Lambda^2_{\omega} (\tr F_A \wedge \tr F_A + \tr F_A \wedge
\rho_{\omega} + F_A \wedge F_A) = \text{const.}.
\end{equation}
Since $c_1(E) = 0$, this second condition reduces to
\[
\Lambda^2_{\omega}\tr(F_A \wedge F_A) = - \frac{4\hat{c}}{(n-1)!} \in \RR.
\]
The condition \eqref{eq:Hongcondition} appears when one splits the
linearization of the cscK equation on $\mathbb{P}(E)$ into vertical and
horizontal parts with respect to the connection $A$.

Hence we conclude that when $c_1(E) = 0$ and $\cG_I$ is finite, the
existence of a solution to~\eqref{eq:ext-HYM} is a sufficient
condition for the existence of a cscK metric in the cohomology class
$[\widehat{\omega}_k]$ for $k\gg 0$
(see~\cite[Theorem~III.A]{Ho2}). It would be interesting to study
further this relation, trying to prove that the existence of solutions
to the coupled equations for small $\frac{\alpha_1}{\alpha_0} > 0$
implies the existence of constant scalar curvature K\"ahler metrics on
$\mathbb{P}(E)$ with K\"ahler class $kc_1(L)$ for large $k$. This
would provide a generalization of Hong's results in \cite{Ho2}.



\begin{thebibliography}{12}
\frenchspacing\smallbreak

\bibitem{ACMM}
        M. C. Abbati, R. Cirelli, A. Mani\`a and P. W. Michor,
        \emph{The Lie group of automorphisms of a principal bundle},
        J. Geom. Phys. \textbf{6} (1989) 215--235.

\bibitem{AGGvortices}
        L. \'Alvarez-C\'onsul, M. Garc\'{\i}a-Fern\'andez and O. Garc\'{\i}a-Prada,
        \emph{Gravitating vortices}, (in preparation).

\bibitem{ABi}
        B. Anchouche and I. Biswas,
        \emph{Einstein--Hermitian connections on polystable principal bundles over a compact K\"{a}hler manifold},
        Amer. J. Math. \textbf{123} (2001) 207--228.


\bibitem{AB}
        M. F. Atiyah and R. Bott,
        \emph{The Yang--Mills equations over Riemann
        surfaces},
        Philos. Trans. Roy. Soc. London \textbf{A 308} (1983)
        523--615.

\bibitem{Au}
        T. Aubin,
        \emph{Some Nonlinear Problems in Riemannian Geometry}, Springer, 1998.


\bibitem{Be}
        A. L. Besse,
        \emph{Einstein Manifolds},
        Springer, 1987.



\bibitem{Bo}
        J. P. Bourguignon,
        \emph{Invariants int\'{e}graux functionnels pour des \'{e}quations
        aux d\'{e}riv\'{e}es partielles d'origine g\'{e}om\'{e}trique},
        Partial differential equations \textbf{1}, \textbf{2} (1990) 65--73.

\bibitem{Ca}
        E. Calabi,
        \emph{Extremal Kahler metrics},
        in `Seminar on differential geometry' (S.-T. Yau ed.), Annals
        of Math. Studies \textbf{102}, Princeton Univ. Press, 1982,
        259--290.

\bibitem{Capo}
       L.~Caporaso,
       \emph{A compactification of the universal Picard variety over the moduli space of stable curves},
       J. Amer. Math. Soc. \textbf{7} (1994) 589--660.


\bibitem{Ch1}
        X. X. Chen,
        \emph{The space of K\"ahler metrics},
        J. Differential Geom. \textbf{56} (2000) 189--234.

\bibitem{Ch2}
        \bysame,
        \emph{Space of K\"ahler metrics. III. On the lower bound of the Calabi energy and geodesic distance},
        Invent. Math. \textbf{175} (2009) 453--503.

\bibitem{ChDoSun}
        X. X. Chen, S. K. Donaldson and S. Sun,
        \emph{K\"ahler--Einstein metrics and stability}, \texttt{arXiv:1210.7494} (2012).

\bibitem{ChT}
        X. X. Chen and G. Tian,
        \emph{Geometry of K\"ahler metrics and foliations by holomorphic discs},
        Publ. Math. Inst. Hautes \'Etudes Sci. \textbf{107} (2008) 1--107.

\bibitem{D2}
        S. K. Donaldson,
        \emph{A new proof of a theorem of Narasimhan and Seshadri},
        J. Differential Geom. \textbf{18} (1983) 269--277.

\bibitem{D3}
        \bysame,
        \emph{Anti-self-dual Yang--Mills connections on a complex algebraic
        surface and stable vector bundles},
        Proc. London Math. Soc. \textbf{50} (1985) 1--26.

\bibitem{D5}
        \bysame,
        \emph{Infinite determinants, stable bundles and curvature},
        Duke Math. J. \textbf{54} (1987) 231--247.

\bibitem{D1}
        \bysame,
        \emph{Remarks on gauge theory, complex geometry and $4$-manifold topology}, in `Fields Medallists' lectures' (Atiyah, Iagolnitzer eds.), World
        Scientific, 1997, 384--403.

\bibitem{D6}
        \bysame,
        \emph{Symmetric spaces, K\"ahler geometry and Hamiltonian Dynamics},
        in `Northern California Symplectic Geometry Seminar' (Y. Eliashberg et al. eds.),
        Amer. Math. Soc., 1999, 13--33.

\bibitem{D7}
        \bysame,
        \emph{Scalar curvature and projective embeddings, I},
        J. Differential Geom. \textbf{59} (2001) 479--522.

\bibitem{DK}
        S. K. Donaldson and P. B. Kronheimer,
        \emph{The geometry of four-manifolds}, Oxford University Press, 1990.

\bibitem{Fj}
        A. Fujiki,
        \emph{Moduli space of polarized algebraic manifolds and
        K\"ahler metrics}, Sugaku
        Expo. \textbf{5} (1992) 173--191.

\bibitem{Ft0}
        A. Futaki,
        \emph{An obstruction to the existence of Einstein K\"ahler metrics},
        Invent. Math. \textbf{73} (1983) 437--443.

\bibitem{Ft1}
        \bysame,
        \emph{Asymptotic Chow Semi-stability and integral invariants},
        Internat J. Math. \textbf{15} (2004) 967--979.

\bibitem{FO}
        A. Futaki and H. Ono,
        \emph{Einstein metrics and GIT stability}, Sugaku {\bf 60} (2008) 175--202.

\bibitem{GF1}
        M. Garc\'{i}a-Fern\'{a}ndez,
        \emph{Coupled equations for K\"ahler metrics and Yang--Mills connections}.
        PhD Thesis. Instituto de Ciencias Matem\'aticas
        (CSIC-UAM-UC3M-UCM), Madrid, 2009, \texttt{arXiv:1102.0985 [math.DG]}.

\bibitem{GFT}
        M. Garc\'{i}a-Fern\'{a}ndez and C. Tipler,
        \emph{Deformation of complex structures and the Coupled K\"ahler--Yang--Mills equations}, \texttt{arXiv:1301.4480 [math.DG]} (2013).

\bibitem{GiesMorr}
       D.~Gieseker and I.~Morrison,
       \emph{Hilbert stability of rank-two bundles on curves},
       J. Differential Geom. \textbf{19} (1984) 1--29.

\bibitem{Gr}
        M. Gromov,
        \emph{Pseudoholomorphic curves in symplectic manifolds},
        Invent. Math. \textbf{82} (1985) 307--347.

\bibitem{Ho1}
        Y.-J. Hong,
        \emph{Constant Hermitian scalar curvature equations on ruled manifolds},
        J. Differential Geom. \textbf{53} (1999) 465--516.

\bibitem{Ho2}
        \bysame,
        \emph{Stability and existence of critical Kaehler metrics on ruled manifolds},
        J. Math. Soc. Japan  \textbf{60} (2008) 265--290.

\bibitem{Hu}
        L. Huang, \emph{On joint moduli spaces}. Mat. Ann. \textbf{302} (2005) 61-79.

\bibitem{KellerTonnesen}
		J. Keller and C. T{\o}nnesen-Friedman,
		\emph{Non trivial examples of coupled equations for K\"ahler metrics and Yang-Mills connections}, Central European J. Math. (5) {\bf 10} (2012) 1673--1687.

\bibitem{KN}
        G. Kempf and L. Ness,
        \emph{The length of vectors in representation spaces},
        Lecture Notes in Mathematics \textbf{732}, Springer, 1982, 233--243.


\bibitem{Ko}
       S. Kobayashi,
       \emph{Differential Geometry of Complex Vector Bundles},
       Princeton University Press, 1987.

\bibitem{KNI}
        S. Kobayashi and K. Nomizu,
        \emph{Foundations of Differential Geometry}, Volume I,
        Interscience Publishers, New york, 1963.

\bibitem{KNII}
        \bysame,
        \emph{Foundations of Differential Geometry}, Volume II,
        Interscience Publishers, New york, 1969.

\bibitem{Lb}
        C. LeBrun,
        \emph{The Einstein-Maxwell Equations, Extremal K\"{a}hler
        Metrics, and Seiberg-Witten Theory}, The many facets of
        geometry, Oxford Univ. Press, Oxford (2010) 17--33. Eds.: O. Garc\'{\i}a-Prada, J.-P. Bourguignon, S. Salamon.


\bibitem{LS2}
        C. LeBrun and R. Simanca,
        \emph{On the K\"ahler Classes of Extremal K\"ahler Metrics},
        Geometry and global analysis (1993) 225--271, Tohoku Univ., Sendai.

\bibitem{LS1}
        \bysame,
        \emph{Extremal K\"ahler metrics and complex deformation theory},
        Geometric and Functional Analysis \textbf{4} (1994) 298--336.


\bibitem{LiYau}
       J.~Li and S.-T.~Yau,
       \emph{The existence of supersymmetric string theory with torsion},
       J. Diff. Geom. \textbf{70} (2005) 143--181.

\bibitem{Mab2}
        T. Mabuchi,
        \emph{K-energy maps integrating Futaki invariants},
        Tohoku Math. J. \textbf{38} (1986) 575--593.

\bibitem{Mab1}
        \bysame,
        \emph{Some symplectic geometry on compact K\"ahler manifolds (I)},
        Osaka J. Math. \textbf{24} (1987) 227--252.

\bibitem{MMOPR}
        J. E. Marsden, G. Misiolek, J.-P. Ortega, M. Perlmutter and T. S. Ratiu,
        \emph{Hamiltonian Reduction by Stages}, in `Lecture Notes in Mathematics', Springer, Berlin, 2007.

\bibitem{Ma}
        M. Maruyama,
        \emph{Moduli of stable sheaves II},
        J. Math. Kyoto Univ. \textbf{18} (1979) 557--614.

\bibitem{McS}
        D. McDuff and D. Salamon,
        \emph{Introduction to Symplectic Topology},
        Oxford University Press, New York, Second edition, 1998.

\bibitem{MR}
        I. Mundet i Riera,
        \emph{A Hitchin--Kobayashi correspondence for K\"ahler Fibrations},
        J. Reine Angew. Math. \textbf{528} (2000) 41--80.

\bibitem{Pd}
        R. Pandharipande,
        \emph{A compactification over $M_g$ of the universal moduli space of slope-semistable vector bundles},
        J. Amer. Mat. Soc. \textbf{9} (1996) 425--471.

\bibitem{ramanathan:1975}
        A.~Ramanathan, \emph{Stable principal bundles on a compact {R}iemann surface},
        Math. Ann. \textbf{213} (1975) 129--152.

\bibitem{RS}
        A.  Ramanathan and S. Subramanian,
        \emph{Einstein-Hermitian connections on principal bundles and stability},
        J. Reine Angew. Math. \textbf{390} (1988) 21--31.

\bibitem{ST}
        G. Schumacher and M. Toma,
        \emph{Moduli of K\"ahler manifolds equipped with Hermite--Einstein vector bundles},
        Rev. Roumaine Math. Pures Appl. \textbf{38} (1993) 703--719.

\bibitem{Se}
        S. Semmes,
        \emph{Complex Monge--Amp\`ere and symplectic manifolds},
        Amer. J. Math. \textbf{114} (1992) 495--550.

\bibitem{Si}
        I. M. Singer,
        \emph{The geometric interpretation of a special connection},
        Pacific J. Math. \textbf{9} (1959) 585--590.

\bibitem{Stoppa}
       J.~Stoppa,
       \emph{Twisted constant scalar curvature K\"ahler metrics and K\"ahler slope stability},
       J. Diff. Geom. (3) \textbf{83} (2009) 663--691.


\bibitem{Te}
        A. Teleman,
        Symplectic stability, analytic stability in non-algebraic complex geometry,
        Internat. J. Math. \textbf{15} (2004) 183--209.

\bibitem{T4}
        G. Tian,
        \emph{Canonical metrics in K\"ahler geometry. Notes taken by Meike Akveld}, Lectures in Mathematics, ETH, Z\"{u}rich, Birkh\"{a}user Verlag,
        Basel, 2000.

\bibitem{T5}
        \bysame,
        \emph{K-stability and K\"ahler--Einstein metrics}, \texttt{arXiv:1211.4669 [math.DG]} (2012).

\bibitem{TY}
        G. Tian and S.-T. Yau,
        \emph{K\"ahler--Einstein metrics on complex surfaces with $c_1 > 0$},
        Commun. Math. Phys. \textbf{112} (1987) 175--203.


\bibitem{UY}
        K. K. Uhlenbeck and S.-T. Yau,
        \emph{On the existence of Hermitian--Yang--Mills connections
        on stable bundles over compact K\"ahler manifolds}, Comm. Pure and Appl. Math. \textbf{39-S} (1986) 257--293; \textbf{42} (1989) 703--707.

\bibitem{W}
        X. Wang,
        \emph{Moment map, Futaki invariant and stability of projective manifolds},
        Communications in analysis and geometry \textbf{12} (2004) 1009--1038.

\bibitem{Yang}
        Y. Yang,
        \emph{Prescribing Topological Defects for the Coupled Einstein and Abelian Higgs Equations},
        Comm. Math. Phys. {\bf 170} (1995) 541--582.

\bibitem{Yau}
        S.-T. Yau,
        \emph{On the Ricci curvature of a compact K\"ahler manifold and the complex Monge--Amp\`ere equation. I},
        Comm. Pure Appl. Math. 31 (1978) 339--411.


\end{thebibliography}
\end{document}